\documentclass[12pt,psamsfonts]{amsart}
\usepackage{fullpage,amssymb,amsfonts,amsmath}
\usepackage[all,cmtip]{xy}
\usepackage{enumerate}
\usepackage{mathrsfs}
\usepackage{comment}
\usepackage{hyperref} 
\usepackage[capitalise]{cleveref}
\usepackage{floatrow}
\usepackage{todonotes}
\usepackage{url}
\usepackage{constants}

\let\oldtocsection=\tocsection

\let\oldtocsubsection=\tocsubsection

\let\oldtocsubsubsection=\tocsubsubsection

\renewcommand{\tocsection}[2]{\hspace{0em}\oldtocsection{#1}{#2}}
\renewcommand{\tocsubsection}[2]{\hspace{1em}\oldtocsubsection{#1}{#2}}
\renewcommand{\tocsubsubsection}[2]{\hspace{2em}\oldtocsubsubsection{#1}{#2}}
\newtheorem{thm}{Theorem}[section]
\newtheorem{cor}[thm]{Corollary}
\newtheorem{prop}[thm]{Proposition}
\newtheorem{lem}[thm]{Lemma}

\newtheorem{quest*}{Question}
\newtheorem{prob*}{Problem}

\theoremstyle{definition}

\theoremstyle{remark}
\newtheorem{rem*}{Remark}
\newtheorem{rems*}[thm]{Remarks}

\numberwithin{equation}{section}
\newcounter{notation}

\DeclareUrlCommand\DOI{}
\newcommand{\doi}[1]{\href{http://dx.doi.org/#1}{\DOI{doi:#1}}}



\crefname{figure}{Figure}{Figures}
\theoremstyle{plain}
\newtheorem*{thm*}{Theorem}
\crefname{thm}{Theorem}{Theorems}
\crefname{cor}{Corollary}{Corollarys}
\newtheorem*{cor*}{Corollary}
\crefname{cor*}{Corollary}{Corollarys}
\crefname{lem}{Lemma}{Lemmas}
\crefname{prop}{Proposition}{Propositions}
\crefname{conj}{Conjecture}{Conjectures}
\newtheorem*{conj*}{Conjecture}
\crefname{conj*}{Conjecture}{Conjectures}
\crefname{defn}{Definition}{Definitions}

\theoremstyle{remark}
\newtheorem*{remark}{Remark}
\newtheorem*{remarks}{Remarks}

\def\addsymbol #1: #2#3{$#1$ \> \parbox{5.4in}{#2 \dotfill \pageref{#3}}\\} 
\def\addsymbolEND #1: #2#3{$#1$ \> \parbox{5.4in}{#2 \dotfill \pageref{#3}}}

\newcommand{\ds}{\displaystyle}

\renewcommand{\bar}{\overline}

\newcommand{\cC}{\mathcal{C}}

\newcommand{\kD}{\mathfrak{D}}
\newcommand{\kd}{\mathfrak{d}}

\newcommand{\bF}{\mathbb{F}}
\newcommand{\kf}{\mathfrak{f}}

\newcommand{\cL}{\mathcal{L}}
\newcommand{\sL}{\mathscr{L}}

\renewcommand{\Im}{\mathrm{Im}}
\newcommand{\km}{\mathfrak{m}}

\renewcommand{\pmod}[1]{\, (\mathrm{mod} {\, #1})}

\newcommand{\N}{\mathrm{N}}
\newcommand{\cN}{\mathcal{N}}

\newcommand{\kp}{\mathfrak{p}}

\newcommand{\kP}{\mathfrak{P}}

\newcommand{\cO}{\mathcal{O}}
\newcommand{\ord}{\mathrm{ord}\,}

\newcommand{\kq}{\mathfrak{q}}
\newcommand{\Q}{\mathbb{Q}}

\newcommand{\cR}{\mathcal{R}}
\newcommand{\R}{\mathbb{R}}

\renewcommand{\Re}{\mathrm{Re}}
\def\Res{\mathop{\mathrm{Res}}}

\newcommand{\sumP}{\sideset{}{'}\sum}
\newcommand{\sumS}{\sideset{}{^\star}\sum}
\newcommand{\sumD}{\sideset{}{^\dagger}\sum}

\newcommand{\Z}{\mathbb{Z}}
\newcommand{\cZ}{\mathcal{Z}}

\newconstantfamily{abcon}{symbol=c}

\title{An explicit bound for the least prime ideal in the Chebotarev density theorem}

\author{Jesse Thorner}
\address{
Department of Mathematics and Computer Science, Emory University \\
400 Dowman Drive, W401, Atlanta, GA, 30322, United States}
\email{jesse.thorner@gmail.com}

\author{Asif Zaman}
\address{
Department of Mathematics, University of Toronto \\
Room 6290, 40 St. George St.,  Toronto, ON, M5S2E4, Canada}
\thanks{The second author was supported in part by an NSERC PGS-D scholarship.} 
\email{asif@math.toronto.edu}

\date{\today}

\begin{document}
\subjclass[2010]{11M41, 11R44, 14H52}
\keywords{Chebotarev density theorem, Linnik's theorem, log-free zero density estimate, elliptic curves, modular forms, quadratic forms}
\maketitle
\begin{abstract}
We prove an explicit version of Weiss' bound on the least norm of a prime ideal in the Chebotarev density theorem, which is a significant improvement on the work of Lagarias, Montgomery, and Odlyzko.  As an application, we prove the first explicit, nontrivial, and unconditional upper bound for the least prime represented by a positive-definite primitive binary quadratic form.  We also present applications to elliptic curves and congruences for the Fourier coefficients of holomorphic cuspidal modular forms.
\end{abstract}

\section{Introduction and Statement of Results}

In 1837, Dirichlet proved that if $a,q\in\mathbb{Z}$ and $(a,q)=1$, then there are infinitely many primes $p\equiv a\pmod q$.  In light of this result, it is natural to ask how big is the first such prime, say $P(a,q)$?  Assuming the Generalized Riemann Hypothesis (GRH) for Dirichlet $L$-functions, Lamzouri, Li, and Soundararajan \cite{LLS} proved that for all $q\geq4$,
\begin{equation}
\label{eqn:Linnik_GRH}
P(a,q)\leq(\varphi(q)\log q)^2,
\end{equation}
where $\varphi$ is Euler's totient function.  Nontrivial, unconditional upper bounds are significantly harder to prove.  The first such bound on $P(a,q)$ is due to Linnik \cite{Linnik}, who proved that for some absolute constant $\Cr{c1}>0$, 
\begin{equation}
\label{eqn:Linnik}
P(a,q)\ll q^{\Cl[abcon]{c1}}
\end{equation}
with an absolute implied constant.  Admissible values of $\Cr{c1}$ are now known explicitly.  Building on the work of Heath-Brown \cite{HBLinnik}, Xylouris \cite{Xylouris} proved that one may take $\Cr{c1}=5.2$ unconditionally.  (Xylouris improved this to $\Cr{c1}=5$ in his Ph.D. thesis.)  For a detailed history of the unconditional progress toward \eqref{eqn:Linnik_GRH}, see Section 1 of Heath-Brown \cite{HBLinnik} and the sources contained therein.

A broad generalization of \eqref{eqn:Linnik} lies in the context of the Chebotarev density theorem.  Let $L/F$ be a Galois extension of number fields with Galois group $G$.  To each prime ideal $\mathfrak{p}$ of $F$ which is unramified in $L$, there corresponds a certain conjugacy class of automorphisms in $G$ which are attached to the prime ideals of $L$ lying above $\mathfrak{p}$.  We denote this conjugacy class using the Artin symbol $[\frac{L/F}{\mathfrak{p}}]$.  For a conjugacy class $C\subset G$, let
\[
\pi_C(x,L/F):=\#\{\mathfrak{p}:\textup{$\mathfrak{p}$ is unramified in $L$, $[\tfrac{L/F}{\mathfrak{p}}]=C$, $\mathrm{N}_{F/\mathbb{Q}}~\mathfrak{p}\leq x$}\}.
\]
The Chebotarev density theorem asserts that
{\small
\[
\pi_C(x,L/F)\sim\frac{|C|}{|G|}\int_2^x\frac{dt}{\log t}.
\]}%
In analogy with \eqref{eqn:Linnik}, it is natural to bound the quantity
\begin{equation}
\label{eqn:least}
P(C,L/F):=\min\{ \mathrm{N}_{F/\mathbb{Q}}~\mathfrak{p}\colon \textup{$\mathfrak{p}$ unramified in $L$, $[\tfrac{L/F}{\mathfrak{p}}]=C$, $\mathrm{N}_{F/\mathbb{Q}}~\mathfrak{p}$ a rational prime} \}. 
\end{equation}
Under GRH for Hecke $L$-functions, Lagarias and Odlyzko proved a bound for $P(C,L/F)$ (see \cite{LO}); Bach and Sorenson \cite{BS} made this bound explicit, proving that
\begin{equation}
\label{eqn:CDT_GRH}
P(C,L/F)\leq (4\log D_L+2.5 [L:\mathbb{Q}]+5)^2
\end{equation}
where $D_L = |\mathrm{disc}(L/\mathbb{Q})|$ is the absolute value of the discriminant of $L/\mathbb{Q}$.  (This can be improved assuming Artin's conjecture; see V. K. Murty \cite[Equation 2]{VKM2}.)  We note that if $L=\mathbb{Q}(e^{2\pi i/q})$ for some integer $q\geq 1$ and $F=\mathbb{Q}$, then one recovers a bound of the same analytic quality as \eqref{eqn:Linnik_GRH}, though the implied constants are slightly larger.

The first nontrivial, unconditional bound on $P(C,L/F)$ is due to Lagarias, Montgomery, and Odlyzko \cite{LMO}; they proved that $P(C,L/F)\leq 2D_L^{\Cl[abcon]{c2}}$ for some absolute constant $\Cr{c2}>0$.  Recently, Zaman \cite{Zaman_2015c} made this explicit, proving that
\begin{equation}
\label{eqn:LMO}
P(C,L/F)\ll D_L^{40}.
\end{equation}
The bound \eqref{eqn:LMO}, up to quality of the exponent, is commensurate with the best known bounds when $L$ is a quadratic extension of $F=\mathbb{Q}$, which reduces to the problem of bounding the least quadratic nonresidue.  We observe, however, that if $q$ is prime, $L=\mathbb{Q}(e^{2\pi i/q})$, and $F=\mathbb{Q}$, then \eqref{eqn:LMO} states that $P(a,q)\ll q^{40(q-2)}$, which is significantly worse than \eqref{eqn:Linnik}.

Weiss significantly improved the results in \cite{LMO}.  
%
%
Let $A$ be any abelian subgroup of $G$ such that $A\cap C$ is nonempty, let $\widehat{A}$ be the character group of $A$, and let $K=L^A$ be the subfield of $L$ fixed by $A$.  Let the $K$-integral ideal $ \mathfrak{f}_{\chi}$ be the conductor of a character $\chi\in\widehat{A}$, and let
\begin{equation}
\label{eqn:Q_def}
\mathcal{Q}(L/K) =\max\{\mathrm{N}_{K/\mathbb{Q}}\mathfrak{f}_{\chi}:\chi\in\widehat{A}\}.
\end{equation}
Weiss \cite{Weiss} proved that for certain absolute constants $\Cl[abcon]{Weiss_1}>0$ and $\Cl[abcon]{Weiss_2}>0$,
\begin{equation}
\label{eqn:Weiss}
P(C,L/F)\leq 2[K:\mathbb{Q}]^{\Cr{Weiss_1}[K:\mathbb{Q}]}(D_K\mathcal{Q}(L/K))^{\Cr{Weiss_2}}.
\end{equation}
To see how this compares to \eqref{eqn:LMO}, we observe that if $A$ is a cyclic subgroup of $G$, then
\[
D_L^{1/|A|}\leq D_K\mathcal{Q}(L/K)\leq D_L^{1/\varphi(|A|)}.
\]
(See \cite[Lemma 4.2]{BS} for a proof of the upper bound; the lower bound holds for all $A$ and follows from the conductor-discriminant formula.)  Furthermore, if $F=\mathbb{Q}$ and $L=\mathbb{Q}(e^{2\pi i/q})$, then one may take $\widehat{A}$ to be the full group of Dirichlet characters modulo $q$, in which case $K=F=\mathbb{Q}$ and $\mathcal{Q}(L/K)=q.$  Thus Weiss proves a bound on $P(C,L/F)$ which provides a ``continuous transition'' from \eqref{eqn:Linnik} to \eqref{eqn:LMO}.  In particular, \eqref{eqn:Linnik} follows from \eqref{eqn:Weiss}.


In this paper, we prove the following bound on $P(C,L/F)$, which makes \eqref{eqn:Weiss} explicit.

\begin{thm}
\label{thm:main_theorem}
Let $L/F$ be a Galois extension of number fields with Galois group $G$, let $C\subset G$ be a conjugacy class, and let $P(C,L/F)$ be defined by \eqref{eqn:least}.  Let $A\subset G$ be an abelian subgroup such that $A\cap C$ is nonempty, $K=L^A$ be the fixed field of $A$, and $\mathcal{Q}=\mathcal{Q}(L/K)$ be defined by \eqref{eqn:Q_def}.  Then
\[
P(C,L/F)\ll D_K^{694} \mathcal{Q}^{521}+ D_K^{232} \mathcal{Q}^{367} [K:\mathbb{Q}]^{290[K:\mathbb{Q}]}
\]
where the implied constant is absolute.
\end{thm}
~
\begin{remarks}
~
\begin{itemize}
\item Theorem \ref{thm:main_theorem} immediately implies that $P(a,q)\ll q^{521}$.  For historical context, this is slightly better than Jutila's bound  \cite{Jutila2} on $P(a,q)$  established in 1970, which was over 25 years after Linnik's original theorem.
\item The bound we obtain on $P(C,L/F)$ follows immediately from the effective lower bound on $\pi_C(x,L/F)$ given by \eqref{eqn:lower_bound_pi_C}, which is of independent interest.
\item
Let $\mathrm{rd}_K=D_K^{1/[K:\Q]}$.  If $[K:\mathbb{Q}]\leq \mathrm{rd}_K^{1.59}$, then $P(C,L/F) \ll D_K^{694}\mathcal{Q}^{521}$.  Situations where $[K:\mathbb{Q}]> \mathrm{rd}_K^{1.59}$ are rare; the largest class of known examples involve infinite $p$-class tower extensions, which were first studied by Golod and {\v{S}}afarevi{\v{c}} \cite{GS}.
\item If $L/K$ is unramified, then $\mathcal{Q} = 1$ and $D_K = D_L^{1/|A|}$.  Thus $P(C,L/F) \ll D_L^{694/|A|} + D_L^{232/|A|} [K:\mathbb{Q}]^{290 [K:\mathbb{Q}]}$.  If $[K:\mathbb{Q}]\leq \mathrm{rd}_K^{1.59}$, this improves \eqref{eqn:LMO} when $|A| \geq 18$.
\end{itemize}
\end{remarks}



We now consider some specific applications of \cref{thm:main_theorem}, the first of which is a bound on the least prime represented by a positive-definite primitive binary quadratic form $Q(x,y)\in\Z[x,y]$ of discriminant $D$.  It follows from \eqref{eqn:Weiss} that the least such prime $p$ satisfies $p \ll |D|^{\Cl[abcon]{QF}}$ for some positive absolute constant $\Cr{QF}$; see Kowalski and Michel \cite{KM} for a similar observation.  Ditchen \cite{Ditchen} proved, on average over $D \not\equiv 0 \pmod{8}$, that $p \ll_{\epsilon} |D|^{20/3+\epsilon}$ and Zaman \cite{Zaman_2015b} showed $p \ll_{\epsilon} |D|^{9+\epsilon}$ in an exceptional case. However, a nontrivial unconditional explicit bound on the least prime represented by $Q$ for {\it all} such forms has not been calculated before now.  Such a bound follows immediately from \cref{thm:main_theorem}.

\begin{thm}
\label{thm:least_prime_QF}
Let $Q(x,y)\in\Z[x,y]$ be a positive-definite primitive binary quadratic form of discriminant $D$.  There exists a prime $p\nmid D$ represented by $Q(x,y)$ such that $p\ll |D|^{694}$ with an absolute implied constant.  In particular, if $n$ is a fixed positive integer, there exists a prime $p\nmid n$ represented by $x^2+ny^2$ such that $p\ll n^{694}$ with an absolute implied constant.
\end{thm}

We now consider applications to the study of the group of points on an elliptic curve over a finite field.  Let $E/\mathbb{Q}$ be an elliptic curve without complex multiplication (CM), and let $N_E$ be the conductor of $E$.  The order and group structure of $E(\mathbb{F}_p)$, the group of $\mathbb{F}_p$-rational points on $E$, frequently appears when doing arithmetic over $E$.  Thus we are interested in understanding the distribution of values and divisibility properties of $\#E(\mathbb{F}_p)$.

V. K. Murty \cite{VKM2} and Li \cite{Xli} proved unconditional and GRH-conditional bounds on the least prime that does not split completely in a number field.  This yields bounds on the least prime $p\nmid\ell N_E$ such that $\ell\nmid\#E(\mathbb{F}_p)$, where $\ell\geq11$ is prime.  As an application of \cref{thm:main_theorem}, we prove a complementary result on the least $p\nmid\ell N_E$ such that $\ell\mid\#E(\mathbb{F}_p)$.  To state the result, we define $\omega(N_E)=\#\{p:p\mid N_E\}$ and $\mathrm{rad}(N_E)=\prod_{p\mid N_E}p$.

\begin{thm}
\label{thm:EC_order}
Let $E/\mathbb{Q}$ be an non-CM elliptic curve of conductor $N_E$, and let $\ell\geq11$ be prime.  There exists a prime $p\nmid \ell N_E$ such that $p\ll  \ell^{(5209+1389\omega(N_E))\ell^2}\mathrm{rad}(N_E)^{1895\ell^2}$ and $\ell\mid\#E(\mathbb{F}_p)$.  The implied constant is absolute.
\end{thm}
\begin{remark}
The proof is easily adapted to allow for elliptic curves over other number fields; we omit further discussion for brevity.
\end{remark}

%

One of the first significant results in the study of the distribution of values of $\#E(\mathbb{F}_p)$ is due to Hasse, who proved that if $p\nmid N_E$, then $|p+1-\#E(\mathbb{F}_p)|<2\sqrt{p}$.  For a given prime $\ell$, the distribution of the primes $p$ such that $\#E(\bF_p)\equiv p+1\pmod \ell$ can also be studied using the mod $\ell$ Galois representations associated to $E$.

\begin{thm}
\label{thm:EC_Fourier}
Let $E/\mathbb{Q}$ be a non-CM elliptic curve of squarefree conductor $N_E$, and let $\ell\geq11$ be prime.  There exists a prime $p\nmid\ell N_E$ such that $\#E(\mathbb{F}_p)\equiv p+1\pmod\ell$ and $p\ll\ell^{(4515+695\omega(N_E))\ell}\mathrm{rad}(N_E)^{1736\ell+1042}$.  The  implied constant is absolute.
\end{thm}

%

\cref{thm:EC_Fourier} will immediately follow from a more general result on congruences for the Fourier coefficients of certain holomorphic cuspidal modular forms.  Let
{\small
\[
f(z) = \sum_{n=1}^{\infty} a_f(n) e^{2\pi i n z}
\]}%
be a cusp form of integral weight $k_f\geq2$, level $N_f\geq1$, and nebentypus $\chi_f$.  Suppose further that $f$ is a normalized eigenform for the Hecke operators.  We call such a cusp form $f$ a newform; for each newform $f$, the map $n\mapsto a_f(n)$ is multiplicative.  Suppose that $a_f(n)\in\Z$ for all $n\geq1$.  In this case, $\chi_f$ is trivial when $f$ does not have CM, and $\chi_f$ is a nontrivial real character when $f$ does have CM.  Furthermore, when $k_f=2$, $f$ is the newform associated to an isogeny class of elliptic curves $E/\mathbb{Q}$.  In this case, $N_f=N_E$, and for any prime $p\nmid N_E$, we have that $a_f(p)=p+1-\#E(\mathbb{F}_p)$.

\begin{thm}
\label{thm:MF_Fourier} 
Let $f(z) = \sum_{n=1}^{\infty} a_f(n) e^{2\pi i n z}\in\mathbb{Z}[[e^{2\pi i z}]]$ be a non-CM newform of even integral weight $k_f\geq 2$, level $N_f$, and trivial nebentypus.  Let $\ell \geq 3$ be a prime such that \eqref{eqn:GL} holds and $\gcd(k_f-1,\ell-1)=1$.  For any progression $a\pmod\ell$, there exists a prime $p\nmid \ell N_f$ such that $a_f(p)\equiv a\pmod{\ell}$ and $p\ll \ell^{(4515+695\omega(N_f))\ell}\mathrm{rad}(N_f)^{1736\ell+1042}$.
\end{thm}
~
\begin{remarks}
~
\begin{itemize}
\item Equation \eqref{eqn:GL} is a fairly mild condition regarding whether the modulo $\ell$ reduction of a certain representation is surjective.  See \cref{sec:proof_MF_Fourier} for further details.
\item The proofs of Theorems \ref{thm:EC_order}$-$\ref{thm:MF_Fourier} are easily adapted to allow composite moduli $\ell$ as well as elliptic curves and modular forms with CM.  Moreover, the proofs can be easily modified to study the mod $\ell$ distribution of the trace of Frobenius for elliptic curves over number fields other than $\Q$.  We omit further discussion for brevity.
\item Using \eqref{eqn:LMO}, \cref{thm:MF_Fourier} follows with $p\ll \ell^{200\ell^4}\mathrm{rad}(N_f)^{40\ell^4}$.  Thus \cref{thm:MF_Fourier} is an improvement for large $\ell$.
\item If $r_{24}(n)$ is the number of representations of $n$ as a sum of 24 squares, then $691r_{24}(p)=16(p^{11}+1)+33152\tau(p)$, where Ramanujan's function $\tau(n)$ is the $n$-th Fourier coefficient of $\Delta(z)$, the unique non-CM newform of weight 12 and level 1.  If $\ell\neq691$ is such that \eqref{eqn:GL} holds for $f(z)=\Delta(z)$, then by \cref{thm:MF_Fourier}, there exists $p\neq\ell$ such that $691r_{24}(p)\equiv 16(p^{11}+1)\pmod\ell$ and $p\ll \ell^{4515\ell}$.
\end{itemize}
\end{remarks}

\subsection*{Acknowlegements}
The authors thank John Friedlander, V. Kumar Murty, Robert Lemke Oliver, Ken Ono, and David Zureick-Brown for their comments and suggestions.  The first author conducted work on this paper while visiting Centre de Recherches Math\'ematiques (hosted by Chantal David, Andrew Granville, and Dimitris Koukoulopoulos) and Stanford University (hosted by Robert Lemke Oliver and Kannan Soundararajan); he is grateful to these departments and hosts for providing a vibrant work environment.

\section{Notation and Auxiliary Estimates}
\label{HeckeLFunctions} 

\subsection{Notation} 
\label{subsec:notation}
We will use the following notation throughout the paper:
\begin{itemize}
	\item $K$ is a number field.
	\item $\mathcal{O}_K$ is the ring of integers of $K$.
	\item $n_K = [K:\mathbb{Q}]$ is the degree of $K/\mathbb{Q}$.
	\item $D_K$ is the absolute value of the discriminant of $K$.
	\item $\N = \N_{K/\mathbb{Q}}$ is the absolute field norm of $K$. 
	\item $\zeta_K(s)$ is the Dedekind zeta function of $K$.
	\item $\mathfrak{q}$ is an integral ideal of $K$.
	\item $\mathrm{Cl}(\mathfrak{q}) = I(\mathfrak{q})/P_{\mathfrak{q}}$ is the narrow ray class group of $K$ modulo $\mathfrak{q}$.
	\item $\chi$, or $\chi \pmod{\kq}$, is a character of $\mathrm{Cl}(\kq)$, referred to as a Hecke character or ray class character of $K$.
	\item $\delta(\chi)$ is the indicator function of the trivial character.
	\item $\kf_{\chi}$ is the conductor of $\chi$; that is, it is the maximal integral ideal such that $\chi$ is induced from a primitive character $\chi^* \pmod{\kf_{\chi}}$.
	\item $D_{\chi} = D_K \N\kf_{\chi}$.
	\item $L(s,\chi)$ is the Hecke $L$-function associated to $\chi$.
	\item $H$, or $H \pmod{\kq}$, is a subgroup of $\mathrm{Cl}(\mathfrak{q})$, or equivalently of $I(\mathfrak{q})$ containing $P_{\mathfrak{q}}$. The group $H$ is referred to as a congruence class group of $K$. 
	\item $\chi \pmod{H}$ is a character $\chi \pmod{\kq}$ satisfying $\chi(H) = 1$.
 	\item $Q = Q_H = \max\{ \N\kf_{\chi}\colon \chi \pmod{H} \}$ is the maximum conductor of $H$. 
	\item $\kf_H = \mathrm{lcm}\{ \kf_{\chi}\colon \chi \pmod{H} \}$ is the conductor of $H$.
	\item $H^* \pmod{\kf_H}$ is the primitive congruence class group inducing $H$.
	\item $h_H = [I(\kq):H]$.   
\end{itemize}
We also adhere to the convention that all implied constants in all asymptotic inequalities $f\ll g$ or $f=O(g)$ are absolute with respect $H$ and $K$.  If an implied constant depends on a parameter, such as $\epsilon$, then we use $\ll_{\epsilon}$ and $O_{\epsilon}$ to denote that the implied constant depends at most on $\epsilon$.  All implied constants will be effectively computable.


\subsection{Hecke $L$-functions}
For a more detailed reference on Hecke $L$-functions, see \cite{LMO} and the sources contained therein. Strictly speaking, a Hecke character $\chi$ is a function on $\mathrm{Cl}(\kq)$ but, by pulling back the domain of $\chi$ and extending it by zero, we regard $\chi$ as a function on integral ideals of $K$.  We will use this convention throughout the paper. 

The \emph{Hecke $L$-function of $\chi$}, denoted $L(s, \chi)$, is defined as 
{\small
\begin{equation}
L(s,\chi) = \sum_{\mathfrak{n}} \chi(\mathfrak{n}) \N\mathfrak{n}^{-s} = \prod_{\mathfrak{p}} \Big(1-\frac{\chi(\mathfrak{p})}{\N\mathfrak{p}^{s} } \Big)^{-1}
\label{def:Hecke_L-fcn}
\end{equation}}%
for $\Re\{s\} > 1$ where the sum is over integral ideals $\mathfrak{n}$ of $K$ and the product is over prime ideals $\mathfrak{p}$ of $K$. Recall that the Dedekind zeta function $\zeta_K(s)$ is the primitive Hecke $L$-function associated to the trivial character $\chi_0$; that is,
{\small
\begin{equation}
\zeta_K(s) = \sum_{\mathfrak{n} \subseteq\cO_K} (\N\mathfrak{n})^{-s} = \prod_{\mathfrak{p}}\Big(1-\frac{1}{\N\mathfrak{p}^s} \Big)^{-1}
\label{def:DedekindZeta}
\end{equation}}%
for $\Re\{s\} > 1$. Returning to $L(s,\chi)$, assume that $\chi$ is primitive for the remainder of this subsection, unless otherwise specified. Define the \emph{completed Hecke $L$-function} $\xi(s, \chi)$ by
\begin{equation}
\xi(s, \chi) = \big[ s(s-1) \big]^{\delta(\chi)} D_{\chi}^{s/2} \gamma_{\chi}(s) L(s, \chi),
\label{def:CompletedHecke}
\end{equation}
where $D_{\chi} = D_K \N\kf_{\chi}$, $\delta(\chi)$ is the indicator function of the trivial character, and $\gamma_{\chi}(s)$ is the \emph{gamma factor of $\chi$} defined by
{\small
\begin{equation}
\gamma_{\chi}(s) =  \Big[ \pi^{-\tfrac{s}{2}} \Gamma\Big(\frac{s}{2}\Big)  \Big]^{a(\chi)} \cdot \Big[ \pi^{-\tfrac{s+1}{2} } \Gamma\Big( \frac{s+1}{2} \Big)   \Big]^{b(\chi)}.
\label{GammaFactor} 
\end{equation}}%
Here $a(\chi)$ and $b(\chi)$ are certain non-negative integers satisfying 
\begin{equation}
a(\chi) + b(\chi) = n_K. 
\label{GammaFactor_Exponents}
\end{equation}
It is a classical fact that $\xi(s, \chi)$ is entire of order 1 and satisfies the functional equation
\begin{equation}
\xi(s, \chi) = w(\chi) \xi(1-s, \bar{\chi})
\label{FunctionalEquation}
\end{equation}
where $w(\chi) \in \mathbb{C}$ is the \emph{root number} of $\chi$ satisfying $|w(\chi)| = 1$. 
The zeros of $\xi(s,\chi)$ are the \emph{non-trivial zeros $\rho$} of $L(s,\chi)$ and are known to satisfy $0 < \Re\{\rho\} < 1$.  The \emph{trivial zeros $\omega$} of $L(s, \chi)$ are given by 
{\small
\begin{equation}
\mathop{\ord}_{s = \omega} L(s, \chi) = 
\begin{cases}
a(\chi) - \delta(\chi) & \text{if } \omega = 0, \\
b(\chi) &  \text{if } \omega = -1,-3,-5,\dots\\
a(\chi) & \text{if } \omega = -2,-4,-6, \dots
\end{cases}
\label{TrivialZeros}
\end{equation}}%
and arise as poles of the gamma factor of $L(s,\chi)$. Since $\xi(s,\chi)$ is entire of order 1, it admits a Hadamard product factorization given by
{\small
\begin{equation}
\label{eqn:Hadamard}
\xi(s,\chi)=e^{A(\chi)+B(\chi)s}\prod_{\rho }\left(1-\frac{s}{\rho}\right)e^{s/\rho}.
\end{equation}}%

\begin{lem} \label{ExplicitFormula} Let $\chi$ be a primitive Hecke character. Then
{\small
\[
- \Re\Big\{\frac{L'}{L}(s,\chi)\Big\} = \frac{1}{2} \log D_{\chi} +\Re\Big\{\frac{\delta(\chi)}{s-1}+ \frac{\delta(\chi)}{s}\Big\} -  \sum_{\rho }\Re\Big\{ \frac{1}{s-\rho}\Big\}  + \Re\Big\{\frac{\gamma_{\chi}'}{\gamma_{\chi}}(s)\Big\}.
\]}%
where the sum is over all non-trivial zeros $\rho$ of $L(s,\chi)$. 
\end{lem}
\begin{proof} See \cite[Lemma 5.1]{LO} for example. 
\end{proof}
By similar arguments, there exists an explicit formula for higher derivatives of $-\frac{L'}{L}(s,\chi)$. 
\begin{lem} \label{ExplicitFormula_HigherDerivatives}
Let $\chi$ be a Hecke character (not necessarily primitive) and $k \geq 1$ be a positive integer. Then
{\small
\begin{align*}
(-1)^{k+1}\frac{d^k}{ds^k} \frac{L'}{L}(s,\chi)  = 
\sum_{\mathfrak{p}} \sum_{m=1}^{\infty} (\log \N\mathfrak{p})  \chi(\mathfrak{p}) \frac{(\log \N\mathfrak{p}^m)^k}{(\N\mathfrak{p}^m)^{s} }   = \frac{\delta(\chi)k!}{(s-1)^{k+1}}  -  \sum_{\omega} \frac{k!}{(s-\omega)^{k+1}}
\end{align*}}%
for $\Re\{s\} > 1$, where the first sum is over prime ideals $\mathfrak{p}$ of $K$ and the second sum is over all zeros $\omega$ of $L(s,\chi)$, including trivial ones, counted with multiplicity. 
\end{lem}
\begin{proof} By standard arguments, this follows from the Hadamard product \eqref{eqn:Hadamard} of $\xi(s,\chi)$ and the Euler product of $L(s,\chi)$. See \cite[Equations (5.2) and (5.3)]{LMO}, for example. 
%
\end{proof}

\subsection{Explicit $L$-function estimates}

In order to obtain explicit results, we must have explicit bounds on a few important quantities.  First, we record a bound for $L(s,\chi)$ in the critical strip $0 < \Re\{s\} < 1$ via a Phragmen-Lindel\"{o}f type convexity estimate due to Rademacher. 

\begin{lem}[Rademacher \cite{Rademacher}]
\label{Rademacher}
Let $\chi$ be a primitive Hecke character and $\eta \in (0, 1/2]$. Then for $s = \sigma+it$,
{\small
\[
|L(s,\chi)| \ll \Big| \frac{1+s}{1-s}\Big|^{\delta(\chi)} \zeta_{\mathbb{Q}}(1+\eta)^{n_K} \Big(  \frac{ D_{\chi} }{ (2\pi)^{n_K}} (3+|t|)^{n_K} \Big)^{ (1+\eta-\sigma)/2} 
\]}%
uniformly in the strip $-\eta \leq \sigma \leq 1+\eta$. 
\end{lem}

Next, we record an explicit bound on the digamma function and $\frac{\gamma_{\chi}'}{\gamma_\chi}(s)$. 
\begin{lem}
\label{digamma}
Let $s = \sigma+it$ with $\sigma > 1$ and $t \in \R$. Then $\Re\{ \frac{\Gamma'}{\Gamma}(s)\}\leq \log|s| + \sigma^{-1}$ and, for any Hecke character $\chi$, $\Re\{\frac{\gamma_{\chi}'}{\gamma_{\chi}}(s)\}\leq \frac{n_K}{2}(\log(|s|+1) + \sigma^{-1} -\log \pi)$.  In particular, for $1 < \sigma \leq 6.2$ and $|t| \leq 1$, $\Re\{\frac{\gamma_{\chi}'}{\gamma_{\chi}}(s)\} \leq 0$.
\end{lem}
\begin{proof} The first estimate follows from \cite[Lemma 4]{OS}. The second estimate is a straightforward consequence of the first combined with the definition of $\gamma_{\chi}(s)$ in \eqref{GammaFactor}. The third estimate is contained in \cite[Lemma 3]{Ahn-Kwon}.
\end{proof}

Next, we establish some bounds on the number of zeros of $L(s,\chi)$ in a circle. 

\begin{lem}
\label{ZerosInCircle-Classical}
Let $\chi$ be a Hecke character. Let $s=\sigma+it$ with $\sigma > 1$ and $t \in \R$. For $r > 0$, denote
\begin{equation}
N_{\chi}(r; s) :=\#\{\rho=\beta+i\gamma: 0 < \beta < 1, L(\rho, \chi)=0,|s-\rho|\leq r\}.
\label{def:CountZerosInCircle}
\end{equation}
If $0 < r \leq 1$, then
\[
N_{\chi}(r; s) \leq \{ 4\log D_K + 2 \log \N \kf_{\chi} +2n_K\log(|t|+3)  + 4 + 4\delta(\chi) \} \cdot r + 4 + 4\delta(\chi).
\]
\end{lem}
\begin{proof} Without loss, we may assume $\chi$ is primitive. 
	Observe $N_{\chi}(r; s)  \leq N_{\chi}(r; 1+it) \leq N_{\chi}(2r; 1+r+it)$ so it suffices to bound the latter quantity. 
Now, if $s_0=1+r+it$, notice
{\small
\[
N_{\chi}(2r; s_0)
	\leq 4r\sum_{|s_0 - \rho|\leq 2r}\Re\left\{ \frac{1}{s_0-\rho}\right\} 
	\leq 4r\sum_{\rho}\Re\left\{ \frac{1}{s_0-\rho}\right\}.
\]}%
Applying \cref{ExplicitFormula,digamma} twice and noting $\Re\left\{\frac{L'}{L}(s_0, \chi)\right\} \leq  -\frac{\zeta_K'}{\zeta_K}(1+r)$ via their respective Euler products, the above is
{\small
\begin{align*}
&\leq 4r \left(\Re\left\{\frac{L'}{L}(s_0, \chi)\right\}+\frac{1}{2}\log D_{\chi}+\Re\left\{\frac{\gamma_{\chi}'}{\gamma_{\chi}}(s_0)\right\}+\delta(\chi)\Re\left\{\frac{1}{s_0}+\frac{1}{s_0-1}\right\}\right)\\
&\leq \{ 4\log D_K + 2 \log \N \kf_{\chi} +2n_K\log(|t|+3) + 4 + 4\delta(\chi) \} \cdot r + 4 + 4\delta(\chi)
\end{align*}}%
as $D_{\chi} = D_K \N\kf_{\chi}$. For the details on estimating $-\frac{\zeta_K'}{\zeta_K}(1+r)$, see \cref{lem:PrimeSum}.  
\end{proof}

To improve the bound in \cref{ZerosInCircle-Classical}, we exhibit an explicit inequality involving the logarithmic derivative of $L(s,\chi)$ comparable with \cite[Theorem 2]{KadiriNg} for the Dedekind zeta function. 

\begin{prop} \label{ExplicitInequality-Convexity} Let  $0 < \epsilon < \tfrac{1}{4}, T  \geq 1$ and $s = \sigma+it$. For a primitive Hecke character $\chi$, define a multiset of non-trivial zeros of $L(s,\chi)$ by
\[
\cZ_{r, t} = \{ \rho = \beta+i\gamma \, \mid \, L(\rho, \chi) = 0,  |1+it - \rho| \leq r \}. 
\]
Then, for $0 < r < \epsilon$, 
{\small
\begin{equation}
\begin{aligned}
 - \Re\Big\{ \frac{L'}{L}(s,\chi) \Big\}  
 & \leq \big( \tfrac{1}{4} + \tfrac{\epsilon}{\pi} + 5\epsilon^{10} \big)\cL_{\chi} + (4 \epsilon^2 + 80 \epsilon^{10})\cL_{\chi}' + \delta(\chi)  \Re\Big\{ \frac{1}{s-1} \Big\}  \\
 & \qquad  - \sum_{\rho \in \cZ_{r,t} }  \Re\Big\{\frac{1}{s-\rho} \Big\}  + O_{\epsilon}(n_K)
\label{eqn:EI_1}
\end{aligned}
\end{equation}}%
and
{\small
\begin{equation}
 - \Re\Big\{ \frac{L'}{L}(s,\chi) \Big\}  \leq \big( \tfrac{1}{4} + \tfrac{\epsilon}{\pi} + 5\epsilon^{10} \big)\cL_{\chi}   + \delta(\chi)  \Re\Big\{ \frac{1}{s-1} \Big\} + O_{\epsilon}(n_K)
\label{eqn:EI_2}
\end{equation}}%
uniformly in the region $1 < \sigma \leq 1 + \epsilon$ and $|t| \leq T$, where $\cL_{\chi} = \log D_{\chi} + n_K \log(T+3)$ and $\cL_{\chi}' = \log D_K + \cL_{\chi}$. 
\end{prop}
\begin{proof} This result is a modified version of \cite[Lemma 4.3]{Zaman_2015a} which is motivated by \cite[Lemma 3.1]{HBLinnik}. The main improvements are the valid range of $\sigma$ and $t$. Consequently, we sketch the argument found in \cite{Zaman_2015a} highlighting the necessary modifications. Assume $\chi$ is non-trivial.  Apply \cite[Lemma 3.2]{HBLinnik} with $f(z) = L(z,\chi), a=s$ and $R = 1-\eta$ where $\eta = \eta_{s,\chi} \in (0,\tfrac{1}{10})$ is chosen sufficiently small so that $L(w,\chi)$ has no zeros on the circle $|w-s| = R$. Then
{\small
\begin{equation}
-\Re \Big\{ \frac{L'}{L}(s,\chi) \Big\} = -\sum_{|s-\rho| < R} \Re\Big\{ \frac{1}{s-\rho} - \frac{s-\rho}{R^2} \Big\} - J
\label{JensenUse}
\end{equation}}%
where
{\small
\[
J :=  \int_0^{2\pi} \frac{\cos \theta}{\pi R} \cdot \log| L(s+ R e^{i\theta},\chi)|d\theta.
\]}%
To bound $J$ from below, write
{\small
\[
J = \int_0^{\pi/2} + \int_{\pi/2}^{3\pi/2} + \int_{3\pi/2}^{2\pi} = J_1 + J_2 + J_3, 
\]}%
say, so we may consider each contribution separately. For $J_1$, notice by \cite[Lemma 2.5]{Zaman_2015a},
\[
\log| L(s+ R e^{i\theta},\chi)| \leq \log \zeta_K(\sigma+ R\cos \theta) \ll n_K \log\Big( \frac{1}{\sigma-1 + R\cos \theta}\Big).
\]
Write $[0,\tfrac{\pi}{2}] = [0,\tfrac{\pi}{2} - (\sigma-1)] \sqcup [\tfrac{\pi}{2}-(\sigma-1), \tfrac{\pi}{2}] = I_1 \sqcup I_2$, say. Then 
{\small
\[
J_1 = \int_{I_1} + \int_{I_2} \ll n_K \int_{I_1} \cos \theta \log(1/\cos \theta) d\theta + n_K \log(1/(\sigma-1)) \int_{I_2} \cos \theta d\theta \ll_{\epsilon} n_K.
\]}%
A similar argument holds for $J_3$ so $J_1 + J_3 \ll_{\epsilon} n_K$.  For $J_2$, consider $\theta \in [\pi/2, 3\pi/2]$. As $1 < \sigma \leq 1 + \epsilon$ and $R < 1$, $0 < \sigma + R\cos\theta \leq 1+ \epsilon$.  Hence, by \cref{Rademacher}, 
\begin{align*}
 \log|L(s+Re^{i\theta},\chi)|  \leq \tfrac{1}{2}  \cL_{\chi}( -R \cos\theta + \epsilon) +  O_{\epsilon}(n_K).
\end{align*}
Thus,
{\small
\[
J_2 \geq \frac{\cL_{\chi}}{2 \pi R} \int_{\pi/2}^{3\pi/2} (-R \cos^2 \theta + \epsilon \cos \theta) \, d\theta + O_{\epsilon}(n_K) 
\]}%
yielding overall
\begin{equation}
J \geq  - ( \tfrac{1}{4}+ \tfrac{\epsilon}{\pi R} )\cL_{\chi}  + O_{\epsilon}(n_K).
\label{Jintegral}
\end{equation}
For the sum over zeros in \eqref{JensenUse}, observe that the terms are non-negative so \eqref{eqn:EI_2} follows immediately from \eqref{JensenUse} and \eqref{Jintegral} after taking $\eta \rightarrow 0$ which implies $R \rightarrow 1$.  To prove \eqref{eqn:EI_1}, consider $0 < r  <  \epsilon$. By the same observation, we may restrict our sum over zeros from $|s-\rho| < R$ to a smaller circle within it: $|1+it-\rho| \leq r$. As $r < \epsilon < 1/4$ by assumption, we discard the zeros outside this smaller circle. For such zeros $\rho$ satisfying $|1+it-\rho| \leq r$, notice $\Re \{ s- \rho  \} = \sigma - \beta \leq \epsilon + r< 2\epsilon$.  This implies, by \cref{ZerosInCircle-Classical}, that 
{\small
\begin{equation}
\sum_{|1+it-\rho| \leq r} \Re\big\{ \frac{s-\rho}{R^2} \} \leq \frac{2\epsilon}{R^2} \cdot  \big\{ \big(2\cL_{\chi}' + 8\big) r + 8\big\} \leq \frac{4 \epsilon^2}{R^2} \cL_{\chi}' + O(1).
\label{SmallerCircle}
\end{equation}}%
Thus, \eqref{eqn:EI_1} immediately follows\footnote{One actually obtains \eqref{eqn:EI_1} without the extra $\epsilon^{10}$ terms.} upon combining \eqref{JensenUse}, \eqref{Jintegral}, and \eqref{SmallerCircle}, and taking $\eta \rightarrow 0$ which implies $R \rightarrow 1$. This completes the proof for $\chi$ non-trivial. 

	For $\chi = \chi_0$ trivial,  similarly proceed with \cite[Lemma 3.2]{HBLinnik} with $f(z) = (\frac{z-1}{z+1}) \zeta_K(z)$ and $a = z$, but the choice of $R$ is different due to the simple pole of the Dedekind zeta function. Observe that the circles $|w-1| = \epsilon^{10}$ and $|w-s| = R$ are disjoint for at least one of the following: (i) all $R \in (1-\epsilon^{10}, 1)$  or  (ii) all $R \in (1-5\epsilon^{10}, 1-4\epsilon^{10})$.  In the case of (i), choose $R = 1-\eta$ for $\eta = \eta_{s,\chi}$ sufficiently small so that $L(w,\chi)$ has no zeros on the circle $|w-s| = R$. Similarly for (ii), take $R = 1-4\epsilon^{10} - \eta$. 
	
	Continuing with the same arguments, the only difference occurs when bounding $J_1$ and similarly $J_3$, in which case one must estimate 
	\[
	\int_0^{\pi/2} \frac{\cos \theta}{\pi R} \log\Big| \frac{s-1+Re^{i\theta}}{s+1+Re^{i\theta}} \Big| d\theta. 
	\]
	By our choice of $R$, the quantity in the logarithm is $\asymp_{\epsilon} 1$ and hence the above is $O_{\epsilon}(1)$. The remainder of the argument is the same, except at the final step one must take $R \rightarrow 1$ in case (i) and $R \rightarrow 1-4\epsilon^{10}$ in case (ii). The latter case yields the additional $\epsilon^{10}$ terms appearing in \eqref{eqn:EI_1}.
\end{proof}

\begin{lem} \label{ZerosInCircle-Convexity} Let $\chi$ be a Hecke character and $0 < r < \epsilon < 1/4$. If $s = \sigma+it$ with $1 < \sigma < 1+\epsilon$ and $N_{\chi}(r; s)$ by \eqref{def:CountZerosInCircle}, then, denoting $\phi = 1 +  \tfrac{4}{\pi}\epsilon + 16 \epsilon^2 + 340\epsilon^{10}$,
\[
N_{\chi}(r; s) \leq \phi  \big( 2 \log D_K  +  \log \N\kf_{\chi} +  n_K \log( |t|+ 3) + O_{\epsilon}(n_K) \big) \cdot r+ 4 + 4\delta(\chi). 
\]
\end{lem}
\begin{proof} Analogous to \cref{ZerosInCircle-Classical} except we bound $N_{\chi}(r; 1+it)$ instead of $N_{\chi}(2r; 1+r+it)$ and further, we apply \cref{ExplicitInequality-Convexity} in place of \cref{ExplicitFormula,digamma}.   
\end{proof}

\subsection{Arithmetic Sums} 

We estimate various sums over integral ideals of $K$ which requires some additional notation.  It is well-known that the Dedekind zeta function $\zeta_K(s)$, defined by \eqref{def:DedekindZeta}, has a simple pole at $s=1$. Thus, we may define
\begin{equation}
\kappa_K := \Res_{s=1} \zeta_K(s) \quad \text{ and } \quad \gamma_K := \kappa_K^{-1} \lim_{s \rightarrow 1} \Big( \zeta_K(s) - \frac{\kappa_K}{s-1} \Big)
\label{def:Zeta_Laurent}
\end{equation}
so the Laurent expansion of $\zeta_K(s)$ at $s=1$ is given by 
\[
\zeta_K(s) =  \frac{\kappa_K}{s-1} + \kappa_K \gamma_K + O_K( |s-1| ).
\]
We refer to $\gamma_K$ as the \emph{Euler-Kronecker constant of $K$}. (See Ihara \cite{Ihara} for details on $\gamma_K$.)
\begin{lem}
\label{lem:harmonic_sum}
For $x > 0$ and $0<\epsilon<1/2$,
{\small
\[
\Big|\sum_{\mathrm{N}\mathfrak{n} < x}\frac{1}{\mathrm{N}\mathfrak{n}}\Big(1-\frac{\mathrm{N}\mathfrak{n}}{x}\Big)^{n_K}-\kappa_K\Big(\log x-\sum_{j=1}^{n_K} \frac{1}{j}\Big)-\kappa_K\gamma_K\Big|\ll_{\epsilon} \big( n_K^{n_K} D_K \big)^{1/4+\epsilon} x^{-1/2}.
\]}%
\end{lem}
\begin{proof}
The quantity we wish to bound equals
{\small
\begin{align*}
\frac{1}{2\pi i}\int_{-\frac{1}{2}-i\infty}^{-\frac{1}{2}+i\infty}\zeta_K(s+1)\frac{x^s}{s}\frac{n_K!}{\prod_{j=1}^{n_K} (s+j)}ds=\frac{n_K!}{2\pi i}\int_{-\frac{1}{2}-i\infty}^{-\frac{1}{2}+i\infty}\zeta_K(s+1)\frac{\Gamma(s)}{\Gamma(n_K+1+s)}x^s ds.
\end{align*}}%
Using Lemma \ref{Rademacher}, Stirling's formula, and $\zeta_{\mathbb{Q}}(1+\epsilon)^{n_K} \ll e^{O_{\epsilon}(n_K)}$, the result follows.
\end{proof}

\begin{cor}
\label{cor:harmonic_sum}
Let $\epsilon > 0$ be arbitrary. If $x \geq 3 \big( n_K^{n_K} D_K)^{1/2+\epsilon}$ then
{\small
\[
\sum_{\mathrm{N}\mathfrak{n} < x} \frac{1}{\mathrm{N}\mathfrak{n}} \geq \{ 1 - \tfrac{1}{1+2\epsilon} + O_{\epsilon}( \tfrac{1}{\log x} ) \} \cdot \kappa_K \log x.
\]}%
\end{cor}
\begin{proof}
It suffices to assume that $\kappa_K \geq 1/\log x$. From \cref{lem:harmonic_sum}, it follows that
{\small
\begin{align*}
\frac{1}{\kappa_K} \sum_{\mathrm{N}\mathfrak{n} < x} \frac{1}{\mathrm{N}\mathfrak{n}}
  \geq  \log x - \sum_{j=1}^{n_K} \frac{1}{j} + \gamma_K + O_{\epsilon}\left( x^{-\tfrac{\epsilon}{8}} \log x \right),
\end{align*}}%
by our assumption on $x$. By \cite[Proposition 3]{Ihara},
{\small
\[
\gamma_K \geq - \frac{1}{2} \log D_K  + \frac{\gamma_{\mathbb{Q}} + \log 2\pi}{2} \cdot n_K - 1
\]}%
where $\gamma_{\mathbb{Q}} = 0.5772\dots$ is Euler's constant. 
Since $\sum_{1 \leq j \leq n_K} j^{-1} \leq \log n_K + 1$,
{\small
\begin{align*}
\frac{1}{\kappa_K} \sum_{\mathrm{N}\mathfrak{n} < x} \frac{1}{\mathrm{N}\mathfrak{n}} 
& \geq (\log x) \{ 1 + O_{\epsilon}( x^{-\epsilon/8} ) \} - \frac{1}{2}\log D_K + \frac{\gamma_{\mathbb{Q}} + \log 2\pi}{2} \cdot n_K - \log n_K - 2  \\
& \geq (\log x) \{ 1 - \tfrac{1}{1+2\epsilon} + O_{\epsilon}( (\log x)^{-1} ) \},
\end{align*}}%
by our assumption on $x$.
\end{proof}

Taking the logarithmic derivative of $\zeta_K(s)$ yields in the usual way
{\small
\begin{equation}
-\frac{\zeta_K'}{\zeta_K}(s) = \sum_{\mathfrak{n} \subseteq \cO_K} \frac{\Lambda_K(\mathfrak{n})}{(\N\mathfrak{n})^s}
\label{eqn:LogDiffZeta}
\end{equation}}%
for $\Re\{s\}>1$, where $\Lambda_K( \, \cdot \,)$ is the von Mangoldt $\Lambda$-function of the field $K$ defined by
{\small
\begin{equation}
\Lambda_K(\mathfrak{n}) = 
	\begin{cases}
		\log \N\mathfrak{p} & \text{if $\mathfrak{n}$ is a power of a prime ideal $\mathfrak{p}$,} \\
		0 & \text{otherwise.}
	\end{cases}
	\label{def:vonMangoldt}
\end{equation}}%
Using this identity, we prove an elementary lemma.

\begin{lem} \label{lem:PrimeSum}
	For $y \geq 3$ and $0 < r < 1$,
	\begin{enumerate}[(i)]
		\item {\small
$\ds - \frac{\zeta_K'}{\zeta_K}(1+r) = \sum_{\mathfrak{n}} \frac{\Lambda_K(\mathfrak{n})}{\N\mathfrak{n}^{1+r}} \leq \frac{1}{2}\log D_K + \frac{1}{r}+1$.}
		\item {\small
		$\ds\sum_{\N\mathfrak{n} \leq y} \frac{\Lambda_K(\mathfrak{n})}{\N\mathfrak{n}} \leq e \log(eD_K^{1/2}y)$.}
	\end{enumerate}
\end{lem}
\begin{proof} Part (i) follows from \cref{ExplicitFormula,digamma}, \eqref{eqn:LogDiffZeta}, and the fact that $\Re\{ (1+r-\rho)^{-1}\} \geq 0$. Part (ii) follows from Part (i) by taking $\sigma = 1+ \frac{1}{\log y}$.
\end{proof}
Finally, we end this section with a bound for $h_H$ in terms of $n_K, D_K,$ and $Q=Q_H$ and a comparison between $Q$ and $\N\kf_H$.

\begin{lem}
\label{lem:h_H-Bound}
	Let $H$ be a congruence class group of $K$. For $\epsilon > 0$, $h_H \leq e^{O_{\epsilon}(n_K)} D_K^{1/2+\epsilon} Q^{1+\epsilon}$.
\end{lem}
\begin{proof} Observe, by the definitions of $Q$ and $\mathfrak{f}_H$ in \cref{subsec:notation}, that for a Hecke character $\chi\pmod{H}$ we have $\kf_{\chi} \mid \kf_H$ and $\N\kf_{\chi} \leq Q$. Hence,  
	{\small
	\[
	h_H = \sum_{ \chi \pmod{H} } 1   \leq \sum_{\substack{ \N\kf \leq Q \\ \kf \,\mid\, \kf_H} } ~\sum_{\chi \pmod{\kf} } 1 = \sum_{\substack{ \N\kf \leq Q \\ \kf \,\mid\, \kf_H} } \#\mathrm{Cl}(\kf). 
	\]}%
	Recall the classical bound $\#\mathrm{Cl}(\kf) \leq 2^{n_K} h_K \N\kf$ where $h_K$ is the class number of $K$ (in the broad sense) from \cite[Theorem 1.7]{milneCFT}, for example. Bounding the class number using Minkowski's bound (see \cite[Lemma 1.12]{Weiss} for example), we deduce that
	{\small
	\[
	h_H \leq \sum_{\substack{ \N\kf \leq Q \\ \kf \,\mid\, \kf_H} } e^{O_{\epsilon}(n_K)} D_K^{1/2+\epsilon} \N\kf \leq e^{O_{\epsilon}(n_K)} D_K^{1/2+\epsilon} Q^{1+\epsilon}  \sum_{ \kf \,\mid\, \kf_H} \frac{1}{(\N\kf)^{\epsilon}}.
	\]}%
	For the remaining sum, notice $\sum_{ \kf \,\mid\, \kf_H}(\N\kf)^{-\epsilon} \leq \prod_{\mathfrak{p} \mid \kf_H} (1-\N\mathfrak{p}^{-\epsilon})^{-1} \leq  e^{O(\omega(\kf_H))}$, where $\omega(\kf_H)$ is the number of prime ideals $\mathfrak{p}$ dividing $\kf_H$. From \cite[Lemma 1.13]{Weiss}, we have $\omega(\kf_H) \ll  O_{\epsilon}(n_K)  + \epsilon \log(D_KQ)$ whence the desired estimate follows after rescaling $\epsilon$. 
\end{proof}

\begin{remark}
Weiss \cite[Lemma 1.16]{Weiss} achieves a comparable bound with $Q^{1+\epsilon}$ replaced by $\mathrm{N}\kf_H$.  This seemingly minor difference will in fact improve the range of $T$ in \cref{LFZD-MainTheorem}.
\end{remark}

\begin{lem}
\label{lem:MaxConductor}
Let $H$ be a congruence class group of $K$. Then $Q \leq \N\kf_H \leq Q^2.$
\end{lem}
\begin{remark}
	The lower bound is achieved when $H = P_{\kf_H}$. We did not investigate the tightness of the upper bound as this estimate will be sufficient our purposes. 
\end{remark}
\begin{proof} The arguments here are motivated by \cite[Lemma 1.13]{Weiss}. Without loss, we may assume $H$ is primitive.  Since $Q = Q_H = \max\{ \N\kf_{\chi} : \chi \pmod{H} \}$ and $\kf_H = \mathrm{lcm}\{ \kf_{\chi} : \chi \pmod{H} \}$, the lower bound is immediate. For the upper bound, we apply arguments similar to \cite[Lemma 1.13]{Weiss}. Consider any $\km \mid \kf_H$. Let $H_{\km}$ denote the image of $H$ under the map $I(\kf_H)/P_{\kf_H} \to I(\km)/P_{\km}$. This induces a map $I(\kf_H)/H \to I(\km)/H_{\km}$, which, since $H$ is primitive, must have non-trivial kernel. Hence, characters of $I(\km)/H_{\km}$ induce characters of $I(\kf_H)/H$. 
	
	Now, for $\mathfrak{p} \mid \kf_H$, choose $e = e_{\mathfrak{p}} \geq 1$ maximum so that $\mathfrak{p}^e \mid \kf_H$. Define $\km_{\mathfrak{p}} := \kf_H \mathfrak{p}^{-1}$ and consider the induced map $I(\kf_H)/H \to I(\km_{\mathfrak{p}})/H_{\km_{\mathfrak{p}}}$ with kernel $V_{\mathfrak{p}}$. Since $H$ is primitive, $V_{\mathfrak{p}}$ must be non-trivial and hence $\#V_{\mathfrak{p}} \geq 2$. Observe that the characters $\chi$ of $I(\kf_H)/H$ such that $\mathfrak{p}^e \nmid \kf_{\chi}$ are exactly those which are trivial on $V_{\mathfrak{p}}$ and hence are $\frac{h_H}{\#V_{\mathfrak{p}}}$ in number. For a given $\mathfrak{p}$, this yields the following identity:
	{\small
	\[
	\frac{h_H}{2} \leq h_H \Big(1 - \frac{1}{\#V_{\mathfrak{p}}}\Big)  = \sum_{\substack{ \chi \pmod{H} \\ \mathfrak{p}^{e_{\mathfrak{p}}} \| \kf_{\chi} }} 1. 
	\]}%
	Multiplying both sides by $\log(\N\mathfrak{p}^{e_{\mathfrak{p}}})$ and summing over $\mathfrak{p} \mid \kf_H$, we have
	{\small
	\begin{align*}
\frac{1}{2}h_H \log \N\kf_H= \frac{h_H}{2} \sum_{\mathfrak{p} \mid \kf_H} \log(\N\mathfrak{p}^{e_{\mathfrak{p}}})   \leq 	\sum_{\mathfrak{p} \mid \kf_H} \sum_{\substack{ \chi \pmod{H} \\ \mathfrak{p}^{e_{\mathfrak{p}}} \| \kf_{\chi} }} \log \N\mathfrak{p}^{e_{\mathfrak{p}}}  \leq \sum_{\chi \pmod{H}} \log \N\kf_{\chi}  \leq h_H \log Q.
	\end{align*}}%
	Comparing both sides, we deduce $\N\kf_H \leq Q^2$ as desired. 
\end{proof}

\begin{lem} \label{lem:ImprimitiveSubstitute} Let $H$ be a congruence subgroup of $K$ and $\epsilon > 0$ be arbitrary. Then $\sum_{\mathfrak{p} \mid \kf_H} \frac{\log \N\mathfrak{p}}{\N\mathfrak{p}} \leq (2\epsilon)^{-1} n_K + \epsilon \log Q$.
\end{lem}
\begin{proof}
This follows from \cite[Lemma 2.4]{Zaman_2015a} and \cref{lem:MaxConductor}. 
 \end{proof}

\section{Proof of \cref{thm:main_theorem} and Linnik's Three Principles}
\label{sec:outline} 

The goal in this paper is to prove the following result, from which Theorem \ref{thm:main_theorem} follows.

\begin{thm}
\label{thm:LPI-MainTheorem}
Let $K$ be a number field, let $H \pmod{\kq}$ be a congruence class group of $K$, and let $\mathfrak{f}_H$ be the conductor of $H$. Let $I(\mathfrak{q})$ be the group of fractional ideals of $K$ which are coprime to $\mathfrak{q}$ and $\cC \in I(\kq)/H$ be arbitrary. Let $\chi \pmod{H}$ be a character of $I(\kq)/H$ of conductor $\kf_{\chi}$. Let $h_H = [I(\kq) : H]$, $Q=\max\{\mathrm{N}_{K/\mathbb{Q}}\kf_{\chi}:\chi \pmod{H} \}$, and $\km$ be the product of prime ideals dividing $\kq$ but not $\kf_H$.  If 
\[
x \geq D_K^{694} Q^{521}+ D_K^{232} Q^{367} n_K^{290n_K}+ (D_K Q n_K^{n_K})^{1/1000} \N_{K/\mathbb{Q}}\km,
\label{eqn:LPI-MainTheorem_xRange}
\]
and $D_K Q [K:\mathbb{Q}]^{[K:\mathbb{Q}]}$ is sufficiently large then
{\small
\[
\#\{ \mathfrak{p} \in \cC:  \mathrm{deg}(\mathfrak{p})=1,  \N_{K/\mathbb{Q}}~\mathfrak{p} \leq x\} \gg  (D_K Q n_K^{n_K} )^{-5} \frac{x}{h_H \log x} 
\]}%
where the implied constant is effectively computable and absolute. 
\end{thm}



Assuming \cref{thm:LPI-MainTheorem}, we now prove \cref{thm:main_theorem}.

\begin{proof}[Proof of \cref{thm:main_theorem}]
The proof proceeds exactly as in \cite[Theorem 6.1]{Weiss}.  Let $L/F$ be a finite Galois extension of number fields with Galois group $G$, and let $C\subset G$ be a given conjugacy class.  Let $A\subset G$ be an abelian subgroup such that $A\cap C$ is nonempty, and let $K=L^A$ be the fixed field of $A$.  Let $\mathfrak{f}_{L/K}$ be the conductor of $L/K$, and let $\mathfrak{m}$ be the product of prime ideals $\mathfrak{P}$ in $K$ which are unramified in $L$ but so that the prime $\mathfrak{p}$ of $F$ lying under $\mathfrak{P}$ is ramified in $L$. If $[L/K,\mathfrak{P}]$ denotes the Artin symbol, then the Artin map $\mathfrak{P}\mapsto[\tfrac{L/K}{\mathfrak{P}}]$ induces a group homomorphism $I(\mathfrak{m}\mathfrak{f}_{L/K})\to A$ because the conjugacy classes in $A$ are singletons; thus if $H$ is the kernel of the homomorphism, then the canonical map $\omega\colon I(\mathfrak{m}\mathfrak{f}_{L/K})/H\to A$ is an isomorphism.  Moreover, $H$ is a congruence class group modulo the ideal $\mathfrak{m}\mathfrak{f}_{L/K}$ of $K$ with $\mathfrak{f}_H=\mathfrak{f}_{L/K}$.

Choose $\sigma_0\in C\cap A$.  Using $\omega$, $\sigma_0$ determines a coset of $I(\mathfrak{m}\mathfrak{f}_H)/H$; thus by \cref{thm:LPI-MainTheorem}, if
\begin{equation}
\label{eqn:lower_bound_x}
x \geq D_K^{694} Q^{521}+ D_K^{232} Q^{367} n_K^{290n_K}+(D_K Q n_K^{n_K})^{1/1000} \N_{K/\mathbb{Q}}\km,
\end{equation}
and $D_K Q n_K^{n_K}$ is sufficiently large then
{\small
\[
\#\{ \N_{K/\mathbb{Q}}\kP \leq x:  \mathrm{deg}(\kP)=1, [\tfrac{L/K}{\kP}] = \{\sigma_0\} \} \gg  (D_K Q n_K^{n_K})^{-5} \frac{x}{h_H \log x}.
\]}%
Let $\mathfrak{p}$ be a prime ideal of $F$ lying under $\mathfrak{P}$.  By the definition of $\mathfrak{m}$, $\mathfrak{p}$ is unramified in $L$ and $\mathrm{N}_{K/\mathbb{Q}}\mathfrak{P}=\mathrm{N}_{F/\mathbb{Q}}~\mathfrak{p}$ because $\mathrm{deg}(\mathfrak{P})=1$.  Furthermore, $[L/F,\mathfrak{p}]=C$.  Thus if $x$ satisfies \eqref{eqn:lower_bound_x}, 
{\small
\[
\#\{ \mathfrak{p}:  \mathrm{deg}(\mathfrak{p})=1, [\tfrac{L/F}{\mathfrak{p}}] = C,\N_{F/\mathbb{Q}}~\mathfrak{p}\leq x \} \gg  (D_K Q n_K^{n_K})^{-5} \frac{x}{h_H \log x}.
\]}%

As in \cite[Theorem 6.1]{Weiss}, $\mathrm{N}_{K/\mathbb{Q}}\km\leq D_K$ and $h_H=[L:K]$.  By the definition of $Q$ and the definition of $H$, we have that $Q=\mathcal{Q}$, so
{\small
\begin{equation}
\label{eqn:lower_bound_pi_C}
\pi_C(x,L/F)\gg(D_K\mathcal{Q}n_K^{n_K})^{-5}\frac{x}{[L:K]\log x}
\end{equation}}%
whenever $D_K Q n_K^{n_K}$ is sufficiently large and $x\geq D_K^{694} \mathcal{Q}^{521}+ D_K^{232} \mathcal{Q}^{367}n_K^{290n_K} + D_K \mathcal{Q} n_K^{n_K}$.  Since $D_K \mathcal{Q} n_K^{n_K} \leq D_L n_L^{n_L}$ and there are only finitely many number fields $L$ with $D_L n_L^{n_L}$ not sufficiently large, we may enlarge the implied constant in \cref{thm:main_theorem} to allow for those exceptions and complete the proof.
\end{proof}

To outline our proof of \cref{thm:LPI-MainTheorem}, we recall the modern approach to proving Linnik's bound on the least prime in an arithmetic progression.  In order to obtain small explicit values of $\Cr{c1}$ in \eqref{eqn:Linnik}, one typically requires three principles, explicit versions of which are recorded in \cite[Section 1]{HBLinnik}:
\begin{itemize}
\item A zero-free region for Dirichlet $L$-functions:  if $q$ is sufficiently large, then the product $\prod_{\chi\pmod q}L(s,\chi)$ has at most one zero in the region
{\small
\begin{equation}
\label{eqn:ZFR_Dirichlet}
s=\sigma+it,\qquad \sigma\geq1-\frac{0.10367}{\log(q(2+|t|))}.
\end{equation}}%
If such a zero exists, it is real and simple and its associated character is also real. 

\item A ``log-free'' zero density estimate:   If $q$ is sufficiently large, $\epsilon>0$, and we define $N(\sigma,T,\chi)=\#\{\rho=\beta+i\gamma:L(\rho,\chi)=0,|\gamma|\leq T,\beta\geq\sigma\}$, then
{\small
\begin{equation}
\label{eqn:LFZDE_Dirichlet}
\sum_{\chi\bmod q}N(\sigma,T,\chi)\ll_{\epsilon} (qT)^{(\frac{12}{5}+\epsilon)(1-\sigma)},\qquad T\geq1.
\end{equation}}%
\item The zero repulsion phenomenon:  if $q$ is sufficiently large, $\lambda>0$ is sufficiently small, $\epsilon>0$, and the exceptional zero in the region \eqref{eqn:ZFR_Dirichlet} exists and equals $1-\lambda/\log q$, then $\prod_{\chi\pmod q}L(s,\chi)$ has no other zeros in the region
{\small
\begin{equation}
\label{eqn:DH_Dirichlet}
\sigma\geq1-\frac{(\frac{2}{3}-\epsilon)(\log \lambda^{-1})}{\log(q(2+|t|))}.
\end{equation}}%
If such an exceptional zero exists, then it is real and simple and it corresponds with a non-trivial real character $\chi$.
\end{itemize}

Number field variants of these principles were proved by Fogels \cite{Fogels}, but his proof did not maintain the necessary field uniformity.  To prove \eqref{eqn:Weiss}, Weiss developed variants of these principles with effective number field dependence; the effective field dependence is critical for the proof of \eqref{eqn:Weiss}.  To prove \cref{thm:LPI-MainTheorem}, we make Weiss' field-uniform results explicit.



In Sections \ref{sec:MeanValue}-\ref{sec:LFZD}, we prove an explicit version of Weiss' variant of \eqref{eqn:LFZDE_Dirichlet} for Hecke characters \cite[Corollary 4.4]{Weiss}.  Assume the notation in the previous section, and define
\[
N(\sigma,T,\chi) := \#\{ \rho = \beta+i\gamma : L(\rho, \chi) = 0, \sigma < \beta < 1, |\gamma| \leq T \}
\]
where the nontrivial zeros $\rho$ of $L(s,\chi)$ are counted with multiplicity.  Weiss \cite[Corollary 4.4]{Weiss} proved that there exists an absolute constant $\Cl[abcon]{Weiss2}>0$ such that if $\tfrac{1}{2} \leq \sigma < 1$ and $T \geq n_K^2 h_H^{1/n_K}$, then
{\small
\begin{equation}
\sum_{\substack{\chi \pmod{H} }} N(\sigma, T,\chi) \ll (e^{O(n_K)}D_K^2 Q T^{n_K})^{\Cr{Weiss2}(1-\sigma)}
\label{eqn:WeissDensity}
\end{equation}}%
with an absolute effective implied constant. We prove an explicit bound on $\Cr{Weiss2}$.

\begin{thm}
\label{LFZD-MainTheorem}
Let $H$ be a congruence class group of a number field $K$.  If $\tfrac{1}{2} \leq \sigma < 1$ and $T \geq \max\{n_K^{5/6} D_K^{-4/3n_K} Q^{-4/9n_K},1\}$, then 
{\small
\begin{equation}
\sum_{\substack{\chi\pmod{H}}} N(\sigma, T,\chi) \ll \{ e^{O(n_K)} D_K^{2} Q T^{n_K+2}\}^{81(1-\sigma)}.
\label{eqn:LFZD-MainTheorem}
\end{equation}}%
All implied constants are absolute. If $1- 10^{-3} \leq\sigma<1$, then one may replace 81 with 73.5.
\end{thm}
~
\begin{remarks}
~
\begin{itemize}
	\item \cref{LFZD-MainTheorem} noticeably improves Weiss' density estimate \eqref{eqn:WeissDensity} in the range of $T$. If $n_K \leq \mathrm{rd}_K^{1.6}$ where $\mathrm{rd}_K = D_K^{1/n_K}$, then \cref{LFZD-MainTheorem} holds for $T \geq 1$.
	\item By appealing to Minkowski's lower bound for $D_K$ and the valid range of $T$, we have that the $e^{O(n_K)}$ factor is always negligible, regardless of how $n_K$ compares to $\mathrm{rd}_K$.
%
\end{itemize}
\end{remarks} 


We prove Theorem \ref{LFZD-MainTheorem} by constructing a Dirichlet polynomial which is bounded away from zero when in close proximity to a nontrivial zero of a Hecke $L$-function.  This is ensured by using the Tur\'an power sum method (see \cref{prop:ZeroDetector}).  The contributions from the detected zeros are summed efficiently using a large sieve inequality for Hecke characters (see \cref{thm:LargeSieve}).  In order to maintain field uniformity in our large sieve inequality, the Selberg sieve is used instead of the usual duality arguments; see \cref{sec:MeanValue} for more details.

In order to bound sums over integral ideals, we are required to smooth the sums using a kernel which is ${n_K}$-times differentiable.  Unfortunately, the smoothing introduces the powers of ${n_K}^{{n_K}}$  (see the comments immediately preceding \cite[Section 1]{Weiss}).  As mentioned after \cref{thm:main_theorem}, the factor of ${n_K}^{n_K}$ is negligible if $n_K$ is small compared to $\log D_K/\log\log D_K$, which implies that the root discriminant of $K$ is large.  The situations where the root discriminant of $K$ is small are very rare; the only commonly known example of such a situation is when considering infinite $p$-class tower extensions.

We note that in the case of bounding the least prime in an arithmetic progression, Tur\'an's power sum method does not produce strong numerical results.  Instead, one typically constructs a suitable mollifier for Dirichlet $L$-functions relying on M{\"o}bius cancellation.  However, relying on M{\"o}bius cancellation for Hecke $L$-functions introduces super-polynomial dependence on $D_K$ in Theorem \ref{LFZD-MainTheorem}, causing a significant decrease in the quality of the field dependence for bounds in \cref{thm:LPI-MainTheorem}.  To the authors' knowledge, the only device by which one can detect zeros to prove a log-free zero density estimate while maintaining suitable field uniformity is the Tur\'an power sum.

In \cref{sec:DH_proof}, we prove an explicit variant of the zero repulsion phenomenon for Hecke $L$-functions.

\begin{thm}
\label{DH-MainTheorem}
Let $H$ be a congruence class group of $K$. Let $\psi \pmod{H}$ be a real Hecke character and suppose $L(s,\psi)$ has a real zero $\beta_1$. Let $T \geq 1$ be arbitrary, and $\chi \pmod{H}$ be an arbitrary Hecke character and let $\rho' = \beta'+i\gamma'$ be a zero of $L(s,\chi)$ satisfying $\frac{1}{2} \leq \beta' < 1$ and $|\gamma'| \leq T$.  Then, for $\epsilon > 0$ arbitrary,
{\small
\[
\beta' \leq 1 - \frac{ \log\Big( \dfrac{c_{\epsilon}}{(1-\beta_1) \log(D_K \cdot Q \cdot T^{n_K} e^{O_{\epsilon}(n_K)}) } \Big) }{b_1  \log D_K + b_2 \log Q  + b_3 n_K \log T + O_{\epsilon}(n_K) }
\]}%
for some absolute, effective constant $c_{\epsilon} > 0$ and where
{\small
\[
(b_1,b_2,b_3) = 
\begin{cases}
(48+\epsilon,60+\epsilon,24+\epsilon)  & \text{if $\psi$ is quadratic}, \\
(24+\epsilon,12+\epsilon,12+\epsilon) & \text{if $\psi$ is trivial.}
\end{cases}
\]}%
\end{thm}  
\begin{remark}
Other versions of the zero repulsion phenomenon by Kadiri and Ng \cite{KadiriNg} and Zaman \cite{Zaman_2015a} apply for an asymptotically smaller range of $\beta'$ and $|\gamma'| \leq 1$.
\end{remark}

%

In \cref{sec:ThreePrinciples}, we collect all existing results and our new theorems on the distribution of zeros of Hecke $L$-functions and package them into versions required for the proof of \cref{thm:LPI-MainTheorem}. The necessary explicit zero-free regions for Hecke $L$-functions have already been in established in previous work of Zaman \cite{Zaman_2015a,Zaman_thesis}, which improved on \cite{Ahn-Kwon,Kadiri}, and are valid in a certain neighborhood of $s=1$. In Sections \ref{sec:LPI_Preliminaries}--\ref{sec:LPI_NonEX}, we will use Theorems \ref{LFZD-MainTheorem} and \ref{DH-MainTheorem}, along with the aforementioned work of Zaman, to prove Theorem \ref{thm:LPI-MainTheorem}.  Finally, in \cref{sec:proof_MF_Fourier}, we prove Theorems \ref{thm:least_prime_QF}--\ref{thm:MF_Fourier} which are applications of \cref{thm:main_theorem}.

\section{Mean Value of Dirichlet Polynomials}
\label{sec:MeanValue}

In \cite{Gallagher}, Gallagher proves a large sieve inequality of the following form.
\begin{thm*}[Gallagher]
Let $\{a_n\}$ be a sequence of complex numbers with the property that $\sum_{n\geq 1}n|a_n|^2<\infty$.  Assume that $a_n=0$ if $n$ has any prime factor less than $R\geq2$.  If $T\geq 1$, then
{\small
\[
\sum_{q\leq R}\log\frac{R}{q}~\sideset{}{^*}\sum_{\chi\bmod q}\int_{-T}^{T}\Big|\sum_{n\geq 1}a_n\chi(n)n^{it}\Big|^2 dt\ll \sum_{n\geq 1}(R^2 T+n)|a_n|^2,
\]}%
where $\sideset{}{^*}\sum$ denotes a restriction of the summation to primitive Dirichlet characters.
\end{thm*} 

The $\log R/q$ savings, which arises from forcing $a_n=0$ when $n$ has a small prime factor, turns out to be decisive in certain applications, such as the  proof of \eqref{eqn:Linnik}.  The key ingredients in its proof  are the duality argument, properties of Gauss sums, and the fact that the Farey fractions up to height $R$ are $R^{-2}$-well-spaced. Optimistically speaking, we would extend these arguments from Dirichlet characters to Hecke characters but, apart from the duality argument, sufficiently strong analogues of these results over number fields do not yet exist.  In order to circumvent these deficiencies, we use the Selberg sieve to prove a variant of Gallagher's result where the $\log R/q$ term on the left hand side is translated to a $(\log R)^{-1}$ savings on the right hand side.  The use of the Selberg sieve introduces several sums over integral ideals whose evaluation requires smoothing.  Ultimately, this introduces the factor of $n_K$ in the lower bound for $T$ in \cref{LFZD-MainTheorem}.

Our desired analogue of Gallagher's theorem is as follows:

\begin{thm} \label{thm:LargeSieve}
Let $\upsilon \geq \epsilon > 0$ be arbitrary.  Let $b( \, \cdot \,)$ be a complex-valued function on the prime ideals $\mathfrak{p}$ of $K$ such that $\sum_{\mathfrak{p}}|b(\mathfrak{p})|<\infty$ and $b(\mathfrak{p})=0$ whenever $\N\mathfrak{p}\leq y$.  Let $H$ be a primitive congruence class group of $K$.  If $T\geq1$ and
\begin{equation}
y\geq  C_{\epsilon}   \big\{h_H n_K^{(5/4+\upsilon)n_K} D_K^{3/2+\upsilon} Q^{1/2} T^{n_K/2+1}  \big\}^{1+\epsilon}
\label{eqn:LargeSieve_yRange}
\end{equation}
for some sufficiently large $C_{\epsilon} > 0$ then
{\small
\begin{equation}
\begin{aligned}
\sum_{\substack{\chi \pmod{H}}}\int_{-T}^{T}\left|\sum_{\mathfrak{p}}b(\mathfrak{p})\chi(\mathfrak{p})\N\mathfrak{p}^{-it}\right|^2 dt 
& \leq \Big( \frac{ 5\pi  \{ 1-\frac{1}{1+\upsilon} \}^{-1}}{\tfrac{1}{1+\epsilon} \log(\frac{y}{h_H}) -  \cL'} + O_{\epsilon}(y^{-\frac{\epsilon}{2}}) \Big) \sum_{\mathfrak{p}} \N\mathfrak{p} |b(\mathfrak{p})|^2,
\end{aligned}
\label{eqn:LargeSieve_Bound}
\end{equation}}%
where $\cL' = \tfrac{1}{2}\log D_K + \tfrac{1}{2}\log Q + \tfrac{1}{4} n_K \log n_K + (\tfrac{n_K}{2}+1)\log T + O_{\epsilon}(1)$.
\end{thm} 
\begin{remark} 
Weiss' analogous result \cite[Corollary 3.8]{Weiss} holds when $y \geq (h_H n_K^{2 n_K} D_K Q T^{2n_K})^8$.  The exponent 8 is large enough to inflate $\Cr{Weiss_1}$ and $\Cr{Weiss_2}$ in \eqref{eqn:Weiss}.  \cref{thm:LargeSieve} ensures that the size of $y$ does not affect the exponent $81$ in \cref{LFZD-MainTheorem}.
\end{remark}

By taking $\upsilon = \epsilon$ and applying \cref{lem:h_H-Bound} to bound $h_H$, \eqref{eqn:LargeSieve_Bound} reduces to
{\small
\begin{align*}
\sum_{\substack{\chi \pmod{H}}}\int_{-T}^{T}\left|\sum_{\mathfrak{p}}b(\mathfrak{p})\chi(\mathfrak{p})\N\mathfrak{p}^{-it}\right|^2 dt 
& \ll_{\epsilon} \frac{1}{\log y}  \sum_{\mathfrak{p}} \N\mathfrak{p} |b(\mathfrak{p})|^2,
\end{align*}}%
which may be of independent interest.

The sole objective of this section is to establish \cref{thm:LargeSieve}. 

\subsection{Preparing for the Selberg Sieve} To apply the Selberg sieve, we will require several weighted estimates involving Hecke characters. Before we begin, we highlight the necessary properties of our weight $\Psi$. 
 
\begin{lem} \label{lem:WeightFunction}
	For $T \geq 1$, let $A = T \sqrt{2n_K}$. There exists a compactly supported weight function $\Psi :(0,\infty)\to \mathbb{R}$ with Mellin transform $\widehat{\Psi}(s)$ such that:
	\begin{enumerate}[(i)]
		\item $0 \leq \Psi(x) \leq A/2$ and $\Psi(x)$ vanishes outside the interval $e^{-2n_K/A} \leq x \leq e^{2n_K/A}$.
		\item $\widehat{\Psi}(s)$ is an entire function; specifically, $\widehat{\Psi}(s) = [ \frac{\sinh(s/A)}{s/A}]^{2n_K}$. 
		\item For all complex $s = \sigma+it$, $|\widehat{\Psi}(s)| \leq ( A/|s| )^{2n_K} e^{|\sigma|/A}$.
		\item For $|s| \leq A$, $|\widehat{\Psi}(s)| \leq (1 + |s|^2/(5A^2))^{2n_K}$. 
		\item Uniformly for $|\sigma| \leq A/\sqrt{2n_K}$, $|\widehat{\Psi}(s)| \ll 1$.  
		\item Let $\{b_m\}_{m \geq 1}$ be a sequence of complex numbers with $\sum_m |b_m| < \infty$. Then
			{\small
			\[
			\int_{-T}^T \Big| \sum_m b_m m^{-it} \Big|^2 dt \leq \frac{5\pi}{2} \int_0^{\infty} \Big| \sum_m b_m \Psi\Big( \frac{x}{m} \Big) \Big|^2 \frac{dx}{x}
			\]}%
	\end{enumerate}
\end{lem}
\begin{proof}
	For (i)--(v), see \cite[Lemma 3.2]{Weiss}; in his notation, $\Psi(x) = H_{2n_K}(x)$ with parameter $A = T\sqrt{2n_K}$. Statement (vi) follows easily from the proof of \cite[Corollary 3.3]{Weiss}. 
\end{proof}
For the remainder of this section, assume:
  \begin{itemize}
  	\item $H \pmod{\kq}$ is an arbitrary \emph{primitive} congruence class group of $K$.
 	\vspace{1mm}
  	\item $0 < \epsilon < 1/2$ and $T \geq 1$ is arbitrary. 
  	\vspace{1mm}
  	\item $\Psi$ is the weight function of \cref{lem:WeightFunction}. 
  \end{itemize}
Next, we establish improved analogues of \cite[Lemmas 3.4 and 3.6 and Corollary 3.5]{Weiss}.

\begin{lem}
\label{lem:smoothed}  Let $\chi \pmod{H}$ be a Hecke character.   For $x > 0$,
{\small
\[
\Big|\sum_{\mathfrak{n}}\frac{\chi(\mathfrak{n})}{\N\mathfrak{n}}\cdot \Psi\Big(\frac{x}{\N\mathfrak{n}}\Big)-\delta(\chi)\frac{\varphi(\kq)}{\N\kq}\kappa_K\Big|
\ll_{\epsilon}  \Big\{ n_K^{n_K/4} D_K^{1/2} Q^{1/2} T^{n_K/2+1} \Big\}^{1+\epsilon}
\]}%
\end{lem}

\begin{proof}
The quantity we wish to bound equals
{\small
\begin{equation}
\label{eqn:integral_11}
\frac{1}{2\pi i}\int_{-1-i\infty}^{-1+i\infty}L(s+1,\chi)\widehat{\Psi}(s)x^s ds.
\end{equation}}%
If $\chi\pmod{\kq}$ is induced by the primitive character $\chi^* \pmod{\mathfrak{f}_{\chi}}$, then
{\small
\[
L(s,\chi)=L(s,\chi^*)\prod_{\substack{\mathfrak{p}\mid \kq \\ \mathfrak{p}\nmid\mathfrak{f}_{\chi}}}(1-\chi^*(\mathfrak{p})\N\mathfrak{p}^{-s}).
\]}%
Thus $|L(it,\chi)|\leq 2^{\omega(\kq)}|L(it,\chi^*)|$ where $\omega(\kq)$ is the number of distinct prime ideal divisors of $\kq$. Since $H \pmod{\kq}$ is primitive,  $\omega(\kq)\leq 6e^{4/\epsilon}n_K+\tfrac{\epsilon}{2} \log(D_K Q)$, by \cite[Lemma 1.13]{Weiss}. So, for $\Re\{s\} = -1$, $|L(s+1,\chi)| \ll e^{O_{\epsilon}(n_K)} (D_K Q)^{\epsilon/2} |L(s+1,\chi^*)|$.  Thus, by \cref{Rademacher}, \eqref{eqn:integral_11} is
{\small
\[
\ll e^{O_{\epsilon}(n_K)} (D_K Q)^{\tfrac{1}{2}+\epsilon} x^{-1}\int_{0}^{\infty}(1+|t|)^{(\frac{1}{2}+\epsilon)n_K}|\widehat{\Psi}(-1+it)|dt
\]}%
as $D_{\chi} \leq D_K Q$. By \cref{lem:WeightFunction}(iii) and (iv), this integral is
{\small
\begin{align*}
\ll\int_{0}^{\frac{A}{2}}(1+|t|)^{(\frac{1}{2}+\epsilon)n_K}|\widehat{\Psi}(-1+it)|dt+\int_{\frac{A}{2}}^{\infty}(1+|t|)^{(\frac{1}{2}+\epsilon)n_K}|\widehat{\Psi}(-1+it)|dt,
\end{align*}}%
which is $\ll e^{O(n_K)}A^{(\frac{1}{2}+\epsilon)n_K+1}$.  Collecting the above estimates, the claimed bound, up to a factor of $\epsilon$, follows upon recalling $A = T\sqrt{2n_K}$ and noting $e^{O(n_K)} \ll_{\epsilon} (n_K^{n_K})^{\epsilon}$. 
\end{proof}

\begin{cor}
Let $C$ be a coset of $H$, and let $\kd$ be an integral ideal coprime to $\kq$. For all $x > 0$, we have
{\small
\[
\big|\sum_{\substack{\mathfrak{n}\in C \\ \kd \mid\mathfrak{n}}}\frac{1}{\N\mathfrak{n}}\Psi\Big(\frac{x}{\N\mathfrak{n}}\Big)-\frac{\varphi(\kq)}{\N\kq}\frac{\kappa_K}{h_H} \cdot \frac{1}{\N\kd}\big|\ll_{\epsilon} \big\{ n_K^{n_K/4} D_K^{1/2} Q^{1/2} T^{n_K/2+1} \big\}^{1+\epsilon} \cdot \frac{1}{x}.
\]}%
\end{cor}
\begin{proof}
The proof is essentially the same as that of \cite[Corollary 3.5]{Weiss}, except for the fact that we have an improved bound in Lemma \ref{lem:smoothed}.
\end{proof}

We now apply the Selberg sieve.  For $z\geq 1$, define
{\small
\begin{equation}
S_z=\{\mathfrak{n}\colon\mathfrak{p}\mid\mathfrak{n}\implies\N\mathfrak{p}>z\} \qquad \text{and} \qquad V(z)=\sum_{\N\mathfrak{n}\leq z}\frac{1}{\N\mathfrak{n}}.
\end{equation}}%

\begin{lem}
\label{lem:3.3}
Let $C$ be a coset of $H$.  For $x > 0$ and $z \geq 1$,  
{\small
\[
\sum_{\mathfrak{n}\in C\cap S_z}\frac{1}{\N\mathfrak{n}} \Psi\Big(\frac{x}{\N\mathfrak{n}}\Big)\leq \frac{\kappa_K}{h_H V(z)}+O_{\epsilon}\Big( \frac{\big\{ n_K^{n_K/4} D_K^{1/2} Q^{1/2} T^{n_K/2+1} \big\}^{1+\epsilon} z^{2+2\epsilon}}{x} \Big).
\]}%
\end{lem}
\begin{proof}
The proof is essentially the same as that of \cite[Lemma 3.6]{Weiss}, except for the fact that we have an improved bound in Lemma \ref{lem:smoothed}.
\end{proof}
\subsection{Proof of \cref{thm:LargeSieve}}
Let $z$ be a parameter satisfying $1 \leq z \leq y$, which we will specify later. Extend $b(\mathfrak{n})$ to all integral ideals $\mathfrak{n}$ of $K$ by zero. Applying \cref{lem:WeightFunction} and writing $b_m=\sum_{\N\mathfrak{n}=m}b(\mathfrak{n})\chi(\mathfrak{n})$, for each Hecke character $\chi \pmod{H}$, it follows that
{\small
\begin{equation}
\sum_{\chi \pmod{H}}\int_{-T}^{T}\left|\sum_{\mathfrak{n}}b(\mathfrak{n})\chi(\mathfrak{n})\N\mathfrak{n}^{-it}\right|^2 dt
\leq \frac{5\pi}{2} \int_0^{\infty}\sum_{\chi \pmod{H}}\left|\sum_{\mathfrak{n}}b(\mathfrak{n})\chi(\mathfrak{n})\Psi\Big(\frac{x}{\N\mathfrak{n}}\Big)\right|^2\frac{dx}{x}.
\label{eqn:LargeSieve_smoothed}
\end{equation}}%
 By the orthogonality of characters and the Cauchy-Schwarz inequality,
{\small
\[
 \sum_{\chi \pmod{H}}\Big|\sum_{\mathfrak{n}}b(\mathfrak{n})\chi(\mathfrak{n}) \Psi\Big(\frac{x}{\N\mathfrak{n}}\Big)\Big|^2 \leq
  h_H \sum_{C \in I(\kq)/H} \Big( \sum_{\mathfrak{n}\in C}\N\mathfrak{n} |b(\mathfrak{n})|^2 \Psi \Big(\frac{x}{\N\mathfrak{n}}\Big) \Big) \sum_{\mathfrak{n}\in C\cap S_z}\frac{\Psi(\frac{x}{\N\mathfrak{n}})}{\N\mathfrak{n}}
\]}%
since $z \leq y$ and $b(\mathfrak{n})$ is supported on prime ideals with norm greater than $y$. For $\delta = \delta(\epsilon) > 0$ sufficiently small and $B_{\delta} > 0$ sufficiently large, denote 
\[
M'_{\delta} = M_{\delta} z^{2+2\delta} \qquad \text{and} \qquad M_{\delta} = B_{\delta}  \big\{ n_K^{n_K/4} D_K^{1/2} Q^{1/2} T^{n_K/2+1} \big\}^{1+\delta}.
\]
By Lemma \ref{lem:3.3}, the RHS of the preceding inequality is therefore at most
{\small
\begin{align*}
 \sum_{C \in I(\kq)/H} \sum_{\mathfrak{n}\in C}\N\mathfrak{n} |b(\mathfrak{n})|^2 \Psi \left(\frac{x}{\N\mathfrak{n}}\right)\Big( \frac{\kappa_K}{V(z)}+ \frac{h_H M'_{\delta}}{x} \Big)   \leq \sum_{\mathfrak{n}} \N\mathfrak{n}|b(\mathfrak{n})|^2 \Psi\Big(\frac{x}{\N\mathfrak{n}}\Big) \Big( \frac{\kappa_K}{V(z)}+ \frac{h_H M'_{\delta}}{x} \Big),
\end{align*}}%
Combining the above estimates into \eqref{eqn:LargeSieve_smoothed} yields
{\small
\begin{align*}
&\sum_{\chi \pmod{H}}\int_{-T}^{T}\left|\sum_{\mathfrak{p}}b(\mathfrak{p})\chi(\mathfrak{p})\N\mathfrak{p}^{-it}\right|^2 dt\\
\leq~&\frac{5\pi}{2} \sum_{\mathfrak{n}} \N\mathfrak{n} |b(\mathfrak{n})|^2\left(\frac{\kappa_K}{V(z)}\int_{0}^{\infty} \Psi\Big(\frac{x}{\N\mathfrak{n}}\Big)\frac{dx}{x}+h_H M'_{\delta}\int_{0}^{\infty}\frac{1}{x}\Psi\Big(\frac{x}{\N\mathfrak{n}}\Big)\frac{dx}{x}\right)\\
\leq~&\frac{5\pi}{2}\sum_{\mathfrak{n}} \N\mathfrak{n} |b(\mathfrak{n})|^2\left(\frac{\kappa_K}{V(z)}|\widehat{\Psi}(0)|+\frac{h_H M'_{\delta}}{\N\mathfrak{n}} |\widehat{\Psi}(1)|\right). 
\end{align*}}%
by \cref{lem:WeightFunction}(v). 
Since $b(\mathfrak{n})$ is supported on prime ideals whose norm is greater than $y$, the above is $\leq \frac{5\pi}{2}(\frac{\kappa_K}{V(z)} + O(h_H M_{\delta} z^{2+2\delta}y^{-1}) ) \sum_{\mathfrak{p}} \N\mathfrak{p} |b(\mathfrak{p})|^2$.  Now, select $z$ satisfying
{\small
\begin{equation}
z = \Big( \frac{y^{(1+\delta)/(1+\epsilon)}}{h_H M_{\delta} } \Big)^{1/(2+2\delta)}
\label{eqn:LargeSieve_zChoice}
\end{equation}}%
so $1 \leq z \leq y$ and hence
{\small
\begin{equation}
\sum_{\chi \pmod{H}}\int_{-T}^{T}\left|\sum_{\mathfrak{p}}b(\mathfrak{p})\chi(\mathfrak{p})\N\mathfrak{p}^{-it}\right|^2 dt \leq \frac{5\pi}{2}\left(\frac{\kappa_K}{V(z)} + O_{\epsilon}(y^{-\epsilon/2})\right) \sum_{\mathfrak{p}} \N\mathfrak{p} |b(\mathfrak{p})|^2
\label{eqn:LargeSieve_Penultimate}
\end{equation}}%
for $\delta = \delta(\epsilon) > 0$ sufficiently small. From \eqref{eqn:LargeSieve_yRange} and \eqref{eqn:LargeSieve_zChoice},  it follows that $z \geq 3(n_K^{n_K}D_K)^{1/2+\upsilon/2}$, provided $C_{\epsilon}$ in \eqref{eqn:LargeSieve_yRange} is sufficiently large. 
Applying \cref{cor:harmonic_sum} to \eqref{eqn:LargeSieve_Penultimate}, it follows that
{\small
\begin{align*}
\sum_{\chi \pmod{H}}\int_{-T}^{T}\left|\sum_{\mathfrak{p}}b(\mathfrak{p})\chi(\mathfrak{p})\N\mathfrak{p}^{-it}\right|^2 dt
\leq \Big( \frac{ 5\pi\upsilon}{ 2\{ 1+\upsilon \} \log z + O_{\epsilon}( 1 ) } + O_{\epsilon}(y^{-\epsilon/2})\Big) \sum_{\mathfrak{p}} \N\mathfrak{p} |b(\mathfrak{p})|^2
\end{align*}}%
since $\upsilon \geq \epsilon > 0$. Finally, by \eqref{eqn:LargeSieve_yRange} and \eqref{eqn:LargeSieve_zChoice}, 
\begin{align*}
2\log z \geq \tfrac{1}{1+\epsilon} \log(\tfrac{y}{h_H}) -  \tfrac{1}{2}\{ \log D_K +  \log Q  + \tfrac{1}{2} n_K \log n_K + (n_K +2) \log T + O_{\epsilon}(1) \}.
\end{align*}
Inputting this estimate into the previous inequality, we obtain the desired conclusion.
\hfill \qed

\section{Detecting the Zeros of Hecke $L$-functions}
\label{sec:ZeroDetector}
\subsection{Notation}
\label{sec:LFZD_notation}
We first specify some additional notation to be used throughout this section. 
\subsubsection{Arbitrary Quantities}

\begin{itemize}
	\item Let $H \pmod{\kq}$ be a primitive congruence class group. 
	\item Let $\epsilon \in (0,1/8)$ and $\phi = 1 + \tfrac{4}{\pi} \epsilon + 16\epsilon^2 + 340\epsilon^{10}$. 
	\item Let $T \geq 1$. Define $Q=Q_H$ and 
		\begin{equation}
		\cL = \cL_{T, \epsilon} := \log D_K + \tfrac{1}{2}\log Q + (\tfrac{n_K}{2}+1)\log(T+3) + \Theta n_K
		\label{def:cL}
		\end{equation}
		where $\Theta = \Theta(\epsilon) \geq 1$ is sufficiently large depending on $\epsilon$.
	\item  Let $\lambda_0 > \tfrac{1}{20}$.  Suppose $\tau \in \R$ and $\lambda > 0$ satisfy
\begin{equation}
\lambda_0 \leq \lambda \leq \tfrac{1}{16} \cL \quad \text{and} \quad  |\tau| \leq T.
\label{eqn:LambdaRange} 
\end{equation}
Furthermore, denote $r = \frac{\lambda}{\cL}.$

\end{itemize}
\subsubsection{Fixed Quantities} 
\begin{itemize}
	\item Let $\alpha, \eta \in (0,\infty)$ and $\omega \in (0,1)$ be fixed.
	\item  Define $A \geq 1$ so that $A_1 = \sqrt{A^2+1}$ satisfies
{\small
\begin{equation}
A_1 = 2( 4e(1+1/\alpha))^{\alpha}  (1+ \eta).
\label{BigDerivative_A}
\end{equation}}%
	\item Let $x = e^{X \cL}$ and $y = e^{Y \cL}$ with $X, Y > 0$ given by
{\small
\begin{equation}
 \begin{aligned}
 Y = Y_{\lambda} = \frac{1}{e A_1 } \cdot \frac{1}{\alpha}\Big\{ 2\phi A + \frac{8}{\lambda} \Big\},\qquad X = X_{\lambda}   = \frac{2\log\big( \frac{2A_1}{1-\omega} \big)}{(1-\omega)}  \cdot \frac{1+\alpha}{\alpha} \Big\{ 2\phi A + \frac{8}{\lambda} \Big\},
\end{aligned}
\label{X_Y_values}
\end{equation}}%
and $\alpha, \eta, \omega$ are chosen so that $2 < Y < X$. Notice $X = X_{\lambda}$ and $Y = Y_{\lambda}$ depend on the arbitrary quantities $\epsilon$ and $\lambda$, but they are uniformly bounded above and below in terms of $\alpha,\eta$, and $\omega$, i.e. $X \asymp 1$ and $Y \asymp 1$. For this reason, while $X$ and $Y$ are technically not fixed quantities, they may be treated as such. 
\end{itemize}

\subsection{Statement of Results}
\subsubsection{Detecting Zeros}
The first goal of this section is to prove the following proposition. 
\begin{prop}
\label{prop:ZeroDetector}
Let $\chi \pmod{H}$ be a Hecke character. Suppose $L(s,\chi)$ has a non-trivial zero $\rho$ satisfying
\begin{equation}
|1+i\tau-\rho| \leq r = \tfrac{\lambda}{\cL}.
\label{DetectedZero} 
\end{equation}
Further assume
{\small
\begin{equation}
J(\lambda) := \frac{W_1 \lambda + W_2 }{A_1 (1+\eta)^{k_0}} < 1
\label{eqn:Jfunction}
\end{equation}}%
where $X = X_{\lambda}, Y = Y_{\lambda},$
{\small
\begin{equation*}
\begin{aligned}
k_0 = k_0(\lambda) & = \alpha^{-1} \big( 2\phi A \lambda + 8 \big), \\
W_1 = W_1(\lambda) & = 8A_1 \big( 1+ \tfrac{1}{k_0} \big)  + 2e A_1 \big( Y + \tfrac{1}{2} + \{ 2X + 1\} e^{-\omega \lambda X} \big) + O(\epsilon), \\
W_2 = W_2(\lambda) & = 2e\omega^{-1} A_1 e^{-\omega \lambda X}  + 18 +O(\epsilon). 
\end{aligned}
\end{equation*}}%
If $\lambda < \frac{\epsilon}{A_1} \cL$ and $2 < Y < X$ then
{\small
\[
	   r^4 \log\Big(\frac{x}{y}\Big) \int_y^x \Big| \sum_{y \leq \N\mathfrak{p} < u} \frac{\chi(\mathfrak{p}) \log \N \mathfrak{p}  }{\N\mathfrak{p}^{1+i\tau}}  \Big|^2 \frac{du}{u}  + \delta(\chi) \mathbf{1}_{\{ |\tau| < A r \}}(\tau)  \geq \Big( \frac{\alpha/(1+\alpha)}{8e2^{1/\alpha} } \Big)^{4\phi A \lambda + 16} \frac{(1- J(\lambda))^{2}}{4}.
\]}%
\end{prop}
\begin{remark}
Note that $W_j(\lambda) \ll 1$ for $j=1,2$. 
%
%
\end{remark}

The proof of \cref{prop:ZeroDetector} is divided into two main steps, with the final arguments culminating in \cref{proof:ZeroDetector}.  The method critically hinges on the following power sum estimate due to Kolesnik and Straus \cite{Kolesnik-Straus}. 
\begin{thm}[Kolesnik--Straus]
\label{KS-PowerSum} For any integer $M \geq 0$ and complex numbers $z_1,\dots,z_N$, there is an integer $k$ with $M+1 \leq k \leq M+N$ such that $|z_1^k + \cdots + z_N^k| \geq 1.007 ( \frac{N}{4e(M+N)} )^N |z_1|^k$.
\end{thm}

\subsubsection{Explicit Zero Density Estimate}
Using \cref{thm:LargeSieve,prop:ZeroDetector}, the second and primary goal of this section is to establish an explicit log-free zero density estimate. Recall, for a Hecke character $\chi$,  
\begin{equation}
N(\sigma, T, \chi) = \# \{ \rho : L(\rho,\chi) = 0, \sigma < \Re\{\rho\} < 1, |\Im(\rho)| \leq T \}.
\label{def:N_sigma_chi}
\end{equation}
where $\sigma \in (0,1)$ and $T \geq 1$. 

\begin{thm} \label{thm:LFZD_general} Let $\xi \in (1,\infty)$ and $\upsilon \in (0,\tfrac{1}{10}]$ be fixed and denote $\sigma = 1- \frac{\lambda}{\cL}$. Suppose
\begin{equation}
\lambda_0 \leq \lambda < \tfrac{\epsilon}{\xi A_1} \cL, \quad X > Y > 4.6, \quad \text{and} \quad T \geq \max\{n_K^{\frac{5}{6}} D_K^{-\frac{4}{3n_K}} Q^{-\frac{4}{9n_K}},1\}.
\label{eqn:LFZD_general_assumptions}
\end{equation}
where $X = X_{\xi \lambda}$ and $Y = Y_{\xi \lambda}$. Then,
	{\small
	\[
	 \sum_{\substack{\chi \pmod{H}}} N(\sigma,T,\chi) \leq \frac{4\xi}{\sqrt{\xi^2-1}} \cdot ( C_4 \lambda^4 + C_3 \lambda^3 + C_1 \lambda + C_0) e^{B_1 \lambda + B_2} \cdot \{1-J(\xi \lambda)\}^{-2}
	\]}%
	 where $J(\, \cdot \,)$ is defined by \eqref{eqn:Jfunction} satisfying $J(\xi \lambda) < 1$, and
	\begin{equation}
	\begin{aligned}
		B_1 & = 4\phi A \xi \log( 4e\alpha^{-1}(1+\alpha)2^{(1+\alpha)/\alpha}), \quad B_2 = 16 \log(4e \alpha^{-1}(1+\alpha)2^{(1+\alpha)/\alpha}),\\
		C_4 & = \frac{ 5\pi e  \phi  X (X-Y)^2 (X+Y + 1 + \epsilon)  \xi^4 }{ \big( 1-\frac{1}{1+\upsilon} \big) \big( \tfrac{1}{1+\epsilon} Y  - 4\big) },~ C_3 = \frac{4}{\phi \xi} C_4,~ C_1 = 4\phi A \xi,~ C_0 = 16A + \epsilon.
 	\end{aligned}	
	\end{equation}
\end{thm}
\begin{remark}
~
\begin{itemize}
	\item In \cref{sec:LFZD,subsec:ThreePrinciples_LFZD}, we will employ \cref{thm:LFZD_general} with various choices of parameters $\alpha, \eta, \upsilon, \epsilon, \omega,$ and $\xi$ depending on the range of $\sigma$. Consequently, this result is written without any explicit choice of the fixed or arbitrary quantities found in \cref{sec:LFZD_notation}. 
	\item The quantities $C_4$ and $C_3$ are technically not constants with respect to $\lambda$ or $\epsilon$ but one can see that both are bounded absolutely according to the definitions in \cref{sec:LFZD_notation}. 
\end{itemize}
\end{remark}

\cref{subsec:LargeDerivative,subsec:ShortSum} are dedicated to preparing for the proof of \cref{prop:ZeroDetector} which is contained in \cref{proof:ZeroDetector}. The proof of \cref{thm:LFZD_general} is finalized in \cref{proof:LFZD_general}. 

\subsection{A Large Derivative}
 \label{subsec:LargeDerivative}
Suppose $\chi \pmod{H}$ is induced from the primitive character $\chi^*$. Denote $F(s) := \frac{L'}{L}(s,\chi^*)$ and $z := 1 + r + i\tau$. Using \cref{KS-PowerSum}, the goal of this subsection is to show $F(s)$ has a large high order derivative, which we establish in the following lemma. 
 
 \begin{lem}
 \label{BigDerivative}
 Keep the above notation and suppose $L(s,\chi)$ has a zero $\rho$ satisfying \eqref{DetectedZero}.  If $\lambda < \frac{\epsilon}{A_1} \cL$ and $\mathbf{1}_S$ is the indicator function of a set $S$, then
{\small
\begin{align*}
 \delta(\chi) \mathbf{1}_{\{ |\tau| < Ar \}}(\tau)  + \Big| \frac{ r^{k+1}}{k!} F^{(k)}(z)  \Big|  \geq \frac{( \frac{\alpha}{4e(1+\alpha)} )^{2 \phi A \lambda + 8}}{2^{k+1}} \Big\{ 1 - \frac{ \big\{ 8 (1+ \tfrac{1}{k} ) A_1 + O(\epsilon) \big\} \lambda + 18}{A_1 (1+\eta)^{k}}    \Big\}   
\end{align*}}%
for some integer $k$ in the range $\frac{1}{\alpha} \cdot ( 2\phi A \lambda + 8 )  \leq  k \leq \frac{1+\alpha}{\alpha} \cdot ( 2\phi A \lambda + 8 ).$
 \end{lem}
 \begin{proof} By \cite[Lemma 1.10]{Weiss}, 
{\small
 \[
 F(s) + \frac{\delta(\chi)}{s-1}  = \sum_{|1+i\tau - \rho| < 1/2 } \frac{1}{s-\rho} + G(s)
 \]}%
uniformly in the region $|1+i\tau - s| < 1/2$, where $G(s)$ is analytic and $|G(s)| \ll \cL$ in this region. Differentiating the above formula $k$ times and evaluating at $z = 1+r+i\tau$, we deduce
{\small
\[
\frac{(-1)^k}{k!} \cdot  F^{(k)}(z) + \frac{\delta(\chi)}{(z-1)^{k+1}} =   \sum_{|1+i\tau - \rho| < 1/2} \frac{1}{(z-\rho)^{k+1}}  +  O(4^k \cL)
\]}%
since $r = \tfrac{\lambda}{\cL} <  \tfrac{1}{16}$ by assumption \eqref{eqn:LambdaRange}. The error term arises from bounding $G^{(k)}(z)$ using Cauchy's integral formula with a circle of radius of $1/4$.  For zeros $\rho$ that satisfy $Ar < |1+i\tau - \rho| < 1/2$, notice
\[
(A^2+1) r^2 < r^2 + |1+i\tau - \rho|^2 \leq |z-\rho|^2 \leq (r + |1+i\tau-\rho|)^2 \leq (r+1/2)^2 < 1.
\]
Recalling $A_1 = \sqrt{A^2+1}$, it follows by partial summation that 
{\small
\[
	\sum_{Ar < |1+i\tau-\rho| < 1/2 } \frac{1}{|z-\rho|^{k+1}} \leq \int_{A_1 r}^{1} u^{-k-1} dN_{\chi}(u; z )  =  (k+1) \int_{A_1 r}^{1} \frac{N_{\chi}(u; z)}{u^{k+2}} du + O(\cL)
\]}%
where we bounded $N_{\chi}(1; z) \ll \cL$ using \cite[Lemma 2.2]{LMO}. 
By \cref{ZerosInCircle-Classical}, the above is therefore
{\small
\[
\leq (k+1)  \int_{A_1 r}^{\infty}  \frac{4u \cL +8}{u^{k+2}} du+O(\mathcal{L})  \leq \frac{ 4 \{ 1+ \tfrac{1}{k} \} A_1 r  \cL+ 8}{(A_1 r)^{k+1}} + O(\cL). 
\]}%
 By considering cases, one may bound the $\delta(\chi)$-term as follows:
{\small
\begin{equation}
r^{k+1} \cdot \Big| \frac{ \delta(\chi) }{(z-1)^{k+1}} \Big| \leq \delta(\chi) \cdot \mathbf{1}_{\{ |\tau| < Ar \}}(\tau) + \frac{1}{A_1^{k+1}}.
\label{BigDerivative_Principal}
\end{equation}}%
The above results now yield
{\small
\begin{equation}
\delta(\chi) \mathbf{1}_{\{ |\tau| < Ar \}}(\tau)  + \Big| \frac{r^{k+1}F^{(k)}(z)}{k!}    \Big| \geq \Big| \sum_{|1+i\tau - \rho| \leq Ar} \frac{r^{k+1}}{(z-\rho)^{k+1}} \Big| -  \Big[ \frac{ 4 \{ 1+ \tfrac{1}{k} \} A_1 r  \cL+ 9}{A_1^{k+1}}  + O\big( (4r)^{k+1} \cL\big) \Big]. 
\label{BigDerivative_UglyError}
\end{equation}}%
To lower bound the remaining sum over zeros, we wish to apply \cref{KS-PowerSum}. Denote
\[
N = N_{\chi}(Ar; 1+i\tau) = \#\{ \rho : L(\rho, \chi) = 0, |1+i\tau - \rho| \leq Ar\}.
\]
Since $\lambda < \tfrac{\epsilon}{A_1}\cL < \tfrac{\epsilon}{A}\cL$ and $\epsilon < \tfrac{1}{8}$, it follows by \cref{ZerosInCircle-Convexity} and \eqref{def:cL} that $N \leq 2\phi  A \lambda + 8.$ Define $M := \lfloor  \frac{2\phi A \lambda + 8}{\alpha} \rfloor$.  Thus, from \cref{KS-PowerSum} and assumption \eqref{DetectedZero},
{\small
\begin{equation}
\Big| \sum_{|1+i\tau-\rho| \leq Ar} \frac{1}{(z-\rho)^{k+1}}  \Big| \geq \Big( \frac{\alpha}{4e(1+\alpha)} \Big)^{2\phi A \lambda + 8} \frac{1}{(2r)^{k+1} } 
\label{BigDerivative_CloseZeros}
\end{equation}}%
for some $M + 1 \leq k \leq M+N$.  To simplify the RHS of \eqref{BigDerivative_UglyError}, observe that 
\begin{equation}
(4r)^{k+1} \cL \leq 4\lambda (4r)^k \ll \lambda (4\epsilon)^k A_1^{-k} \ll \epsilon \lambda A_1^{-k},
\label{BigDerivative_NeglibleError}
\end{equation}
since $r = \tfrac{\lambda}{\cL} < \tfrac{\epsilon}{A_1} < \tfrac{1}{4A_1}$ by assumption. Moreover, our choice of $A_1$ in \eqref{BigDerivative_A} implies
{\small
\begin{equation}
A_1^{-(k+1)}  = \Big( \frac{\alpha}{4e(1+\alpha)} \Big)^{\alpha k}   \frac{1}{2^{k}} \cdot \frac{1}{ A_1 (1+\eta)^{k}}  \leq \Big( \frac{\alpha}{4e(1+\alpha)} \Big)^{2\phi A \lambda + 8} \frac{1}{2^{k+1}} \cdot \frac{2}{A_1 (1+\eta)^{k} }
\label{eqn:kPower_A1}
\end{equation}}%
since $\alpha k \geq \alpha (M+1) \geq 2\phi A \lambda +8$. Incorporating \eqref{BigDerivative_CloseZeros}-\eqref{eqn:kPower_A1}  into \eqref{BigDerivative_UglyError} yields the desired result. The range of $k$ in \cref{BigDerivative} is determined by the above choice of $M$ and $N$. 
\end{proof}

\subsection{Short Sum over Prime Ideals} 
\label{subsec:ShortSum}
Continuing with the discussion and notation of \cref{subsec:LargeDerivative}, from the Euler product for $L(s,\chi^*)$, we have
\[
F(s) = \frac{L'}{L}(s,\chi^*)  = -\sum_{\mathfrak{n}} \chi^*(\mathfrak{n})  \Lambda_K(\mathfrak{n}) (\N\mathfrak{n})^{-s} 
\]
for $\Re\{s\} > 1$ and where $\Lambda_K(\, \cdot \,)$ is given by \eqref{def:vonMangoldt}. Differentiating the above formula $k$ times, we deduce
\begin{equation}
\frac{(-1)^{k+1} r^{k+1}}{k!} \cdot  F^{(k)}(z) =   \sum_{\mathfrak{n}} \frac{ \Lambda_K(\mathfrak{n}) \chi^*(\mathfrak{n})}{\N\mathfrak{n}^{1+i\tau}}  \cdot r E_k(r \log \N\mathfrak{n})
\label{F_kDerivative_Lseries}
\end{equation}
 for any integer $k \geq 1$, where $z = 1 + r + i\tau$ and $E_k(u) =  \frac{u^k e^{-u}}{k!}.$  From Stirling's bound (see \cite{DLMF}) in the form $k^k e^{-k} \sqrt{2\pi k} \leq k! \leq k^k e^{-k} \sqrt{2\pi k} e^{1/12k}$, one can verify
{\small
\begin{equation}
E_k(u) \leq 
\begin{cases}
A_1^{-k} & \text{if $u \leq \dfrac{k}{eA_1}$}, \\
A_1^{-k} e^{-\omega u} & \text{if $u \geq \tfrac{2}{1-\omega} \log\big( \tfrac{2A_1}{1-\omega} \big) k$},
\end{cases}
\label{E_k-Bounds}
\end{equation}}%
for any $k \geq 1$ and $A_1 > 1, \omega \in (0,1)$ defined in \cref{sec:LFZD_notation}. The goal of this subsection is to bound the infinite sum in \eqref{F_kDerivative_Lseries} by an integral average of short sums over prime ideals. 
\begin{lem} \label{ShortPrimeSum} Suppose the integer $k$ is in the range given in \cref{BigDerivative}. If $\lambda < \tfrac{\epsilon}{A_1} \cL$ then 
{\small
\begin{align*}
\Big| \sum_{\mathfrak{n}} \frac{\chi^*(\mathfrak{n}) \Lambda_K(\mathfrak{n})}{\N\mathfrak{n}^{1+i\tau}}  \cdot r E_k(r \log \N\mathfrak{n}) \Big| &\leq  r^2 \int_y^x \Big| \sum_{\substack{y \leq \N\mathfrak{p} < u}} \frac{\chi^*(\mathfrak{p}) \log \N \mathfrak{p}  }{\N\mathfrak{p}^{1+i\tau}}  \Big| \frac{du}{u} \\
&+  ( e[ Y + \tfrac{1}{2} + \{ 2X + 1\} e^{-\omega \lambda X} + O(\epsilon) ] \lambda + e^{1-\omega \lambda X}/\omega )  A_1^{-k}
\end{align*}}%
 where $x = e^{X \cL}$ and $y = e^{Y \cL}$ with $X = X_{\lambda}, Y = Y_{\lambda}$ defined by \eqref{X_Y_values}. 
 \end{lem}
\begin{proof}
First, divide the sum on the LHS into four sums:
{\small
\[
\sum_{\mathfrak{n}}  = \sum_{\N\mathfrak{p} < y} + \sum_{y \leq \N\mathfrak{p} < x} + \sum_{\N\mathfrak{p} \geq x} + \sum_{\mathfrak{n} \text{ not prime} }  = S_1 + S_2 + S_3 + S_4.
\]}%
Observe that \eqref{X_Y_values} and \eqref{E_k-Bounds}, along with the range of $k$ in \cref{BigDerivative} imply that 
{\small
\begin{equation}
E_k(r \log \N\mathfrak{n})  \leq
\begin{cases}
A_1^{-k} & \text{if $\N\mathfrak{n} \leq y$}, \\
A_1^{-k} (\N\mathfrak{n})^{-\omega r} & \text{if $\N\mathfrak{n} \geq x$}.
\end{cases}
\label{E_k-Bounds_Simple}
\end{equation}}%
Hence, for $S_1$, it follows by \cref{lem:PrimeSum} that
{\small
\[
|S_1| \leq  r A_1^{-k} \sum_{\substack{\N\mathfrak{p} < y}} \frac{\log \N\mathfrak{p} }{\N\mathfrak{p}}  \leq r A_1^{-k} \cdot  e \log(eD_K^{1/2} y) \leq e \Big( \lambda Y + \frac{\lambda}{2} + \epsilon\Big)  A_1^{-k}
\]}%
since $r = \tfrac{\lambda}{\cL} < \epsilon, \log D_K \leq \cL$, and $y = e^{Y \cL}$. 
Similarly, for $S_3$, apply partial summation using \cref{lem:PrimeSum} to deduce
{\small
\[
|S_3| \leq  rA_1^{-k} \sum_{\substack{\N\mathfrak{p} \geq x}} \frac{\log \N\mathfrak{p}}{(\N\mathfrak{p})^{1+ \omega r}}   \leq  rA_1^{-k} \int_x^{\infty} \frac{\omega r e \log(eD_K^{1/2} t) }{t^{1+\omega r}} dt  \leq ( \{X + \tfrac{1}{2} \} \lambda  + \omega^{-1} + \epsilon) \frac{e^{1- \omega \lambda X}}{A_1^{k}}. 
\]}%
For $S_4$, since $\sum_{k=0}^{\infty} E_k(u) = 1$, observe
\[
E_k(r \log \N\mathfrak{n}) = (2r)^k (\N\mathfrak{n})^{1/2-r} E_k(\tfrac{1}{2}\log \N\mathfrak{n}) \leq (2r)^{k} (\N\mathfrak{n})^{1/2-r}.
\]
Thus, by \cref{lem:PrimeSum},
{\small
\begin{align*}
|S_4|  \leq r \sum_{\mathfrak{p}} \sum_{m \geq 2} \frac{\log \N\mathfrak{p}}{(\N\mathfrak{p}^m)} E_k(r \log \N\mathfrak{p}^m) \leq (2r)^{k} r \sum_{\mathfrak{p}} \sum_{m \geq 2} \frac{\log \N\mathfrak{p}}{(\N\mathfrak{p}^m)^{1/2+r} }  & \ll (2r)^{k} r \sum_{\mathfrak{p}} \frac{\log \N\mathfrak{p}}{\N\mathfrak{p}^{1+2r} } \\
& \ll \lambda \epsilon A_1^{-k}
\end{align*}}%
since $\log D_K \leq  \cL$ and $\cL^{-1} \ll r = \frac{\lambda}{\cL} < \tfrac{\epsilon}{A_1}$. Also note that $\epsilon \in (0,\tfrac{1}{8})$ implies $(2\epsilon)^k \ll \epsilon$. Finally, for the main term $S_2$, define
{\small
\[
W(u) = W_{\chi}(u; \tau) := \sum_{y \leq \N\mathfrak{p} < u} \frac{\chi(\mathfrak{p}) \log \N \mathfrak{p}  }{\N\mathfrak{p}^{1+i\tau}},
\]}%
so by partial summation,
\begin{equation}
S_2    =  r W(x) E_k(r \log x) - r^2 \int_y^x  W(u) E_k'(r \log u) \frac{du}{u}
\label{ShortPrimeSum_S2}
\end{equation}
as $W(y) = 0$.  Similar to $S_1,S_3,$ and $S_4$, from \eqref{E_k-Bounds_Simple} and \cref{lem:PrimeSum} it follows
{\small
\[
|r W(x) E_k(r \log x)| \leq r A_1^{-k} x^{-\omega r}  \sum_{y \leq \N\mathfrak{p} < x} \frac{\Lambda_K(\mathfrak{n})}{\N\mathfrak{n}}   \leq e  \big(\{X + \tfrac{1}{2} \} \lambda  + \epsilon \big)  e^{- \omega \lambda X} A_1^{-k}.
\]}%
We have $|E_k'(u)| = |E_{k-1}(u)  - E_k(u)| \leq E_{k-1}(u) + E_k(u) \leq 1$ from definition of $E_k(u)$, so 
{\small
\[
|S_2| \leq r^2 \int_y^x |W(u)| \frac{du}{u} +  e \big(\{X + \tfrac{1}{2} \} \lambda + \epsilon\big)   e^{- \omega \lambda X} A_1^{-k}. 
\]}%
Collecting all of our estimates, we conclude the desired result as $\lambda \geq \lambda_0 \gg 1$. 
\end{proof}

\subsection{Proof of \cref{prop:ZeroDetector}} \label{proof:ZeroDetector} If
$\delta(\chi) \mathbf{1}_{\{|\tau| < Ar\}}(\tau) = 1$ then the inequality in  \cref{prop:ZeroDetector} holds trivially, as the RHS is certainly less than 1. Thus, we may assume otherwise.

Combining \cref{BigDerivative,ShortPrimeSum} via \eqref{F_kDerivative_Lseries}, it follows that 
{\small
\begin{equation}
\begin{aligned}
& r^2 \int_y^x \Big| \sum_{\substack{y \leq \N\mathfrak{p} < u}} \frac{\chi^*(\mathfrak{p}) \log \N \mathfrak{p}  }{\N\mathfrak{p}^{1+i\tau}}  \Big| \frac{du}{u}   \geq \Big( \frac{\alpha}{4e(1+\alpha)} \Big)^{2\phi A \lambda + 8} \cdot  \frac{1}{2^{k+1}} \big\{ 1 - J(\lambda) \big\},
\end{aligned} 
\label{ZeroDetector-1}
\end{equation}}%
after bounding $A_1^{-k}$ as in \eqref{eqn:kPower_A1} and noting $k \geq k_0$ in the range of \cref{BigDerivative}.  By assumption,  $J(\lambda) < 1$ and hence the RHS of \eqref{ZeroDetector-1} is positive. Therefore, squaring both sides and applying Cauchy-Schwarz to the LHS gives
{\small
\[
r^4 \log(x/y) \int_y^x \Big| \sum_{\substack{y \leq \N\mathfrak{p} < u}} \frac{\chi^*(\mathfrak{p}) \log \N \mathfrak{p}  }{\N\mathfrak{p}^{1+i\tau}}  \Big|^2 \frac{du}{u}   \geq \Big( \frac{\alpha}{4e(1+\alpha)} \Big)^{4\phi A \lambda + 16} \cdot  \frac{1}{2^{2k+2}} \big\{ 1 - J(\lambda) \big\}^2.
\]}%
By assumption, $y = e^{Y \cL} > e^{2\cL} \geq \N\kf_{\chi}$, so it follows $\chi^*(\mathfrak{p}) = \chi(\mathfrak{p})$ for $y \leq \N\mathfrak{p} < x$ so we may replace $\chi^*$ with $\chi$ in the above sum over prime ideals. Finally, we note $k \leq \frac{1+\alpha}{\alpha}(2\phi A \lambda + 8)$ since $k$ is in the range of \cref{BigDerivative}, yielding the desired result. 
\hfill \qed

\subsection{Proof of \cref{thm:LFZD_general}} \label{proof:LFZD_general}
For $\chi \pmod{H}$, consider zeros $\rho = \beta+i\gamma$ of $L(s,\chi)$ such that
\begin{equation}
1-\lambda/\cL \leq \beta < 1 \qquad |\gamma| \leq T.
\label{eqn:ZeroRegion}
\end{equation}
Denote $\lambda^{\star} = \xi \lambda$ and $r^{\star} = \lambda^{\star}/\cL = \xi(1-\sigma)$, 
so by \eqref{eqn:LFZD_general_assumptions} we have $r^{\star} < \tfrac{\epsilon}{A_1}$.  For any zero $\rho = \beta+i\gamma$ of $L(s,\chi)$, define $\Phi_{\rho, \chi}(\tau) := \mathbf{1}_{\{ |1+i\tau-\rho| \leq r^{\star} \}}(\tau)$.  If $\rho$ satisfies \eqref{eqn:ZeroRegion} then one can verify by elementary arguments that
{\small
\[
\frac{1}{r^{\star}} \int_{-T}^T \Phi_{\rho, \chi}(\tau) d\tau \geq \frac{\sqrt{\xi^2-1}}{\xi}.
\]}%
Applying \cref{prop:ZeroDetector} to such zeros $\rho$, it follows that
{\small
\begin{equation*}
	\begin{aligned}
	&  \int_{-T}^T \frac{1}{r^{\star}} \Phi_{\rho,\chi}(\tau) \Big[ (r^{\star})^4 \log(x/y) \int_y^x \Big| \sum_{y \leq \N\mathfrak{p} < u} \frac{\chi(\mathfrak{p}) \log \N \mathfrak{p}  }{\N\mathfrak{p}^{1+i\tau}}  \Big|^2 \frac{du}{u}  + \delta(\chi) \mathbf{1}_{\{ |\tau| < A r^{\star} \}}(\tau) \Big] d\tau  \\
	& \qquad
	\geq \frac{\sqrt{\xi^2-1}}{4\xi}  \Big( \frac{\alpha}{4e(1+\alpha)2^{(1+\alpha)/\alpha} } \Big)^{2\phi  A \xi \lambda + 16} \times \big\{ 1 -  J(\xi \lambda)  \big\}  =: w(\lambda),
	\end{aligned}
\end{equation*}}%
say. Note $x = e^{X \cL}$ and $y = e^{Y \cL}$ where $X = X_{\lambda^{\star}}$ and $Y = Y_{\lambda^{\star}}$. Summing over all zeros $\rho$ of $L(s,\chi)$ satisfying \eqref{eqn:ZeroRegion}, we have that
{\small
\begin{equation}
	\label{eqn:ChiZeros}
	\begin{aligned}
	w(\lambda)  N(\sigma, T,\chi) 
	& \leq (X-Y) (2\phi r^{\star}\cL + 8 )  (r^{\star})^3 \cL  \int_y^x \Big( \int_{-T}^T  \Big| \sum_{y \leq \N\mathfrak{p} < u} \frac{\chi(\mathfrak{p}) \log \N \mathfrak{p}  }{\N\mathfrak{p}^{1+i\tau}}  \Big|^2 d\tau  \Big) \frac{du}{u}   \\
	& \qquad\qquad +  \delta(\chi) (4 \phi A r^{\star}\cL+ 16A )  
	\end{aligned}
\end{equation}}%
since, for $|\tau| \leq T$ and $r^{\star} < \epsilon$, 
{\small
\[
\sum_{\substack{ \rho \\ L(\rho, \chi) = 0} } \Phi_{\rho,\chi}(\tau) = N_{\chi}(r^{\star}; 1+i\tau) \leq 2\phi r^{\star}\cL + 8
\]}%
by \cref{ZerosInCircle-Convexity}. From the conditions on $Y$ and $T$ in \eqref{eqn:LFZD_general_assumptions} and the definition of $\cL$ in \eqref{def:cL}, observe that, for $\nu = \nu(\epsilon) > 0$ sufficiently small, \cref{lem:h_H-Bound} implies
\[
y = e^{Y \cL} \geq C_{\nu} \{ h_H n_K^{(5/4 + 2\upsilon) n_K} D_K^{3/2+2\upsilon} Q^{1/2} T^{n_K/2 + 1} \} ^{1+\nu} 
\]
since $\upsilon \leq \tfrac{1}{10}$ and $\Theta = \Theta(\epsilon) \geq 1$ is sufficiently large. Therefore, we may sum \eqref{eqn:ChiZeros} over $\chi \pmod{H}$ and apply \cref{thm:LargeSieve} to deduce
{\small
	\begin{align}
	w(\lambda) \sum_{\substack{ \chi \pmod{H} }} N(\sigma, T,\chi)  &\leq   \Big( C' (2 \phi r^{\star}\cL + 8 )(r^{\star})^3  + O_{\epsilon}\big( \frac{(r^{\star})^4 \cL^2}{e^{\epsilon Y\cL/2} }\big)  \Big) \int_y^x  \sum_{y \leq \N\mathfrak{p} < u} \frac{ (\log \N\mathfrak{p})^2}{\N\mathfrak{p}} \frac{du}{u}\notag  \\
	& +  4A\phi r^{\star} \cL+ 16A   
\label{LFZD-Penultimate}
\end{align}}%
where $C'=5\pi(X-Y)(1-\frac{1}{1+\upsilon})^{-1}(\frac{1}{1+\epsilon}Y-4)^{-1}$.  To calculate $C'$, we replaced $\cL'$ (found in \cref{thm:LargeSieve}) by observing from \cref{lem:h_H-Bound} that $\cL' + \tfrac{1}{1+\epsilon}\log h_H \leq 4 \cL$ since $T \geq \max\{ n_K^{5/6} D_K^{-4/3n_K} Q^{-4/9n_K}, 1\}$ and $\Theta = \Theta(\epsilon)$ is sufficiently large. For the remaining integral in \eqref{LFZD-Penultimate}, notice by \cref{lem:PrimeSum}, 
{\small
\begin{equation*}
\begin{aligned}
\int_y^x  \sum_{y \leq \N\mathfrak{p} < u} \frac{ (\log \N\mathfrak{p})^2}{\N\mathfrak{p}} \frac{du}{u}  \leq \log x \int_y^x e \log(e D_K^{1/2} u) \frac{du}{u}  \leq \tfrac{e}{2} X(X-Y)(X+Y + 1 + \tfrac{2}{\cL})\cL^3.
\end{aligned}
\end{equation*}}%
Substituting this estimate in \eqref{LFZD-Penultimate} and recalling $r^{\star} = {\lambda}^{\star}/\cL = \xi\lambda/\cL$, we have shown 
{\small
\[
w(\lambda) \sum_{\chi \pmod{H}} N(\sigma, T,\chi) \leq   2\phi C''\xi^4 \cdot \lambda^4+8C''\xi^3 \cdot \lambda^3   + 4\phi A\xi \cdot \lambda + 16A + O_{\epsilon}( \lambda^3 \cL e^{-\epsilon\cL}) 
\]}%
where $C'' = \tfrac{e}{2} X (X-Y)(X+Y + 1 + \tfrac{2}{\cL}) C'$. Since $\cL \geq \Theta$ and $\Theta$ is sufficiently large depending on $\epsilon$, the big-O error term above and the quantity $\tfrac{2}{\cL}$ in $C''$ may both be bounded by $\epsilon$. This completes the proof of \cref{thm:LFZD_general}. 

\section{Log-Free Zero Density Estimate}
\label{sec:LFZD}
Having established \cref{thm:LFZD_general}, this section is dedicated to the proof of \cref{LFZD-MainTheorem}.  

\subsection*{Proof of \cref{LFZD-MainTheorem}:}
Without loss, we may assume $H \pmod{\kq}$ is primitive because $Q= Q_H = Q_{H'}, h_H = h_{H'}$ and
{\small
\[
\sum_{\chi \pmod{H}} N(\sigma,T,\chi) = \sum_{ \chi  \pmod{H'}} N(\sigma,T,\chi)
\]}%
if $H'$ induces $H$. 
Suppose $\frac{1}{2} \leq \sigma \leq 1 -  \frac{0.05}{4}.$ By a naive application of \cite[Lemma 2.1]{LMO}, one can verify that for $T \geq 1$,
{\small
\begin{equation}
\sum_{\chi \pmod{H}} N(\sigma, T,\chi) \ll h_H T \log( D_K Q T^{n_K}) \ll ( e^{O(n_K)} D_K^2QT^{n_K+2})^{81(1-\sigma)} 
\label{eqn:ThickStrip}
\end{equation}}%
after bounding $h_H$ with \cref{lem:h_H-Bound}.

Now, let $\epsilon \in (0,1/8)$ be fixed and define $\cL$ as in \eqref{def:cL}. Suppose $1- \frac{\epsilon}{4} < \sigma < 1$.  Let $R \geq 1$ be fixed and sufficiently large. By applying the bound in \cref{lem:h_H-Bound} to \cite[Theorem 4.3]{Weiss}, we deduce that for $T \geq 1$,
{\small
\begin{equation}
\sum_{\chi \pmod{H}} N(1-\tfrac{R}{\cL}, T,\chi) \ll 1,
\label{eqn:ThinStrip}
\end{equation}}%
so it suffices to bound $\sum_{\chi(H)=1} N(\sigma,T,\chi)$ in the range 
{\small
\begin{equation}
1- \frac{\epsilon}{4} < \sigma < 1 - \frac{R}{\cL}
\label{eqn:SigmaCondition}
\end{equation}}%
or equivalently, if $\sigma = 1- \frac{\lambda}{\cL}$, in the range $R < \lambda < \frac{\epsilon}{4} \cL$.  According to \cref{thm:LFZD_general} and the notation defined in \cref{sec:LFZD_notation}, select
\[
\xi = 1+ 10^{-5}, \qquad \upsilon = 10^{-5}, \qquad \eta = 10^{-5}, \qquad \omega = 10^{-5}, \quad \text{and} \quad \alpha = 0.15. 
\]
It follows that the constants $B_2, C_0, C_1, C_3, C_4$ in \cref{thm:LFZD_general} are bounded absolutely, 
\[
X > Y > 4.6, \quad B_1 \leq 146.15 \phi, \quad \text{and} \quad \xi A_1 < 4
\]
where $\phi = 1 + \tfrac{4}{\pi} \epsilon + 16 \epsilon^2 + 340\epsilon^{10}$. Moreover, since $\lambda > R$,  $J(\xi \lambda) \ll  \frac{\lambda}{(1+10^{-5})^{\lambda} } \ll \frac{R}{(1+10^{-5})^{R}  }$ and therefore $J(\xi \lambda) < \tfrac{1}{2}$ for $R$ sufficiently large. Thus, by \cref{thm:LFZD_general}, 
{\small
\begin{equation}
\sum_{\chi \pmod{H}} N(\sigma, T,\chi) \ll \lambda^4 e^{ 146.15 \phi \lambda} \ll e^{ 146.2 \phi \lambda} = e^{146.2 \phi (1-\sigma) \cL}
\label{LFZD_FinalStep}
\end{equation}}%
for $\sigma$ satisfying \eqref{eqn:SigmaCondition} and $T \geq \max\{n_K^{5/6} D_K^{-4/3n_K} Q^{-4/9n_K},1\}$. To complete the proof of \cref{LFZD-MainTheorem}, it remains to choose $\epsilon$ in \eqref{LFZD_FinalStep}. If $\epsilon = 0.05$ then $146.2 \phi < 162 = 2\cdot 81$ yielding the desired result when combined with \eqref{eqn:ThickStrip}. If $\epsilon = 10^{-3}$ then $146.2\phi < 147 = 2 \cdot 73.5$ as claimed. \qed

\section{Zero Repulsion: The Deuring-Heilbronn Phenomenon}
\label{sec:DH_proof}

In this section, we prove \cref{DH-MainTheorem} and establish Deuring-Heilbronn phenomenon for $L$-functions of Hecke characters $\chi \pmod{H}$  where $H \pmod{\kq}$ is a (not necessarily primitive) congruence class group. We will critically use the following power sum inequality. 

\begin{thm}[Lagarias-Montgomery-Odlyzko] \label{LMO-PowerSum} Let $\epsilon > 0$ and a sequence of complex numbers $\{z_n\}_n$ be given. Suppose that $|z_n| \leq |z_1|$ for all $n \geq 1$. Define $M := \frac{1}{|z_1|}\sum_{n} |z_n|$.  Then there exists $m_0$ with $1 \leq m_0 \leq (12+\epsilon) M$ such that $\Re\{ \sum_{n=1}^{\infty} z_n^{m_0}\} \geq \frac{\epsilon}{48+5\epsilon} |z_1|^{m_0}$.
\end{thm}
\begin{proof} This is a modified version of \cite[Theorem 4.2]{LMO}; see \cite[Theorem 2.3]{Zaman_2015c} for details. 
\end{proof}
We prepare for the application of this result by establishing a few preliminary estimates and then end this section with the proof of \cref{DH-MainTheorem}.

\subsection{Preliminaries} 
 
 \begin{lem} \label{lem:ReduceToPrimitive}
 Let $\chi \pmod{\kq}$ be a Hecke character. For $\sigma \geq 2$ and $t \in \R$, 
{\small
 \[
 -\Re\Big\{ \frac{L'}{L}(\sigma+it, \chi) \Big\} \leq  - \Re\Big\{ \frac{L'}{L}(\sigma+it, \chi^*) \Big\} + \frac{1}{2^{\sigma}-1} \big( n_K + \log \N\kq\big).
 \]}%
 where $\chi^*$ is the primitive character inducing $\chi$. 
 \end{lem}
 \begin{proof}
By definition,
{\small
\[
L(s,\chi) = P(s, \chi) L(s,\chi^*) \quad \text{where } \quad P(s, \chi) = \prod_{\substack{ \mathfrak{p} \mid \kq \\ \mathfrak{p} \nmid \kf_{\chi}} } \Big(1 - \frac{\chi^*(\mathfrak{p})}{\N\mathfrak{p}^{s}} \Big)
\]}%
so it suffices to show $| \frac{P'}{P}(s,\chi)| \leq \frac{1}{2^{\sigma}-1} ( n_K + \log \N\kq)$.  Observe, by elementary arguments, 
{\small
\[
\Big|\frac{P'}{P}(s,\chi)\Big| 
	 = \Big| \sum_{\substack{\mathfrak{p} \mid \kq \\ \mathfrak{p} \nmid \kf_{\chi}}} \sum_{k=1}^{\infty} \frac{ \chi^*(\mathfrak{p}^k) \log \N\mathfrak{p}^k}{k (\N\mathfrak{p}^k)^s} \Big| 
	 = \sum_{\mathfrak{p} \mid \kq} \frac{\log \N\mathfrak{p}}{\N\mathfrak{p}^{\sigma}-1}  \leq \frac{1}{1-2^{-\sigma}} \cdot \frac{1}{2^{\sigma-1}} \sum_{\mathfrak{p} \mid \kq} \frac{\log \N\mathfrak{p}}{\N\mathfrak{p}}.
\]}%
From \cite[Lemma 2.4]{Zaman_2015a}, $\sum_{\mathfrak{p} \mid \kq} \frac{\log \N\mathfrak{p}}{\N\mathfrak{p}} \leq \sqrt{n_K \log \N\kq} \leq \frac{n_K}{2} + \frac{\log \N\kq}{2}$.  Combining this fact with the previous inequality gives the desired estimate. 
\end{proof}

  \begin{lem} \label{DH-TrivialZeroSum} Let $\chi \pmod{\kq}$ be a Hecke character. For $\sigma > 1$ and $t \in \R$, 
{\small
\begin{align*}
& \sum_{\omega \, \mathrm{ trivial}} \frac{1}{|\sigma+it-\omega|^2}  \leq  \begin{cases} \big( \frac{1}{2\sigma} + \frac{1}{\sigma^2} \big)  \cdot n_K & \text{if $\chi$ is primitive,} \\
\big( \frac{1}{2\sigma} + \frac{1}{\sigma^2} \big) \cdot n_K + \big(  \frac{1}{2 \sigma} + \frac{2}{\sigma^2 \log 2}  \big) \cdot  \log \N\kq   & \text{unconditionally}, 
\end{cases}
\end{align*}}%
where the sum is over all trivial zeros $\omega$ of $L(s,\chi)$ counted with multiplicity.
\end{lem}
\begin{proof} Suppose $\chi \pmod{\kq}$ is induced by the primitive character $\chi^* \pmod{\kf_{\chi}}$. Then 
{\small
\[
L(s,\chi) = P(s, \chi) L(s,\chi^*) \quad \text{where } \quad P(s, \chi) = \prod_{\substack{ \mathfrak{p} \mid \kq,~\mathfrak{p} \nmid \kf_{\chi}} } \Big(1 - \frac{\chi^*(\mathfrak{p})}{\N\mathfrak{p}^{s}} \Big)
\]}%
for all $s \in \mathbb{C}$. Thus, the trivial zeros of $L(s,\chi)$ are either zeros of the finite Euler product $P(s,\chi)$ or trivial zeros of $L(s,\chi^*)$. We consider each separately. From \eqref{TrivialZeros} and \eqref{GammaFactor_Exponents}, observe
{\small
\begin{align*}
 \sum_{\substack{ \omega \, \mathrm{ trivial} \\ L(\omega, \chi^*) = 0} } \frac{1}{|\sigma+it-\omega|^2} 
 & \leq a(\chi) \sum_{k=0}^{\infty}  \frac{1}{(\sigma+2k)^2 + t^2}  + b(\chi) \sum_{k=0}^{\infty} \frac{1}{(\sigma+2k+1)^2 +t^2} \\
 & \leq n_K \sum_{k=0}^{\infty}  \frac{1}{(\sigma+2k)^2} \leq \Big( \frac{1}{2\sigma} + \frac{1}{\sigma^2} \Big) n_K
\end{align*}}%
Now, if $\chi$ is primitive then $P(s,\chi) \equiv 1$ and hence never vanishes. Otherwise, notice the zeros of each $\mathfrak{p}$-factor in the Euler product of $P(s,\chi)$ are totally imaginary and are given by $a_{\chi}(\mathfrak{p}) i + \frac{2\pi i \Z}{\log \N\mathfrak{p}}$ for some $0 \leq a_{\chi}(\mathfrak{p})  < 2\pi/\log \N\mathfrak{p}$.  Translating these zeros $\omega \mapsto \omega + it$ amounts to choosing another representative $0 \leq b_{\chi}(\mathfrak{p}; t) < 2\pi /\log\N\mathfrak{p}$. Therefore, 
{\small
\[
 \sum_{\substack{ \omega \, \mathrm{ trivial} \\ P(\omega, \chi) = 0} } \frac{1}{|\sigma+it-\omega|^2} 
 \leq   2 \sum_{ \substack{ \mathfrak{p} \mid \kq \\ \mathfrak{p} \nmid \kf_{\chi} } } \sum_{k=0}^{\infty}  \frac{1}{\sigma^2 + (2\pi k/\log \N\mathfrak{p})^2}  \leq   \Big(\frac{1}{2 \sigma} +  \frac{2}{\sigma^2 \log 2}  \Big) \log \N\kq,
\]}%
as required.  
\end{proof}

 \begin{lem} \label{DH-ZeroSum} Let $H \pmod{\kq}$ be a congruence class group of $K$. Suppose $\psi \pmod{H}$ is real and $\chi \pmod{H}$ is arbitrary. For $\sigma = \alpha+1$ with $\alpha \geq 1$ and $t \in \R$, 
{\small
\begin{align*}
& \sum_{\substack{\rho \\  \zeta_K(\rho) = 0}} \frac{1}{|\sigma-\rho|^2}  + \sum_{\substack{\rho \\  L(\rho, \psi) = 0}}  \frac{1}{|\sigma -\rho|^2}    + \sum_{\substack{\rho \\  L(\rho, \chi) = 0}}  \frac{1}{|\sigma + it -\rho|^2}  + \sum_{\substack{\rho \\  L(\rho, \psi\chi) = 0}}  \frac{1}{|\sigma +it -\rho|^2} \\
& \qquad \leq \frac{1}{\alpha} \cdot \Big[ \frac{1}{2}\log(D_K^3 Q^2 D_{\psi} )  +\Big( \log(\alpha+2) + \frac{2}{\alpha+1}  + \frac{1}{2^{\alpha+1}-1} - 2\log \pi \Big) n_K \\
& \qquad\qquad\qquad +  n_K \log( \alpha+2+|t|) + \frac{2}{2^{\alpha+1}-1} \log Q + \frac{4}{\alpha} + \frac{4}{\alpha+1} \Big],  
\end{align*}}%
where the sums are over all non-trivial zeros  of the corresponding $L$-functions. 
\end{lem}
\begin{remark}
If $\psi$ is trivial, notice that the LHS equals 
{\small
\[
2  \Big( \sum_{\substack{\rho \\  \zeta_K(\rho) = 0}} \frac{1}{|\sigma-\rho|^2} +  \sum_{\substack{\rho \\  L(\rho, \chi) = 0}}  \frac{1}{|\sigma + it -\rho|^2} \Big). 
\]}%
This additional factor of $2$ will be useful to us later. 
\end{remark}

\begin{proof}
Suppose $\psi$ and $\chi$ are induced from the primitive characters $\psi^*$ and $\chi^*$ respectively. From the identity $0 \leq ( 1 + \psi^*(\mathfrak{n}) )(1 + \Re\{ \chi^*(\mathfrak{n}) (\N\mathfrak{n})^{-it} \} )$, it follows that 
{\small
\begin{equation}
0 \leq - \Re\Big\{ \frac{\zeta_K'}{\zeta_K}(\sigma) +  \frac{L'}{L}(\sigma, \psi^*) + \frac{L'}{L}(\sigma+it, \chi^*) + \frac{L'}{L}(\sigma+it, \psi^*\chi^*)  \Big\}.
\label{eqn:DH-ZeroSum_TrigId}	
\end{equation}}%
The first three $L$-functions are primitive, but  $\xi := \psi^*\chi^*$ is a character modulo $[\kf_{\chi},\kf_{\psi}]$, the least common multiple of $\kf_{\psi}$ and $\kf_{\chi}$, and hence is not necessarily primitive. Hence, by \cref{lem:ReduceToPrimitive}, we deduce
{\small
\begin{align*}
	0 & \leq - \Re\Big\{ \frac{\zeta_K'}{\zeta_K}(\sigma) +  \frac{L'}{L}(\sigma, \psi^*) + \frac{L'}{L}(\sigma+it, \chi^*) + \frac{L'}{L}(\sigma+it, \xi^*)  \Big\} + \frac{n_K + \log \N[\kf_{\chi},\kf_{\psi}]}{2^{\sigma}-1}.
\end{align*}}%
Note $\N[\kf_{\chi},\kf_{\psi}] \leq Q^2$ since  $\psi$ and $\chi$ are both characters trivial on the congruence subgroup $H$ and therefore the norms of their respective conductors are bounded by $Q$. Inputting this bound, we apply \cref{ExplicitFormula,digamma} to each of the primitive $L$-functions term yielding
{\small
\begin{equation}
\begin{aligned}
0 & \leq \tfrac{1}{2}\log(D_K D_{\psi} D_{\chi} D_{\xi}) +  \frac{2}{2^{\sigma}-1} \log Q + n_K \log(\sigma+1+|t|)  + A_{\sigma} n_K\\
&  \quad - \Re\Big\{ \sum_{\substack{\rho \\  \zeta_K(\rho) = 0}} \frac{1}{\sigma-\rho}  + \sum_{\substack{\rho \\  L(\rho, \psi) = 0}}  \frac{1}{\sigma -\rho}    + \sum_{\substack{\rho \\  L(\rho, \chi) = 0}}  \frac{1}{\sigma + it -\rho}  + \sum_{\substack{\rho \\  L(\rho, \psi\chi) = 0}}  \frac{1}{\sigma +it -\rho} \Big\} \\
& \quad + \frac{1+\delta(\psi)}{\alpha} + \frac{1+\delta(\psi)}{\alpha+1} + \Re\Big\{ \frac{\delta(\chi) + \delta(\chi \psi) }{\alpha+it} + \frac{\delta(\chi) + \delta(\chi \psi)}{\alpha+1+it} \Big\}  \\
\end{aligned}
\label{BoundZeroSum}
\end{equation}}%
where $A_{\sigma} = \log(\sigma+1) + \tfrac{2}{\sigma}  + \tfrac{1}{2^{\sigma}-1} - 2\log \pi$. Since $0 < \beta < 1$, we notice $\Re\{ \frac{1}{\sigma+it-\rho}\} \geq \frac{\alpha}{|\sigma+it-\rho|^2}$ and $\Re\{ \frac{1}{\alpha+it} + \frac{1}{\alpha+1+it}\}  \leq  \frac{1}{\alpha} + \frac{1}{\alpha+1}$.  Further, $D_{\chi}$ and $D_{\xi}$ are both $\leq D_K Q$ as $\xi = \psi^* \chi^*$ induces the character $\psi \chi \pmod{\kq}$ which is trivial on $H$. 
Rearranging \eqref{BoundZeroSum} and employing all of the subsequent observations gives the desired conclusion. 
\end{proof}

\subsection{Proof of \cref{DH-MainTheorem}}
If $\tilde{H} \pmod{\km}$ induces $H \pmod{\kq}$, then a character $\chi \pmod{H}$ is induced by a character $\tilde{\chi} \pmod{\tilde{H}}$. It follows that
{\small
\[
L(s,\chi) = L(s,\tilde{\chi}) \prod_{\substack{\mathfrak{p} \mid \kq \\ \mathfrak{p} \nmid \km }} \Big(1 - \frac{\tilde{\chi}(\mathfrak{p})}{\N\mathfrak{p}^s}\Big)
\]}%
for all $s \in \mathbb{C}$. This implies that the non-trivial zeros of $L(s,\chi)$ are the same non-trivial zeros of $L(s,\tilde{\chi})$. Therefore,  without loss of generality, we may assume $H \pmod{\kq}$ is primitive. 

We divide the proof according to whether $\psi$ is quadratic or trivial. The arguments in each case are similar but require some minor differences. 

\subsubsection{$\psi$ is quadratic.}
\label{DH-MainTheorem_Quadratic}
Let $m$ be a positive integer, $\alpha \geq 1$ and $\sigma = \alpha+1$.  From the identity $0 \leq (1+\psi^*(\mathfrak{n}) )(1+ \Re\{ \chi^*(\mathfrak{n}) (\N\mathfrak{n})^{-i\gamma'} \} )$ and \cref{ExplicitFormula_HigherDerivatives} with $s = \sigma+i\gamma'$, it follows 
{\small
\begin{equation}
\Re\Big\{ \sum_{n=1}^{\infty}  z_n^m \Big\} \leq \frac{1}{\alpha^m} - \frac{1}{(\alpha+1-\beta_1)^{2m}} + \Re\Big\{ \frac{\delta(\chi)+\delta(\psi \chi)}{(\alpha+i\gamma')^{2m}} - \frac{\delta(\chi)+\delta(\psi\chi)}{(\alpha +1+i\gamma' -\beta_1)^{2m}} \Big\} 
\label{DH-Quadratic_PowerSumPrep}
\end{equation}}%
where $z_n = z_n(\gamma')$ satisfies $|z_1| \geq |z_2| \geq \dots$ and runs over the multisets
{\small
\begin{equation}
\begin{aligned}
& \{  (\sigma-\omega)^{-2} :  \omega \text{ is any zero of $\zeta_K(s)$} \}, \\
& \{  (\sigma-\omega)^{-2}  :  \omega \neq \beta_1 \text{ is any zero of $L(s, \psi^*)$} \}, \\
& \{  (\sigma+i\gamma'-\omega)^{-2}  :  \omega \neq \beta_1 \text{ is any zero of $L(s, \chi^*)$} \}, \\
& \{  (\sigma+i\gamma'-\omega)^{-2}  :  \omega \neq \beta_1 \text{ is any zero of $L(s, \psi^*\chi^*)$} \}. \\
\end{aligned}
\label{DH-Quadratic_z_n}
\end{equation}}%
Note that the multisets includes trivial zeros of the corresponding $L$-functions and $\psi^*\chi^*$ is a Hecke character (not necessarily primitive) modulo the least common multiple of $\kf_{\chi}$ and $\kf_{\psi}$. With this choice, it follows
\begin{equation}
(\alpha+1/2)^{-2} \leq (\alpha+1-\beta')^{-2} \leq |z_1| \leq \alpha^{-2}. 
\label{DH-zProperty}
\end{equation}
The RHS of  \eqref{DH-Quadratic_PowerSumPrep} may be bounded via the observation
{\small
\[
\Big| \frac{1}{(\alpha+it)^{2m}} - \frac{1}{(\alpha + it + 1-\beta_1)^{2m}} \Big| \leq \alpha^{-2m} \Big| 1 - \frac{1}{(1 + \frac{1-\beta_1}{\alpha+it})^{2m} } \Big|   \ll \alpha^{-2m-1} m(1-\beta_1),
\]}%
whence
{\small
\begin{equation}
\Re\Big\{ \sum_{n=1}^{\infty}  z_n^m \Big\} \ll \alpha^{-2m-1} m(1-\beta_1). 
\label{DH-PowerSum_RHS}
\end{equation}}%
On the other hand, by \cref{LMO-PowerSum}, for $\epsilon > 0$, there exists some $m_0 = m_0(\epsilon)$ with $1 \leq m_0 \leq (12+\epsilon)M$ such that
{\small
\[
\Re\Big\{ \sum_{n=1}^{\infty} z_n^{m_0} \Big\} \geq \tfrac{\epsilon}{50}  |z_1|^{m_0} \geq \tfrac{\epsilon}{50} (\alpha+1-\beta')^{-2m_0} \geq \tfrac{\epsilon}{50} \alpha^{-2m_0} \exp(-\tfrac{2m_0}{\alpha}(1-\beta') ), 
\]}%
where $M= |z_1|^{-1} \sum_{n=1}^{\infty} |z_n|$. Comparing with \eqref{DH-PowerSum_RHS} for $m = m_0$, it follows that
\begin{equation}
\exp(-(24+2\epsilon)\tfrac{M}{\alpha} (1-\beta') ) \ll_{\epsilon} \tfrac{M}{\alpha} (1-\beta_1). 
\label{DH-Penultimate}
\end{equation}
Therefore, it  suffices to bound $M/\alpha$ and optimize over $\alpha \geq 1$.

By \eqref{DH-Quadratic_z_n}, the quantity $M$ is a sum involving non-trivial and trivial zeros of certain $L$-functions. For the non-trivial zeros, we employ \cref{DH-ZeroSum} with $D_{\psi} = D_K \N\kf_{\psi} \leq D_K Q$ since $\psi$ is quadratic. For the trivial zeros, apply \cref{DH-TrivialZeroSum}   in the ``primitive" case for $\zeta_K(s), L(s,\psi^*), L(s,\chi^*)$ and  in the ``unconditional" case for $L(s,\psi^*\chi^*)$. In the latter case, we additionally observe that, as $H \pmod{\kq}$ is primitive, $\log \N\kq \leq 2 \log Q$ by \cref{lem:MaxConductor}. Combining these steps along with \eqref{DH-zProperty}, it follows that
{\small
\begin{equation}
\label{DH-M_Quadratic}
\begin{aligned}
\frac{M}{\alpha}  & \leq \frac{(\alpha+1/2)^{2}}{\alpha^2}   \cdot \Big[ 2 \log D_K + \Big( \frac{3}{2} + \frac{2\alpha}{2\alpha+2} + \frac{4 \alpha}{(\alpha+1)^2 \log 2}  + \frac{2}{2^{\alpha+1}-1} \Big)\log Q \\
& \qquad\qquad +\Big( \log(\alpha+2) + \log(\alpha+3)  + 2 - 2\log \pi + \frac{4\alpha}{(\alpha+1)^2} + \frac{1}{2^{\alpha+1}-1} \Big) n_K \\
& \qquad\qquad+  n_K \log T + \frac{4}{\alpha} + \frac{4}{\alpha+1} \Big],  
\end{aligned}
\end{equation}}%
for $\alpha \geq 1$. Note, in applying \cref{DH-ZeroSum}, we used that $\log(\alpha + 2 + T) \leq \log(\alpha+3) + \log T$ for $T \geq 1$. Finally, select $\alpha$ sufficiently large, depending on $\epsilon > 0$, so the RHS of \eqref{DH-M_Quadratic} is
		\[
		\leq (2+ \tfrac{\epsilon}{100} ) \log D_K + (2.5 + \tfrac{\epsilon}{100})\log Q + (1 + \tfrac{\epsilon}{100}) n_K \log T + O_{\epsilon}(n_K). 
		\]
Inputting the resulting bounds in \eqref{DH-Penultimate} completes the proof of \cref{DH-MainTheorem} for $\psi$ quadratic. 

\subsubsection{$\psi$ is trivial.}
Begin with the identity $0 \leq 1 + \Re\{\chi^*(\mathfrak{n}) (\N\mathfrak{n})^{-i\gamma'} \}$.  This similarly implies
{\small
\begin{equation}
\Re\Big\{ \sum_{n=1}^{\infty}  z_n^m \Big\} \leq \frac{1}{\alpha^m} - \frac{1}{(\alpha+1-\beta_1)^{2m}} + \Re\Big\{ \frac{\delta(\chi)}{(\alpha+i\gamma')^{2m}} - \frac{\delta(\chi)}{(\alpha +1+i\gamma' -\beta_1)^{2m}} \Big\} 
\label{DH-AllZeros_PowerSumPrep}
\end{equation}}%
for a new choice $z_n = z_n(\gamma')$ satisfying $|z_1| \geq |z_2| \geq \dots$ and which runs over the multisets
{\small
\begin{equation}
\begin{aligned}
& \{  (\sigma-\omega)^{-2} :  \omega \neq \beta_1 \text{ is any zero of $\zeta_K(s)$} \}, \\
& \{  (\sigma+i\gamma'-\omega)^{-2}  :  \omega \neq \beta_1 \text{ is any zero of $L(s, \chi^*)$} \}. \\
\end{aligned}
\label{DH-AllZeros_z_n_Principal}
\end{equation}}%
Following the same arguments as before, we may arrive at \eqref{DH-Penultimate} for the new quantity $M = |z_1|^{-1} \sum_{n=1}^{\infty} |z_n|$. To bound the non-trivial zeros arising in $M$, apply \cref{DH-ZeroSum} with $D_{\psi} = D_K$ since $\psi$ is trivial.  For the trivial zeros, apply \cref{DH-TrivialZeroSum}  in the ``primitive" case for both $\zeta_K(s)$ and $L(s,\chi^*)$. It follows from \eqref{DH-zProperty} that, for $\alpha \geq 1$,
{\small
\begin{equation}
\label{DH-M_Principal}
\begin{aligned}
\frac{M}{\alpha}  & \leq \frac{(\alpha+1/2)^{2}}{\alpha^2}   \cdot \Big[ \log D_K + \Big( \frac{1}{2} + \frac{1}{2^{\alpha+1}-1} \Big) \log Q+  \frac{1}{2} n_K \log T+ \frac{2}{\alpha} + \frac{2}{\alpha+1} \\
& + \Big( \frac{1}{2} \log(\alpha+2)  + \frac{1}{2} \log(\alpha+3) + 1 - \log \pi + \frac{2\alpha}{(\alpha+1)^2 } + \frac{1/2}{2^{\alpha+1}-1}   \Big) n_K   \Big]. 
\end{aligned}
\end{equation}}%
Again, we select $\alpha$ sufficiently large, depending on $\epsilon > 0$, so the RHS of \eqref{DH-M_Principal} is
		\[
		\leq (1+ \tfrac{\epsilon}{50} ) \log D_K + (0.5 + \tfrac{\epsilon}{50})\log Q + (0.5 + \tfrac{\epsilon}{50}) n_K \log T + O_{\epsilon}(n_K). 
		\]
Inputting the resulting bound into \eqref{DH-Penultimate} completes the proof of \cref{DH-MainTheorem}. \hfill \qed

\begin{remark}
To obtain a more explicit version of \cref{DH-MainTheorem}, the only difference in the proof is selecting an explicit value of $\alpha$, say $\alpha = 18$, in the final step of each case. 	The possible choice of $\alpha$ is somewhat arbitrary because the coefficients of $\log D_K, \log Q$ and $n_K$ in \eqref{DH-M_Quadratic} and \eqref{DH-M_Principal} cannot be simultaneously minimized. Hence, in the interest of having relatively small coefficients of comparable size for all quantities, one could choose the value $\alpha = 18$. 	
\end{remark}

\section{Zeros in Low-Lying Rectangles}
\label{sec:ThreePrinciples}

Analogous to Heath-Brown's work \cite{HBLinnik} for the classical case, most of the information pertains to zeros in a ``low-lying" rectangle. In this section, we shall record the relevant existing results and establish some new ones. These will encompass the required three principles in \cref{sec:outline}  and will be applied in the final arguments for the proof of \cref{thm:LPI-MainTheorem}. We begin with some notation.

\subsection{Logarithmic Quantity} Let $\delta_0 > 0$ be fixed and sufficiently small. For the remainder of the paper, denote
{\small
\begin{equation}	
\sL := \begin{cases}
(\tfrac{1}{3}+\delta_0) \log D_K + (\tfrac{19}{36}+\delta_0) \log Q + (\tfrac{5}{12} + \delta_0)   n_K \log n_K  & \text{if $n_K^{5n_K/6} \geq D_K^{4/3} Q^{4/9},$} \\
(1+\delta_0) \log D_K + (\tfrac{3}{4}+\delta_0) \log Q +  \delta_0 n_K \log n_K & \text{otherwise.}	
 \end{cases}
 \label{def:sL}
\end{equation}}%
Notice that
\begin{equation}
\sL   \geq (1+\delta_0)\log D_K + (\tfrac{3}{4}+\delta_0) \log Q + \delta_0 n_K\log n_K\quad\textup{and}\quad \sL  \geq (\tfrac{5}{12} + \delta_0) n_K \log n_K
\label{eqn:sL_lowerbound}
\end{equation}
unconditionally. For $T_{\star} \geq 1$ fixed\footnote{For the purposes of this paper, setting $T_{\star} = 1$ would suffice but we avoid this choice to make the results of \cref{sec:ThreePrinciples} more widely applicable.}, set $T_0 := \max\{ n_K^{5/6} D_K^{-4/3n_K} Q^{-4/9n_K}, T_{\star}\}$. We compare $\cL = \cL_{T_0, \delta_0}$ given by \eqref{def:cL} with $\sL$ and deduce $\cL \leq \sL$ for $\sL$ sufficiently large. This observation implies that 
\begin{equation}
N(1-\tfrac{\lambda}{\sL}, T, \chi) \leq N(1-\tfrac{\lambda}{\cL}, T,\chi) 
\label{eqn:CompareZD}
\end{equation}	
for $\lambda > 0$ and $N(\sigma,T,\chi)$ defined in \eqref{def:N_sigma_chi}. We will utilize this fact in \cref{subsec:ThreePrinciples_LFZD}. 
\subsection{Low-Lying Zeros}
\label{subsec:LPI-notation}

Next, we specify some important zeros of $\prod_{\chi \pmod{H} } L(s,\chi)$ which will be used for the remainder of the paper. Consider the multiset of zeros given by
{\small
\begin{equation}
\cZ := \Big\{ \rho \in \mathbb{C} : \prod_{\chi \pmod{H} } L(\rho,\chi) = 0, 0 < \Re\{ \rho\} < 1, |\Im(\rho)| \leq T_0 \Big\}.
\label{eqn:ZeroMultiset}
\end{equation}}%
We select three important zeros of $\mathcal{Z}$ as follows:  
\begin{itemize}
	\item Choose $\rho_1 \in \cZ$ such that $\Re\{\rho_1\}$ is maximal. Let $\chi_1$ be its associated Hecke character so $L(\rho_1,\chi_1) = 0$. Denote $\rho_1 = \beta_1 + i\gamma_1 = (1-\frac{\lambda_1}{\sL}) + i \frac{\mu_1}{\sL}$, where $\beta_1 = \Re\{\rho_1\}, \gamma_1 = \Im\{\rho_1\}, \lambda_1 > 0,$ and $\mu_1 \in \R$. 
	\item Choose\footnote{If $\rho_1$ is real then $\rho' \in \cZ \setminus \{\rho_1\}$ instead with the other conditions remaining the same.} $\rho' \in \cZ \setminus\{ \rho_1, \bar{\rho_1}\}$ satisfying  $L(\rho',\chi_1) = 0$ such that $\Re\{\rho'\}$ is maximal with respect to these conditions. Similarly denote $\rho' = \beta' + i\gamma' =(1-\frac{\lambda'}{\sL}) + i \frac{\mu'}{\sL}$.  
	\item Choose $\rho_2 \in \cZ \setminus \cZ_1$ such that $\Re\{\rho_2\}$ is maximal and where $\cZ_1$ is the multiset of zeros of $L(s,\chi_1)$ contained in $\cZ$. Let $\chi_2$ be its associated Hecke character so $L(\rho_2,\chi_2) = 0$. Similarly, denote $\rho_2 = \beta_2 + i\gamma_2 = (1-\frac{\lambda_2}{\sL}) + i \frac{\mu_2}{\sL}$. 
\end{itemize}

\subsection{Zero-Free Regions} 
With the above notation, we may introduce the first of three principles. We record the current best-known existing explicit result regarding zero-free regions  of Hecke $L$-functions.

\begin{thm}[Zaman] \label{thm:ZeroFreeRegions}
For $\sL$ sufficiently large, $\min\{\lambda',\lambda_2\} > 0.2866$.  If $\lambda_1 < 0.0875$ then $\rho_1$ is a simple real zero of $ \prod_{\chi \pmod{H}} L(s,\chi)$ and is associated with a real character $\chi_1$. 
\end{thm}
\begin{proof} 
	When $T_{\star} = 1$ and $H = P_{\kq}$ in which case $Q = \N\kq$, this is implied by \cite[Theorems 1.1 and 1.3]{Zaman_2015a} since $\sL$ satisfies \eqref{eqn:sL_lowerbound}. For general congruence subgroups $H$ and any fixed $T_0 \geq 1$, one may easily modify \cite{Zaman_2015a} and obtain results with the same numerical values by:
	\begin{itemize}
		\item Assuming $H \pmod{\kq}$ is primitive, i.e $\kf_H = \kq$. 
		\item Restricting to characters $\chi \pmod{\kq}$ satisfying $\chi(H) =1$ throughout.
		\item Redefining $\cL$ and $\cL^*$ in \cite[Equation (3.2)]{Zaman_2015a} to replace $\log \N\kq$ with $\log Q$. 
		\item Substituting applications of \cite[Lemma 2.4]{Zaman_2015a} with \cref{lem:ImprimitiveSubstitute} since $\kq = \kf_H$. When estimating certain sums, this allows one to transfer from imprimitive characters $\chi \pmod{H}$ to primitive ones.  
		\item Modifying \cite[Lemma 3.2]{Zaman_2015a} so the special value $T_0(\kq)$, in that lemma's notation, instead satisfies $T_{\star} \leq T_0(\kq) \leq T_{\star} \mathcal{T}/10$; one can achieve this by analogously supposing, for a contradiction, that the region $\alpha \leq \sigma \leq 1$ and $T_{\star} 10^j \leq |t| \leq T_{\star} 10^{j+1}$ for $0 \leq j < J$ with $J = [ \frac{\log \mathcal{T}}{\log 10} ]$ contains at least one zero of $\prod_{\chi \pmod{H}} L(s,\chi)$. After applying \cite[Equation (3.4)]{Zaman_2015a} with $T = T_{\star} \mathcal{T}$, the rest of the argument follows similarly. 
	\end{itemize}
	For full details and a complete proof with these modifications, see \cite{Zaman_thesis}. 
\end{proof}

\subsection{Zero Repulsion}

Here we record two explicit estimates for zero repulsion when an exceptional zero exists, also known as ``Deuring-Heilbronn phenomenon". 

\begin{thm}[Zaman]
\label{thm:ZeroRepulsion}
If $\lambda_1 < 0.0875$ then unconditionally, for $\sL$ sufficiently large, $\min\{\lambda', \lambda_2\} > 0.44$.  If $\eta \leq \lambda_1 < 0.0875$ then, for $\sL$ sufficiently large depending on $\eta > 0$, $\min\{\lambda', \lambda_2\} > 0.2103 \log(1/\lambda_1)$.
\end{thm}
\begin{proof} When $T_{\star} = 1$ and $H = P_{\kq}$, this is contained in \cite[Theorem 1.4]{Zaman_2015a} since $\sL$ satisfies \eqref{eqn:sL_lowerbound}. Similar to the proof of \cref{thm:ZeroFreeRegions}, one may modify \cite{Zaman_2015a} to deduce the same theorem for general congruence subgroups $H$ and any fixed $T_{\star} \geq 1$. See \cite{Zaman_thesis} for details.
\end{proof}

\cref{thm:ZeroRepulsion} is not equipped to deal with exceptional zeros $\rho_1$ extremely close to 1 due to the requirement $\lambda_1 \geq \eta$. Thus, we require a more widely applicable version of zero repulsion; this is precisely the purpose of \cref{DH-MainTheorem}, which we restate here in the current notation. 

\begin{thm}
\label{thm:ZeroRepulsion_Full}
Let $T \geq 1$ be arbitrary. Suppose $\chi_1$ is a real character and $\rho_1$ is a real zero. For $\chi \pmod{H}$, let $\rho \neq \rho_1$ be any non-trivial zero of $L(s,\chi)$ satisfying $\frac{1}{2} \leq \Re\{\rho\} = 1 - \frac{\lambda}{\sL} < 1$ and $|\Im\{\rho\}| \leq T$.  For $\sL$ sufficiently large depending on $\epsilon > 0$ and $T$, we have $\lambda > \frac{1}{80+\epsilon} \log(c_{\epsilon}/\lambda_1)$ where $c_{\epsilon} > 0$ is an effective constant depending only on $\epsilon$. 
\end{thm}
\begin{proof}
	This follows immediately from \cref{DH-MainTheorem}  since
	\[
	(48+\epsilon) \log D_K + (60 + \epsilon) \log Q + (24+\epsilon) n_K \log T + O_{\epsilon}(n_K)  \leq (80+2\epsilon) \sL
	\]	
	for $\sL$ sufficiently large depending on $\epsilon$ and $T$. 
\end{proof}

The repulsion constant $\frac{1}{80+\epsilon} \approx 0.0125$ in \cref{thm:ZeroRepulsion_Full} is much smaller than $0.2103$ in \cref{thm:ZeroRepulsion}. This deficiency follows from using power sum arguments; see the remarks following \cref{DH-MainTheorem}. We now quantify how close an exceptional zero $\rho_1$ can be to $1$.

\begin{thm}[Stark \cite{Stark}]
\label{thm:EffectiveSiegelZero}
Unconditionally, $\lambda_1 \gg e^{-24\sL/5}$ where the implicit constant is absolute and effectively computable.
\end{thm}
\begin{proof} 	This follows from \eqref{def:sL}, \eqref{eqn:sL_lowerbound}, and the proof of \cite[Theorem $1'$, p.148]{Stark}.
\end{proof}

\subsection{Log-Free Zero Density Estimates} 
\label{subsec:ThreePrinciples_LFZD}
First, we restate a slightly weaker form of \cref{LFZD-MainTheorem} in the current notation. 

\begin{thm}
\label{thm:LFZD_HighLying}
	Let $T \geq 1$ be arbitrary. If $0 < \lambda < \sL$ then $\sum_{\chi \pmod{H}} N(1-\tfrac{\lambda}{\sL}, T, \chi) \ll e^{162 \lambda}$ provided $\sL$ is sufficiently large depending on $T$. 
\end{thm}
\begin{proof} This follows from \eqref{def:sL} and \cref{LFZD-MainTheorem}. 
\end{proof}

In addition to \cref{thm:LFZD_HighLying}, we will require a completely explicit zero density estimate for ``low-lying" zeros. Define\footnote{Note $\cN(\lambda)$ defined here is \emph{not} the same as $N(\lambda)$ as defined in \cite{Zaman_2015a}. Instead, one has $N(\lambda) \leq \cN(\lambda)$.} 
{\small
\begin{equation}
\begin{aligned}
\cN(\lambda)  = \cN_H(\lambda) & := \sum_{\chi \pmod{H}} N(1-\tfrac{\lambda}{\sL}, T_{\star}, \chi) \\
& = \sum_{\chi \pmod{H}} \#\{ \rho : L(\rho,\chi) = 0, 1-\tfrac{\lambda}{\sL} < \Re\{\rho\} < 1, |\Im\{\rho\}| \leq T_{\star} \}. 
\end{aligned}
\label{def:N_lambda}
\end{equation}}%
By \cref{thm:ZeroFreeRegions}, observe that $\cN(0.0875) \leq 1$ and $\cN(0.2866) \leq 2$. In light of these bounds, we exhibit explicit numerical estimates for $\cN(\lambda)$ in the range with $0.287 \leq \lambda \leq 1$. For each fixed value of $\lambda$, we apply \cref{thm:LFZD_general} with $\upsilon = 0.1$ and $\epsilon \in (0,10^{-5})$ assumed to be fixed and sufficiently small and obtain a bound for $\cN(\lambda \sL/\cL)$. By \eqref{eqn:CompareZD}, the same bound holds for $\cN(\lambda)$. Using computer software \texttt{MATLAB}, we roughly optimize the bound in \cref{thm:LFZD_general} by numerical experimentation over the remaining parameters $(\alpha, \eta, \omega, \xi)$  which produces \cref{table:ZeroDensity}. Note that we have verified $J(\xi \lambda) < 1$ and $X_{\xi \lambda} > Y_{\xi \lambda} > 4.6$ in each case. 

{\small
\begin{table}
\begin{tabular}{l|c|c|c|c|c|c|c|c|} 
$\lambda$ & $\log N(\lambda) \leq$ & $\alpha$ & $\eta$ & $\omega$ & $\xi$ & $J(\xi \lambda)$ & $Y_{\xi\lambda}$ & $X_{\xi\lambda}$\\
\hline
 .287 &  198.1 &  .3448 &  .09955 &  .03466 &  1.0082 &  .46 &  5.8 &  993 \\ 
 .288 &  198.3 &  .3444 &  .09943 &  .03462 &  1.0082 &  .46 &  5.8 &  991 \\ 
 .289 &  198.5 &  .3441 &  .09931 &  .03458 &  1.0082 &  .46 &  5.8 &  988 \\ 
 .290 &  198.7 &  .3437 &  .09918 &  .03454 &  1.0082 &  .46 &  5.8 &  986 \\ 
 .291 &  198.9 &  .3433 &  .09906 &  .03450 &  1.0082 &  .46 &  5.8 &  984 \\ 
 .292 &  199.1 &  .3429 &  .09894 &  .03446 &  1.0081 &  .46 &  5.8 &  982 \\ 
 .293 &  199.3 &  .3426 &  .09882 &  .03442 &  1.0081 &  .46 &  5.8 &  979 \\ 
 .294 &  199.5 &  .3422 &  .09870 &  .03439 &  1.0081 &  .46 &  5.8 &  977 \\ 
 .295 &  199.8 &  .3418 &  .09859 &  .03435 &  1.0081 &  .46 &  5.8 &  975 \\ 
 .296 &  200.0 &  .3415 &  .09847 &  .03431 &  1.0081 &  .46 &  5.8 &  973 \\ 
 .297 &  200.2 &  .3411 &  .09835 &  .03427 &  1.0080 &  .46 &  5.8 &  970 \\ 
 .298 &  200.4 &  .3408 &  .09823 &  .03423 &  1.0080 &  .46 &  5.8 &  968 \\ 
 .299 &  200.6 &  .3404 &  .09811 &  .03420 &  1.0080 &  .46 &  5.8 &  966 \\ 
 .300 &  200.8 &  .3400 &  .09800 &  .03416 &  1.0080 &  .46 &  5.8 &  964 \\ 
 .325 &  205.9 &  .3316 &  .09518 &  .03326 &  1.0075 &  .47 &  5.8 &  914 \\ 
 .350 &  211.0 &  .3240 &  .09257 &  .03242 &  1.0071 &  .47 &  5.7 &  871 \\ 
 .375 &  216.0 &  .3171 &  .09014 &  .03163 &  1.0067 &  .47 &  5.7 &  833 \\ 
 .400 &  220.9 &  .3108 &  .08787 &  .03090 &  1.0064 &  .48 &  5.7 &  800 \\ 
 .425 &  225.7 &  .3054 &  .08678 &  .02878 &  1.0061 &  .46 &  5.6 &  769 \\ 
 .450 &  230.4 &  .2998 &  .08373 &  .02956 &  1.0059 &  .48 &  5.6 &  744 \\ 
 .475 &  235.1 &  .2948 &  .08184 &  .02895 &  1.0056 &  .48 &  5.6 &  720 \\ 
 .500 &  239.8 &  .2903 &  .08006 &  .02837 &  1.0054 &  .49 &  5.6 &  699 \\ 
 .550 &  249.0 &  .2821 &  .07677 &  .02729 &  1.0050 &  .49 &  5.5 &  661 \\ 
 .600 &  258.0 &  .2748 &  .07379 &  .02631 &  1.0046 &  .50 &  5.5 &  629 \\ 
 .650 &  266.9 &  .2684 &  .07109 &  .02542 &  1.0043 &  .50 &  5.4 &  602 \\ 
 .700 &  275.6 &  .2627 &  .06862 &  .02460 &  1.0041 &  .50 &  5.4 &  579 \\ 
 .750 &  284.3 &  .2576 &  .06634 &  .02383 &  1.0039 &  .51 &  5.4 &  559 \\ 
 .800 &  292.9 &  .2529 &  .06424 &  .02313 &  1.0037 &  .51 &  5.4 &  541 \\ 
 .850 &  301.4 &  .2486 &  .06230 &  .02247 &  1.0035 &  .51 &  5.3 &  525 \\ 
 .900 &  309.8 &  .2447 &  .06049 &  .02186 &  1.0033 &  .51 &  5.3 &  510 \\ 
 .950 &  318.2 &  .2412 &  .05880 &  .02128 &  1.0032 &  .52 &  5.3 &  497 \\ 
 1.00 &  326.5 &  .2378 &  .05722 &  .02074 &  1.0030 &  .52 &  5.3 &  486 \\ 

\end{tabular}
\caption{Bounds for $N(\lambda)$ using \cref{thm:LFZD_general} with $\upsilon = 0.1$ and $\epsilon \in (0,10^{-5}]$ }
\label{table:ZeroDensity}
\end{table}}%

Based on \cref{table:ZeroDensity}, we may also establish an explicit estimate for $\cN(\lambda)$ by specifying parameters in \cref{thm:LFZD_general}.  
\begin{thm} \label{thm:LFZD-LowLying}
Let $\epsilon_0 > 0$ be fixed and sufficiently small. If $0 < \lambda <  \epsilon_0 \sL$ then $\cN(\lambda) \leq e^{162 \lambda + 188}$ for $\sL$ sufficiently large. If $0 < \lambda \leq 1$,  $\cN(\lambda)$ is also bounded as in Table 1.
\end{thm}

\begin{proof} For $\lambda \leq 0.2866$, the result is immediate as $\cN(0.2866) \leq 2$ by \cref{thm:ZeroFreeRegions}. For $0.2866 \leq \lambda \leq 1$, one can directly verify the desired bound by using \cref{table:ZeroDensity}.  Now, consider $\lambda \geq 1$. Apply \cref{thm:LFZD_general} with
	{\small
\[
	\begin{array}{llllllll}
	T = T_0, \quad & \lambda_0 = 1, \quad & \alpha  = 0.1549, \quad & \eta = 0.05722, \quad \\
	\epsilon = 10^{-5}, &  \upsilon = 0.1,  & \xi = 1.0030, &  \omega = 0.02074. 
	 \end{array}
	 \]}%
	This choice of values is motivated by the last row of \cref{table:ZeroDensity}, but with a more suitable choice for $\alpha$. With this selection, one can check that for any $\lambda \geq 1$,
	\[
	4.61 \leq Y_{\xi \lambda} \leq 9.2, \qquad 264 \leq X_{\xi \lambda} \leq 526, \qquad J(\xi \lambda) \leq 0.272.
	\]
	These inequalities can be verified by elementary arguments involving the definitions in \cref{sec:LFZD_notation} and \eqref{eqn:Jfunction}. In particular, for any $\lambda \geq 1$, the assumptions of \cref{thm:LFZD_general} are satisfied for all $1 \leq \lambda < \epsilon_0 \sL$. 
	
	Now with these estimates, we may deduce upper bounds for $C_4, C_3, C_1, C_0, B_2, B_1$ in \cref{thm:LFZD_general} as follows:
{\small
	\begin{align*}
	C_4 = C_4(\lambda) & \leq 6.0 \times 10^{13}, \qquad\qquad C_1 \leq 17, \qquad\qquad B_2 \leq 154, \\
	C_3 = C_3(\lambda) & \leq 2.4 \times 10^{14}, \qquad\qquad C_0 \leq 65, \qquad\qquad B_1 \leq 156, 
	\end{align*}}%
	for $\lambda \geq 1$. Thus, by \cref{thm:LFZD_general}, for $1 \leq \lambda \leq \epsilon_0 \sL$,  
	\[
	N(\lambda) \leq 52 \big( 6.0 \times 10^{13} \cdot \lambda^4 + 2.4 \times 10^{14} \cdot \lambda^3 + 17 \cdot \lambda + 65\big) e^{156 \lambda + 154}.
	\]
	To simplify the expression on the RHS, we crudely observe that the above is
{\small
	\begin{align*}
	& \leq 52 \cdot 65 \big( 6.0 \times 10^{13} \cdot \frac{24}{6^4 \cdot 65} \cdot \frac{(6\lambda)^4}{4!} + 2.4 \times 10^{14} \cdot \frac{6}{6^3 \cdot 65} \cdot \frac{(6\lambda)^3}{3!} + 6\lambda + 1\big) e^{156 \lambda + 154} \\
	& \leq 52 \cdot 6.7 \times 10^{12} \cdot \big( \frac{(6\lambda)^4}{4!} + \frac{(6\lambda)^3}{3!} + 6\lambda + 1\big) e^{156 \lambda + 154}  \leq e^{162 \lambda +  188},
	\end{align*}}%
	 as desired. 
\end{proof}

\section{Proof of \cref{thm:LPI-MainTheorem}: Preliminaries}
\label{sec:LPI_Preliminaries}
We may finally begin the proof of  \cref{thm:LPI-MainTheorem}. The arguments below are motivated by \cite[Section 10]{HBLinnik} and mostly follows the structure of \cite[Section 4]{Zaman_2015c}. Recall that we retain the notation introduced in \cref{sec:ThreePrinciples} for the remainder of the paper. 
\subsection{Choice of Weight}
We define a weight function (see \cite[Lemma 2.6 and 2.7]{Zaman_2015c}) and describe its properties. 

\begin{lem}
\label{lem:WeightChoice}

For real numbers $A,B > 0$ and positive integer $\ell \geq 1$ satisfying $B > 2\ell A$, there exists a real-variable function $f(t)  = f_{\ell}(t; B,A)$ such that:
\begin{enumerate}[(i)]
	\item $0 \leq f(t) \leq A^{-1}$ for all $t \in \R$.
	\item The support of $f$ is contained in $[B-2\ell A, B]$. 
	\item Its Laplace transform $F(z) = \int_0^{\infty} f(t) e^{-zt}dt$ is given by
{\small
			\begin{equation}	
				F(z) = e^{-(B-2\ell A)z} \Big( \frac{1-e^{-Az}}{Az} \Big)^{2\ell}. 
				\label{eqn:WeightLaplace}
			\end{equation}}%
	\item For $x >0$ and $y \in \R$, $|F(x+iy)| \leq e^{-(B-2\ell A) x}(\frac{1-e^{-Ax}}{Ax})^{2\ell} \leq e^{-(B-2\ell A) x}.$
\end{enumerate}
\end{lem}

For the entirety of this section, let real numbers $A,B > 0$ and positive integer $\ell \geq 1$ be arbitrary satisfying $B  > 2\ell A$, and denote $f(\, \cdot \,) = f_{\ell}(\, \cdot \, ; B, A)$. The Laplace transform of $f(t)$ will be written as $F(z)$. 

\subsection{A weighted sum of prime ideals} For the congruence class group $H \pmod{\kq}$, let $\cC$ be an element of class group of $H$; that is, $\cC \in I(\kq)/H$.  Using the compactly-supported weight $f$, define
{\small
\begin{equation}
S := \sum_{\substack{ \mathfrak{p} \, \nmid \, \kq \kD_K \\ \N\mathfrak{p} \text{ is a rational prime} }} \frac{\log \N\mathfrak{p}}{\N\mathfrak{p}}	f\Big( \frac{\log \N\mathfrak{p}}{\sL} \Big) \cdot \mathbf{1}_{\cC}(\mathfrak{p})
\label{def:S_WeightedPrimes}
\end{equation}}%
where $\mathbf{1}_{\mathcal{C}}(\, \cdot\,)$ is an indicator function for the coset $\mathcal{C}$, $\kD_K$ is the different of $K$, and the sum is over degree 1 prime ideals $\mathfrak{p}$ of $K$ not dividing $\kq\kD_K$. We reduce the proof of \cref{thm:LPI-MainTheorem} to verifying the following lemma. 
\begin{lem} \label{lem:ReduceMainTheorem}
	Let $\eta > 0$ be sufficiently small and let $\km$ be the product of prime ideals dividing $\kq$ but not $\kf_H$. If $h_H \sL^{-1} S \gg_{\eta} \min\{ 1, \lambda_1\}$ for 
	\begin{equation}
	B \geq \max\{ 693.5, \tfrac{\log \N\km}{\sL} + 8\eta\}, \qquad A = 4/\sL, \qquad \ell = \lfloor \eta \sL \rfloor
	\label{eqn:ReduceMainTheorem}
	\end{equation}
	and $\sL$ sufficiently large then \cref{thm:LPI-MainTheorem} holds. 	
\end{lem}
\begin{proof} Select $B = \frac{\log x}{\sL}$ with $A = \frac{4}{\sL}$ and $\ell = \lfloor \eta \sL \rfloor$. From the definition \eqref{def:sL} of $\sL$ and the condition on $x$ in \eqref{eqn:LPI-MainTheorem_xRange}, one can verify that $B$ satisfies \eqref{eqn:ReduceMainTheorem}. 
	Now, since $f$ is supported in $[B-2\ell A, B]$ and $|f| \leq A^{-1} \leq \sL$ by \cref{lem:WeightChoice},
\[
S \leq \sL e^{8\eta \sL}x^{-1}\log x \#\{ \mathfrak{p} : \N\mathfrak{p} \leq x,  \deg(\mathfrak{p}) = 1, \mathfrak{p} \in \cC \}.
\]
Multiplying both sides by $h_H \sL^{-1}$ and noting $B$ satisfies \eqref{eqn:ReduceMainTheorem}, we conclude
{\small
\begin{align*}
\#\{ \mathfrak{p} : \N\mathfrak{p} \leq x,  \deg(\mathfrak{p}) = 1, \mathfrak{p} \in \cC \}  \geq   \frac{4 S}{ \sL  } \cdot \frac{x e^{-8\eta \sL}}{\log x}  \gg_{\eta} e^{-5\sL}\cdot \frac{x }{h_H \log x}. 
\end{align*}}%
by \cref{thm:ZeroFreeRegions,thm:EffectiveSiegelZero}. Fixing $\eta$ and noting $\sL \leq \log( D_K Q n_K^{n_K})$ yields \cref{thm:LPI-MainTheorem}.
\end{proof}

Now, by orthogonality of characters,
{\small
\begin{equation}
S = \frac{1}{h_H}\sum_{\chi \pmod{H}} \bar{\chi}(\cC) S_{\chi},\quad \text{where} \quad S_{\chi} := \sum_{\substack{ \mathfrak{p} \, \nmid \, \kq\kD_K \\ \N\mathfrak{p} \text{ is a rational prime} }} \frac{\log \N\mathfrak{p}}{\N\mathfrak{p}}	\chi(\mathfrak{p}) f\Big( \frac{\log \N\mathfrak{p}}{\sL} \Big).
\label{eqn:S_orthgonality}
\end{equation}}%
We wish to write $S_{\chi}$ as a contour integral involving a logarithmic derivative of a \emph{primitive} Hecke $L$-function. Before doing so, we define
{\small
\begin{equation}
\km = \prod_{\substack{\mathfrak{p} \mid \kq,~\mathfrak{p} \nmid \kf_H}} \mathfrak{p}.
\label{def:RadicalQ}
\end{equation}}%

\begin{lem} \label{lem:SumToContour}
If $B - 2\ell A > \max\{ 1,  \frac{\log \N\km}{\sL}\}$ then
{\small
\[
\sL^{-1} S_{\chi} = \frac{1}{2 \pi i} \int_{2 - i\infty}^{2+i\infty} - \frac{L'}{L}(s,\chi^*) F((1-s)\sL) ds + O\big(A^{-1} e^{-(B-2\ell A) \sL/2} \big)
\]}%
where $\chi^*$ is the primitive Hecke character inducing $\chi \pmod{H}$. 
\end{lem}
\begin{proof} Observe
{\small
\[
\frac{1}{2 \pi i} \int_{2 - i\infty}^{2+i\infty} - \frac{L'}{L}(s,\chi^*) F((1-s)\sL) ds = \sL^{-1} \sum_{\mathfrak{n}} \frac{\Lambda(\mathfrak{n})}{\N\mathfrak{n}} \chi^*(\mathfrak{n}) f\Big( \frac{\log \N\mathfrak{n}}{\sL} \Big) = \sL^{-1} \tilde{S}_{\chi},
\]}%
say. Thus, we must show $\tilde{S}_{\chi}$ equals $S_{\chi}$ up to a negligible contribution from prime ideal powers, prime ideals whose norm is not a rational prime, and prime ideals dividing $\kq\kD_K$.  For simplicity, denote $X = e^{(B-2\ell A)\sL}$. 
\subsubsection*{Prime ideal powers}
By \cref{lem:WeightChoice},  the contribution of such ideals in $\tilde{S}_{\chi}$ is bounded by
{\small
\[
 \sum_{\mathfrak{p}} \sum_{m\geq 2} \frac{\log \N\mathfrak{p}}{\N\mathfrak{p}^m} f\Big( \frac{\log \N\mathfrak{p}^m}{\sL} \Big)
  \leq A^{-1} \sum_{\mathfrak{p}} \sum_{\substack{ m \geq 2\\  \N\mathfrak{p}^m  \geq X}} \frac{\log \N\mathfrak{p}}{\N\mathfrak{p}^m}. 
\]}%
Since a rational prime $p$ splits into at most $n_K$ prime ideals in $K$, the RHS is
{\small
\[
  \leq A^{-1} \sum_{p \text{ rational} } \sum_{(p) \subseteq \mathfrak{p}}  \sum_{\substack{ m \geq 2\\  \N\mathfrak{p}^m  \geq X}} \frac{\log \N\mathfrak{p}}{\N\mathfrak{p}^m}  \leq A^{-1} \sum_{\substack{ p \text{ rational} \\ p \geq X^{1/2}}} \frac{1}{p^2} \sum_{(p) \subseteq \mathfrak{p}} \log \N\mathfrak{p} \ll A^{-1} \sL X^{-1/2}
\]}%
by partial summation and noting $n_K \ll \sL$ from Minkowski's bound. 

\subsubsection*{Prime ideals with norm not equal to a rational prime} 
By \cref{lem:WeightChoice}, 
{\small
\[
\sum_{\substack{ \mathfrak{p}   \\ \N\mathfrak{p} \text{ not a rational prime}}} \sum_{m=1}^{\infty} \frac{\log \N\mathfrak{p}}{\N\mathfrak{p}^m} f\Big( \frac{\log \N\mathfrak{p}^m}{\sL} \Big) 
  \ll A^{-1}\sum_{\substack{ \N\mathfrak{p} \geq X \\ \N\mathfrak{p} \text{ not a rational prime}}}  \frac{\log \N\mathfrak{p}}{\N\mathfrak{p}}. 
\]}%
For $\mathfrak{p}$ appearing in the righthand sum and lying above the rational prime $p$, notice $\N\mathfrak{p} \geq p^2$. Thus, arguing as in the previous case, we deduce
{\small
\[
\ll A^{-1} \sum_{\substack{p \geq X^{1/2} \\ p \text{ rational prime}}} \frac{1}{p^2} \sum_{(p) \subseteq \mathfrak{p}} \log \N\mathfrak{p}  \ll A^{-1} \sL X^{-1/2}. 
\]}%
\subsubsection*{Prime ideals dividing $\kq\kD_K$}  Since $B - 2\ell A > \max\{1, \tfrac{\log \N\km}{\sL} \}$, $\N \kD_K \leq D_K \leq e^{\sL}$ by \eqref{eqn:sL_lowerbound}, and  $f$ is supported in $[B-2\ell A, B]$, we have $f(\frac{\log \N\mathfrak{p}}{\sL}) = 0$ for $\mathfrak{p} \mid \km\kD_K$.  As $\chi(\mathfrak{p}) = \chi^*(\mathfrak{p})$ for all $\mathfrak{p} \nmid \km$ this implies that $\chi(\mathfrak{p}) f( \frac{\log \N\mathfrak{p}}{\sL}) = \chi^*(\mathfrak{p}) f(\frac{\log \N\mathfrak{p}}{\sL})$ for all prime ideals $\mathfrak{p}$. Combining all of these contributions to compare $S_{\chi}$ with $\tilde{S}_{\chi}$ yields the desired result. 
\end{proof}

Applying \cref{lem:SumToContour} to \eqref{eqn:S_orthgonality}, we deduce
{\small
\begin{equation}
\sL^{-1} S = \frac{1}{h_H}\sum_{\chi \pmod{H}}  \frac{1}{2\pi i} \int_{2-i\infty}^{2+i\infty} -\frac{L'}{L}(s,\chi^*) F((1-s)\sL) ds + O\big(A^{-1} e^{-(B-2\ell A) \sL/2} \big),
\label{eqn:S_AfterMellin}
\end{equation}}%
provided $B-2\ell A > \max\{1, \frac{\log \N\km}{\sL}\}	$. 
\subsection{A sum over low-lying zeros} The next step is to shift the contour in \eqref{eqn:S_AfterMellin} and pick up the arising poles. Our objective in this subsection is to reduce the analysis to the ``low-lying" zeros of Hecke $L$-functions. 

\begin{lem} 
\label{lem:Height1} Let $T_{\star} \geq 1$ be fixed, and let $\rho_1$ and $\chi_1$ be as in \cref{subsec:LPI-notation}. If $B - 2\ell A  > \max\{ 162, \frac{\log\N\km}{\sL}\}$, $\ell > \frac{81 n_K+162}{4}$, and $A > \frac{1}{\sL}$, and $\sL$ is sufficiently large then
{\small
\begin{equation*}
\begin{aligned}
\big| h_H  \sL^{-1} S - F(0) + \bar{\chi_1}(\cC) F( (1-\rho_1) \sL)\big|
& \leq \sum_{\substack{ \chi \pmod{H} \\} } \sumP_{\rho} |F( (1-\rho) \sL)| \\
& \qquad + O\Big( \Big( \frac{2}{AT_{\star}\sL}\Big)^{2\ell} T_{\star}^{40.5 n_K} +   e^{-78 \sL} \Big)
\end{aligned}
\end{equation*}}%
where the  sum $\sum'$ indicates a restriction to non-trivial zeros $\rho \neq \rho_1$ of $L(s,\chi)$, counted with multiplicity, satisfying $0 < \Re\{\rho\} < 1$ and $|\Im\{\rho\}| \leq T_{\star}$.
\end{lem}

\begin{proof}
	Shift the contour in \eqref{eqn:S_AfterMellin} to the line $\Re\{s\} = -\tfrac{1}{2}$. For each primitive character $\chi^*$, this picks up  the non-trivial zeros of $L(s,\chi)$, the simple pole at $s=1$ when $\chi$ is trivial, and the trivial zero at $s=0$ of $L(s,\chi)$ of order $r(\chi)$. To bound the remaining contour, by \cite[Lemma 2.2]{LMO} and \cref{lem:WeightChoice}(iii) with \cite[Lemma 2.7]{Zaman_2015c}, we have for $\Re\{s\} = -1/2$,
{\small
	\[
	- \frac{L'}{L}(s,\chi^*) \ll \sL + n_K \log(|s|+2), 
	\quad \text{and} 
	\quad |F((1-s)\sL)| \ll e^{-\frac{3}{2}(B-2\ell A)\sL} \cdot |s|^{-2}
	\]}%
	since $A > 1/\sL$. It follows that
	{\small
\[
	\frac{1}{2\pi i } \int_{-1/2-i\infty}^{-1/2+i\infty} -\frac{L'}{L}(s,\chi^*) F((1-s)\sL) ds \ll \sL e^{-\frac{3}{2}(B-2\ell A)\sL}.
	\]}%
	Overall, \eqref{eqn:S_AfterMellin} becomes
	{\small
	\begin{equation}
	\frac{h_H S}{\sL}  - F(0) + \sum_{ \substack{ \chi \pmod{H} }} \bar{\chi}(\cC) \sum_{\rho} F((1-\rho)\sL)\ll \sum_{\chi \pmod{H}} r(\chi) F(\sL) + \frac{ \sL}{ e^{(B-2\ell A)\sL/2}} ,
	\label{eqn:I_ShiftContour}
	\end{equation}}%
	where the inner sum over $\rho$ is over all non-trivial zeros of $L(s,\chi)$. From  \eqref{GammaFactor_Exponents} and \eqref{TrivialZeros}, notice $r(\chi) \leq n_K$. Thus, by \cref{lem:WeightChoice} and Minkowski's bound $n_K\ll\sL$,
	{\small
	\[
	\frac{1}{h_H} \sum_{\chi \pmod{H}} r(\chi) F(\sL) \ll \sL e^{-(B-2\ell A)\sL}.
	\]}%
	Since $h_H \ll e^{2\sL}$ by \cref{lem:h_H-Bound} and \eqref{eqn:sL_lowerbound}, it follows from \eqref{eqn:I_ShiftContour} that
	{\small
	\[
	h_H \sL^{-1}S  = F(0) - \sum_{ \substack{\chi \pmod{H}}} \bar{\chi}(\cC) \sum_{\rho} F((1-\rho)\sL) + O\Big( \sL e^{-(B-2\ell A-4)\sL/2} \Big).
	\]}%
	The error term is bounded by $O(e^{-78\sL})$ as $B-2\ell A > 162$. Therefore, it suffices to show
	{\small
	\[
	Z := \sum_{\chi \pmod{H}} \sum_{k=0}^{\infty} \sum_{ \substack{\rho \\ 2^k T_{\star} \leq \Im\{\rho\} < 2^{k+1} T_{\star} }} |F((1-\rho)\sL)| \ll \Big( \frac{2}{A T_{\star}\sL}\Big)^{2\ell} T_{\star}^{40.5 n_K}.
	\]}%
	From \cref{lem:WeightChoice}, writing $\rho = \beta+i\gamma$ with $\beta \geq 1/2$, observe
	{\small
	\[
	|F(\rho \sL)| + |F((1-\rho)\sL)| \leq 2e^{ -(B-2\ell A) (1-\beta) \sL} \Big( \frac{2}{A |\gamma| \sL}\Big)^{2\ell}
	\]}%
	and moreover, from \cref{LFZD-MainTheorem}, 
		{\small
\[
	\tilde{N}(\sigma) := \sum_{\chi \pmod{H}} N(\sigma, 2T,\chi) \ll \big( e^{162 \sL} T^{81 n_K +162} \big)^{(1-\sigma)}
	\]}%
	for $\tfrac{1}{2}\leq \sigma \leq 1, T \geq 1$, and $\sL$ sufficiently large. Thus, by partial summation, 
	{\small
	\begin{align*}
	\sum_{\chi \pmod{H}} \sum_{\substack{\rho \\ T \leq |\Im\{\rho\}| \leq 2T} } |F((1-\rho)\sL)| 
	& \ll  \Big( \frac{2}{A T  \sL}\Big)^{2\ell} \int_{1/2}^1 e^{ -(B-2\ell A) (1-\sigma) \sL} d\tilde{N}(\sigma)\\
	& \ll \Big( \frac{2}{A\sL}\Big)^{2\ell}  T^{40.5n_K+81 - 2\ell}
	\end{align*}}
	since $B > 2\ell A + 162$. Note we have used that the zeros  of $\prod_{\chi \pmod{H}}L(s,\chi)$ are symmetric across the critical line $\Re\{s\} =1/2$. Overall, we deduce
	\[
		Z   \ll \Big( \frac{2}{A\sL}\Big)^{2\ell} T_{\star}^{40.5 n_K + 81 - 2\ell} \sum_{k=0}^{\infty} (2^k)^{40.5n_K+81-2\ell}  \ll \Big( \frac{2}{A T_{\star} \sL}\Big)^{2\ell} T_{\star}^{40.5 n_K}, 
	\]
	since $\ell > \tfrac{81n_K + 162}{4}$ and $T_{\star}$ is fixed, as desired. 
\end{proof}

We further restrict the sum over zeros in \cref{lem:Height1} to zeros $\rho$ close to the line $\Re\{s\} = 1$. To simplify the statement, we also select parameters $\ell$ and $A$ for the weight function. 

\begin{lem}
	\label{lem:LowLyingZeros}
	Let $T_{\star} \geq 1$ and $\eta \in (0,1)$ be fixed and $1 \leq R \leq \sL$ be arbitrary. Suppose 
		\begin{equation}
		 B - 2\ell A > \max\Big\{ 162, \frac{\log \N\km}{\sL}\Big\},  \qquad A = \frac{4}{\sL}, \qquad \ell = \lfloor \eta \sL \rfloor.
		\label{eqn:LowLyingZeros_assumptions}
		\end{equation}
	If $\sL$ is sufficiently large then
{\small
	\begin{align*}
	\big| h_H  \sL^{-1} S - F(0) + \bar{\chi_1}(\cC) F( (1-\rho_1) \sL)\big|
	& \leq \sum_{\substack{\chi \pmod{H} } } \sumS_{\rho} |F( (1-\rho) \sL)| \\
	& + O( e^{-(B -2\ell A -162)R} + ( 2 T_{\star})^{-2 \eta \sL} e^{\eta \sL} + e^{-78 \sL})
	\end{align*}}%
where the marked sum $\sum^{\star}$ runs over zeros $\rho \neq \rho_1$ of $L(s,\chi)$, counting with multiplicity, satisfying $1 - \frac{R}{\sL} < \Re\{\rho\} < 1$ and $|\Im\{\rho\}| \leq T_{\star}$.
\end{lem}
\begin{proof} For $\sL$ sufficiently large depending on $\epsilon$ and $\eta$, the quantities $B, A$ and $\ell$ satisfy the assumptions of \cref{lem:Height1}. Denote $B' = B - 2\ell A$. We claim it suffices to show
{\small
	\begin{equation}
	\sum_{\chi \pmod{H}}~\sumP_{ \Re\{\rho\} \leq 1 - R/\sL }  |F((1-\rho)\sL)| \ll e^{-(B'-162)R}
	\label{eqn:LowLyingZeros_1}
	\end{equation}}%
	where $\sum'$ is defined in \cref{lem:Height1}. 	To see the claim, we need only show that the error term in \cref{lem:Height1} is absorbed by that of \cref{lem:LowLyingZeros}. For $\sL$ sufficiently large, notice $T_{\star}^{40.5 n_K} \leq e^{\eta \sL}$ as $n_K \log T_{\star} = o(\sL)$; hence, for our choices of $A$ and $\ell$, we have
{\small
	\[
	\Big( \frac{2}{A T_{\star} \sL}\Big)^{2\ell} T_{\star}^{40.5 n_K} \leq \Big( \frac{1 }{2 T_{\star}}\Big)^{2 \eta \sL} e^{\eta \sL}.
	\]}%
	This proves the claim. 	Now, to establish \eqref{eqn:LowLyingZeros_1}, define the multiset of zeros
	\[
	\cR_m(\chi) := \big\{ \rho : L(\rho,\chi) = 0, \quad 1 - \tfrac{m+1}{\sL} \leq \Re\{\rho\} \leq 1 - \tfrac{m}{\sL}, \quad  |\Im(\rho\}| \leq T_{\star} \big\} 
	\]
	for $1 \leq m \leq \sL$. By \cref{thm:LFZD_HighLying} and \cref{lem:WeightChoice}, it follows that
{\small
	\[
	\sum_{\chi \pmod{H}} \sum_{\rho \in R_m(\chi)} |F((1-\rho)\sL)|  \leq e^{-B' m} \sum_{\chi \pmod{H}} \# \mathcal{R}_m(\chi) \ll e^{-(B' - 162) m}
	\]}%
	for $\sL$ sufficiently large. Summing over $m \geq R$ yields the desired conclusion. 
\end{proof}

\section{Proof of \cref{thm:LPI-MainTheorem}: Exceptional Case}
\label{sec:LPI_EX}
For this section, we assume $\lambda_1 < 0.0875$, in which case $\rho_1$ is an exceptional real zero by \cref{thm:ZeroFreeRegions}. Thus, $\rho_1$ is a simple real zero and $\chi_1$ is a real Hecke character. For fixed $\eta \in (0,10^{-3})$ sufficiently small, assume $\sL$ is sufficiently large, $B \geq \max\{ 163, \frac{\log \N\km}{\sL} + 8\eta \}$, $\ell = \lfloor \eta \sL \rfloor$, and $A = \frac{4}{\sL}$.  Thus $B, \ell$, and $A$ satisfy \eqref{eqn:LowLyingZeros_assumptions} and $B' := B-2\ell A > 162$. For the moment, we do not make any additional assumptions on the minimum size of $B$ and hence $B'$. To prove \cref{thm:LPI-MainTheorem} when $\rho_1$ is an exceptional zero, it suffices to show, by \cref{lem:ReduceMainTheorem}, that $h_H \sL^{-1} S \gg \min\{1, \lambda_1\}$ for  $B \geq \max\{593, \tfrac{\log \N\km}{\sL} + 8\eta \}$ and $\sL$ sufficiently large.

For a non-trivial zero $\rho$ of a Hecke $L$-function, write $\rho = \beta + i\gamma = ( 1 - \frac{\lambda}{\sL}) + i \gamma$ so by \cref{lem:WeightChoice}, $|F((1-\rho)\sL)| \leq e^{-B'\lambda}$.  From \cref{lem:LowLyingZeros,} with  $T_{\star} \geq 1$ fixed and $1 \leq R \leq \sL$ arbitrary, it follows that if we define
{\small
\begin{equation}
	\Delta = \begin{cases} 
 				\eta & \text{if $T_{\star}=1$ and $R=R(\eta)$ is sufficiently large,} \\
 				O(e^{-(B'-162) R} + e^{-78 \sL}) & \text{if $T_{\star} = T_{\star}(\eta)$ is sufficiently large and $1 \leq R \leq \sL$,}
			 \end{cases}
	\label{eqn:LPI_EX_zeros_error}
\end{equation}}%
then
{\small
\begin{equation}
h_H \sL^{-1}S \geq 1 -  \chi_1(\cC) e^{-B' \lambda_1} - \sum_{\chi \pmod{H}} \sumS_{\rho} e^{-B' \lambda} - \Delta
\label{eqn:LPI_EX_zeros}
\end{equation}}%
where the restricted sum $\sum^{\star}$ is over zeros $\rho \neq \rho_1$, counted with multiplicity, satisfying $0 < \lambda \leq R$ and $|\gamma| \leq T_{\star}$.

Suppose the arbitrary parameter $\lambda^{\star} > 0$ satisfies 
\begin{equation}
	\lambda > \lambda^{\star}
	\label{eqn:Lambda_star_Assumption} \quad \text{for every zero $\rho$ occurring in the restricted sum of \eqref{eqn:LPI_EX_zeros}.}
\end{equation}
It remains for us to divide into cases according to the range of $\lambda_1$ and value of $\chi_1(\cC) \in \{ \pm 1\}$.  In each case, we make a suitable choice for $\lambda^{\star}$. 

\subsection{Moderate Exceptional Zero ($\eta \leq \lambda_1 < 0.0875$ or $\chi_1(\cC) = -1$)}
For the moment, we do not make any assumptions on the size of $\lambda_1$ other than that $0 < \lambda_1 < 0.0875$. Select $T_{\star} = 1$ and $R = R(\eta)$ sufficiently large so $\Delta = \eta$ according to \eqref{eqn:LPI_EX_zeros_error}.  By partial summation, our choice of $\lambda^{\star}$ in \eqref{eqn:LPI_EX_zeros}, and \cref{thm:LFZD-LowLying}, 
{\small
\[
\sum_{\chi \pmod{H}} \sumS_{\rho} e^{-B' \lambda}
  \leq \int_{\lambda^{\star}}^R e^{-B' \lambda} d\cN(\lambda)  
 \leq  e^{- (B'-162) R + 188} + \int_{\lambda^{\star}}^{\infty} B' e^{-(B'-162) \lambda + 188} d\lambda. 
\]}%
As $R = R(\eta)$ is sufficiently large and $B' > 162$, the above is $\leq \frac{1}{1-  162/B'} e^{188-(B'-162) \lambda^{\star}} + \eta$.  Comparing with \eqref{eqn:LPI_EX_zeros}, we have
\begin{equation}
h_H \sL^{-1}S \geq 1 -  \chi_1(\cC) e^{- B' \lambda_1} - (1-162/B')^{-1}e^{- (B'-162) \lambda^{\star} + 188} - 2\eta. 
\label{eqn:LPI_EX_preDH}
\end{equation}
Finally, we further subdivide into cases according to the size of $\lambda_1$ and value of $\chi_1(\cC) \in \{\pm 1\}$. Recall $\eta > 0$ is sufficiently small. 
\subsubsection{$\lambda_1$ medium ($10^{-3} \leq \lambda_1 < 0.0875$)} Here we also assume $B \geq 593$ in which case $B' \geq 592$. Select  $\lambda^{\star} = 0.44$ which,  by \cref{thm:ZeroRepulsion}, satisfies \eqref{eqn:Lambda_star_Assumption} for the specified range of $\lambda_1$. Inputting this estimate in \eqref{eqn:LPI_EX_preDH} and noting  $|\chi_1(\cC)| \leq 1$, we deduce
{\small
\[
h_H \sL^{-1}S \geq 1 - e^{-592 \times 10^{-3}} - \frac{592}{430} e^{-430 \times 0.44 + 188} - 2\eta \geq 0.032 - 2\eta
\]}%
for $\lambda \in [10^{-3},0.0875]$. Hence, for $\eta$ sufficiently small, $S \gg 1$ in this subcase, as desired.

\subsubsection{$\lambda_1$ small ($\eta \leq \lambda_1 < 10^{-3}$)}
Here we also assume $B \geq 297$ in which case $B' \geq 296.5$. Select $\lambda^{\star} = 0.2103 \log(1/\lambda_1)$, which, by \cref{thm:ZeroRepulsion}, satisfies \eqref{eqn:Lambda_star_Assumption}. For $\lambda < 10^{-3}$, this implies $\lambda^{\star} > 1.45$. Applying both of these facts in \eqref{eqn:LPI_EX_preDH} and noting $|\chi_1(\cC)| \leq 1$, we see
{\small
\[
h_H \sL^{-1}S  \geq 1 -  e^{- 296.5 \lambda_1} - \frac{296}{134} e^{- (134.5  - 188/1.45)\lambda^{\star}}- 2\eta     \geq 1 -  e^{- 296.5 \lambda_1} - \frac{296}{134} \lambda_1 - 2\eta
\]}%
since $4.84 \times 0.2103 = 1.017\dots > 1$. As  $1-e^{-x} \geq x-x^2/2$ for $x \geq 0$, the above is
{\small
\[
 \geq  296.5 \lambda_1 - \frac{(296.5)^2}{2} \lambda_1^2 - \frac{296}{134} \lambda_1 - 2\eta  \geq  294.2 \lambda_1 \big( 1 -   150 \lambda_1\big) - 2\eta  \geq  250 \eta
\]}%
because $\eta \leq \lambda_1 <  10^{-3}$. Therefore, $S \gg 1$ completing the proof of this subcase. 

\subsubsection{$\lambda_1$ very small ($\lambda_1 < \eta$) and $\chi_1(\cC) = -1$}
Here we also assume $B \geq 163$ in which case $B' > 162.5$. From \eqref{eqn:LPI_EX_preDH}, it follows that
\begin{align*}
h_H \sL^{-1}S  \geq 1 + e^{-162.5 \lambda_1} - 325e^{- 0.5 \lambda^{\star} +188} - 2\eta  \geq 2 - O\big( e^{-0.5 \lambda^{\star} } + \eta + \lambda_1 \big)
\end{align*}
By \cref{thm:ZeroRepulsion_Full}, the choice $\lambda^{\star} = \tfrac{1}{81} \log(c_1/\lambda_1)$ satisfies \eqref{eqn:Lambda_star_Assumption} for some absolute constant $c_1 > 0$. Since $\lambda_1 < \eta$, the above is therefore $ \geq 2 - O( \eta^{0.5/81} + \eta) \geq 2 - O( \eta^{1/162})$.  As $\eta$ is fixed and sufficiently small, we conclude $S \gg 1$ as desired. This completes the proof for a ``moderate" exceptional zero. 

\subsection{Truly Exceptional Zero ($\lambda_1 < \eta$ and $\chi_1(\cC) = +1$)}
\label{subsec:TrulyExceptional}
Select $T_{\star} = T_{\star}(\eta)$ sufficiently large and let $R = \frac{1}{80.1} \log(c_1/\lambda_1)$, where $c_1 > 0$ is a sufficiently small absolute constant. By \cref{thm:ZeroRepulsion_Full}, it follows that the restricted sum over zeros $\rho$ in \eqref{eqn:LPI_EX_zeros} is empty and therefore by \eqref{eqn:LPI_EX_zeros} and \eqref{eqn:LPI_EX_zeros_error},  $h_H \sL^{-1} S \geq 1 - e^{-B' \lambda_1} - O( \lambda_1^{(B'-162)/80.1} + e^{-78 \sL} )$ as $\chi_1(\cC) = 1$. 
Additionally assuming $B \geq 243$ in which case $B' \geq 242.2$ and noting $1-e^{-x} \geq x - x^2/2$ for $x \geq 0$, we conclude that
\[
h_H \sL^{-1} S \geq 242.2 \lambda_1 -  O( \lambda_1^2 + \lambda_1^{80.2/80.1} + e^{-78 \sL} )  \geq \lambda_1 ( 242.2 -  O( \lambda_1^{0.001} +  e^{-73 \sL} ) ).
\]
since $\lambda_1 \gg e^{-4.8 \sL}$ by \cref{thm:EffectiveSiegelZero}. As $\lambda_1 \leq \eta$ for fixed $\eta >0$ sufficiently small, we conclude $S \gg \lambda_1$ as desired. 

Comparing all cases, we see that the most stringent condition is $B \geq 593$, thus completing the proof of \cref{thm:LPI-MainTheorem} in the exceptional case. \hfill \qed

\begin{remark}
	When $H \pmod{\kq}$ is primitive, the ``truly exceptional" subcase considered in \cref{subsec:TrulyExceptional} is implied by a numerically much stronger result of Zaman \cite[Theorem 1.1]{Zaman_2015b} using entirely different methods. 
%
%
%
\end{remark}

\section{Proof of \cref{thm:LPI-MainTheorem}: Non-Exceptional Case}
\label{sec:LPI_NonEX}
For this section, we assume $\lambda_1 \geq 0.0875$. Thus, we no longer have any additional information as to whether $\rho_1$ is real or not, or whether $\chi_1$ is real or not.  We proceed in a similar fashion as the exceptional case, but require a slightly more refined analysis due to the absence of Deuring-Heilbronn phenomenon. Assume $\lambda^{\star}>0$ satisfies $\lambda^{\star} < \min\{ \lambda', \lambda_2\}$, where $\lambda'$ and $\lambda_2$ are defined in \cref{subsec:LPI-notation}. For $0 < \eta \leq 10^{-3}$ fixed, suppose $B \geq \max\{ 693.5, \tfrac{\log \N\km}{\sL} + 8\eta\}$, $\ell = \lfloor \eta \sL \rfloor$, and $A = \frac{4}{\sL}$.  Thus $B, \ell$, and $A$ satisfy \eqref{eqn:LowLyingZeros_assumptions}. By \cref{lem:ReduceMainTheorem}, it suffices to show $S \gg 1$. For simplicity, denote $B' = B - 2\ell A \geq 693$.  For a non-trivial zero $\rho$ of a Hecke $L$-function, as usual, write $\rho = \beta + i\gamma = ( 1 - \frac{\lambda}{\sL} ) + i \frac{\mu}{\sL}$.  From \cref{lem:LowLyingZeros}, as $F(0) = 1$, it follows that
\[
h_H \sL^{-1}S \geq 1 - |F(\lambda_1+i\mu_1)| - |F(\lambda_1-i\mu_1)| - \sum_{\chi \pmod{H}} \sumD_{\rho} |F(\lambda+i\mu)| - \eta
\]
where the marked sum $\sum^{\dagger}$ runs over non-trivial zeros $\rho \neq \rho_1$ (or $\rho \neq \rho_1, \bar{\rho_1}$ if $\rho_1$ is complex) of $L(s,\chi)$, counted with multiplicity, satisfying $\lambda^{\star} \leq \lambda \leq R$ and $|\gamma| \leq 1$ for some $R = R(\eta) \geq 1$ sufficiently large. By \cref{lem:WeightChoice}, this implies
{\small
\begin{equation}
h_H \sL^{-1}S \geq 1 - 2e^{-B' \lambda_1} - \sum_{\chi \pmod{H}} \,\,\, \sum_{\substack{ \lambda^{\star} \leq \lambda \leq R \\ |\gamma| \leq 1} } e^{-B' \lambda} - \eta.
\label{eqn:Non-Exceptional_PrePartialSum}
\end{equation}}%
Let $\Lambda > 0$ be a fixed parameter to be specified later. To bound the remaining sum over zeros, we will apply partial summation using the quantity $\cN(\lambda)$, defined in \eqref{def:N_lambda}, over two different ranges: (i) $\lambda^{\star} \leq \lambda \leq \Lambda$ and (ii) $\Lambda < \lambda \leq R$. 

For (i), partition the interval $[\lambda^{\star},\Lambda]$ into $M$ subintervals with sample points
\[
\lambda^{\star} = \Lambda_0 < \Lambda_1 < \Lambda_2 < \dots < \Lambda_M = \Lambda.
\]
By partial summation, we see 
{\small
\begin{align*}
Z_1 := \sum_{\chi \pmod{H}} \,\,\, \sum_{\substack{\lambda^{\star} < \lambda \leq \Lambda \\ |\gamma| \leq 1} } e^{-B' \lambda} 
& = \sum_{j=1}^M \sum_{\chi \pmod{H}}\,\, \sum_{\Lambda_{j-1} < \lambda \leq \Lambda_j} e^{-B'\lambda} \\
& \leq e^{-B' \Lambda_{M-1} } \cN(\Lambda_M) + \sum_{j=1}^{M-1} \big(e^{-B'\Lambda_{j-1} } - e^{-B'\Lambda_j} \big) \cN(\Lambda_j).  
\end{align*}}%
By \cref{thm:ZeroFreeRegions}, we may choose $\lambda^{\star} = 0.2866$. Furthermore, select
{\small
\[
\Lambda = 1, \qquad M = 32, \qquad 
\Lambda_r = \begin{cases} 
	0.286+0.001r &  1 \leq r \leq 14, \\
	0.300 + 0.025(r-14) &    15 \leq r \leq 22, \\
	0.5 + 0.05(r-22) & 23 \leq r \leq 32,
 \end{cases}
\]}%
and input the estimates from \cref{table:ZeroDensity} to bound $\cN(\, \cdot \,)$, yielding $Z_1 \leq 0.9926$.

For (ii), apply partial summation along with \cref{thm:LFZD-LowLying}. Since $B' \geq 693 > 162$ and $R = R(\eta)$ is sufficiently large, it follows that
{\small
\[
Z_2 := \sum_{\chi \pmod{H}} \,\,\, \sum_{\substack{\Lambda < \lambda \leq R \\ |\gamma| \leq 1} } e^{-B' \lambda} \leq  e^{188-(B'-162) R} + B' \int_{\Lambda }^{\infty} e^{188-(B'-162) \lambda } d\lambda
\]}%
for $\sL$ sufficiently large depending on $\eta$. Evaluating the RHS with $B' \geq 693$ and $\Lambda = 1$, we deduce $Z_2 \leq 10^{-400}$.  Incorporating (i) and (ii) into \eqref{eqn:Non-Exceptional_PrePartialSum}, we conclude
\begin{align*}
h_H\sL^{-1} S  \geq 1 - 2e^{-B' \lambda_1} - 0.9926 - 10^{-400} - 2\eta   \geq 0.0073- 2\eta
\end{align*}
as $\lambda_1 > 0.0875$ and $B' \geq 693$. 
Since $\eta \in (0,10^{-3}]$ is fixed and sufficiently small, we conclude $S \gg 1$ as desired. This completes the proof of \cref{thm:LPI-MainTheorem}. \hfill \qed

\section{Proofs of Theorems \ref{thm:least_prime_QF}$-$\ref{thm:MF_Fourier}}
\label{sec:proof_MF_Fourier}

First, we prove \cref{thm:least_prime_QF}.

\begin{proof}[Proof of \cref{thm:least_prime_QF}]
Let $Q(x,y)\in\mathbb{Z}[x,y]$ be a positive-definite primitive binary quadratic form of discriminant $D$.  Let $K=\mathbb{Q}(\sqrt{D})$, and let $L$ be the ring class field of the order of discriminant $D$ in $K$.  By Theorem 9.12 of Cox \cite{Cox}, the rational primes $p\nmid D$ represented by $Q$ are the primes which split in $K$ that satisfy a certain Chebotarev condition in $L$.  We have that $D_K\mathcal{Q}\leq |D|$.  The result now follows.
\end{proof}

\subsection{Class functions}
In order to proceed, we now review some facts about class functions (cf. \cite{Ser1}).  Let $L/F$ be a Galois extension of number fields with Galois group $G$, and let $\phi:G\to\mathbb{C}$ be a class function.  For each prime ideal $\mathfrak{p}$ of $F$, choose any prime ideal $\mathfrak{P}$ of $L$ dividing $\mathfrak{p}$.  Let $D_{\kP}$ and $I_{\kP}$ be the decomposition and inertia subgroups of $G$ at $\mathfrak{p}$, respectively.  We then have a distinguished Frobenius element $\sigma_{\kP}\in D_{\kP}/I_{\kP}$.  For each integer $m\geq1$, let
{\small
\[
\phi(\mathrm{Frob}_{\mathfrak{p}}^m)=\frac{1}{|I_{\kP}|}\sum_{\substack{g\in D_{\kP} \\ gI_{\kP}=\sigma_{\kP}^m\in D_{\kP}/I_{\kP}}}\phi(g).
\]}%
Note that $\phi(\mathrm{Frob}_{\mathfrak{p}}^m)$ is independent of the aforementioned choice of $\kP$.  If $\mathfrak{p}$ is unramified in $L$, this definition agrees with the value of $\phi$ on the conjugacy class $\mathrm{Frob}_{\mathfrak{p}}^m$ of $G$.  For $x\geq2$, we define
{\small
\begin{equation}
\label{eqn:def_phi_tilde}
\pi_{\phi}(x)=\sum_{\substack{\textup{$\mathfrak{p}$ unramified in $L$} \\ \mathrm{N}_{F/\mathbb{Q}}~\mathfrak{p}\leq x}}\phi(\mathrm{Frob}_{\mathfrak{p}}),\qquad \widetilde{\pi}_{\phi}(x)=\sum_{\substack{\textup{$\mathfrak{p}$ unramified in $L$} \\ \mathrm{N}_{F/\mathbb{Q}}~\mathfrak{p}^m\leq x}}\frac{1}{m}\phi(\mathrm{Frob}_{\mathfrak{p}}^m).
\end{equation}}%

Let $C\subset G$ be stable under conjugation, and let $\mathbf{1}_C:G\to\{0,1\}$ be the class function given by the indicator function of $C$.  Now, define $\pi_C(x,L/F)=\pi_{\mathbf{1}_C}(x)$ and $\widetilde{\pi}_C(x,L/F)=\widetilde{\pi}_{\mathbf{1}_C}(x)$.  Serre \cite[Proposition 7]{Ser1} proved that if $x\geq 2$, then
{\small
\begin{equation}
\label{eqn:lower_bound_pi}
|\pi_C(x,L/F)-\tilde{\pi}_C(x,L/F)|\leq\frac{2n_F}{\log 2}\Big(\frac{\log D_L}{n_L}+\sqrt{x}\Big).
\end{equation}}%
By arguments similar to the proof of \cref{thm:main_theorem}, we have that if $A$ is an abelian subgroup of $G$ such that $A\cap C$ is nonempty, then $\widetilde{\pi}_C(x,L/F)=\widetilde{\pi}_{\mathrm{Ind}_A^G C}(x,L/L^A)$.

We now state a slightly weaker version of \eqref{eqn:lower_bound_pi_C} and \cref{thm:main_theorem} which will be convenient for the remaining proofs.  For positive integers $n$, let $\omega(n)=\#\{p:p\mid n\}$ and $\mathrm{rad}(n)=\prod_{p\mid n}p$.  

\begin{thm}
\label{thm:main_theorem_v2}
Let $L/F$ be a Galois extension of number fields with Galois group $G$ and $L\neq\Q$, and let $C$ be any conjugacy class of $G$.  Let $H$ be an abelian subgroup of $G$ such that $H\cap C$ is nonempty, let $K$ be the subfield of $L$ fixed by $H$, and assume that $K\neq\Q$.  If
\[
M(L/K):=([L:K]^{2}n_K^{1+\omega(D_L)}\mathrm{rad}(D_L)^{3})^{n_K}
\]
is sufficiently large
and
\[
x\gg \frac{n_L^3}{\mathrm{rad}(D_L)^{694}}\Big([L:K]^{1042}n_K^{694 \omega(D_L)}\mathrm{rad}(D_L)^{1736}\Big)^{n_K},
\]
then
\[
\pi_C(x,L/F)\gg \frac{M(L/K)^{-5}}{n_L[L:K]}\frac{x}{\log x}.
\]
Consequently, for all $L/F$, we have that
\[
P(C,L/F)\ll \frac{n_L^3}{\mathrm{rad}(D_L)^{694}}\Big([L:K]^{1042}n_K^{694 \omega(D_L)}\mathrm{rad}(D_L)^{1736}\Big)^{n_K}.
\]
\end{thm}
\begin{proof}
Let $\mathcal{P}(L/K)$ be the set of rational primes $p$ such that there is a prime ideal $\kp$ of $K$ such that $\kp\mid p$ and $\kp$ ramifies in $L$.  By \cite[Proposition 6]{Ser1}, we have
\[
D_K\leq n_K^{n_K\omega(D_K)}\mathrm{rad}(D_K)^{n_K-1}.
\]
(We can replace $\omega(D_K)$ with 1 if $K/\Q$ is Galois.)  Since $L/K$ is abelian, we have by \cite[Proposition 2.5]{MMS} that
\[
\mathcal{Q}\leq \Big([L:K]\prod_{p\in \mathcal{P}(L/K)}p\Big)^{2n_K}.
\]
Since $K\neq\Q$, we have $\omega(D_K)\geq1$ and $n_K\geq2$.  Moreover, the primes in $\mathcal{P}(L/K)$ and the primes dividing $D_K$ all divide $D_L$.  Using these facts, we obtain the claimed results using \eqref{eqn:def_phi_tilde}, \eqref{eqn:lower_bound_pi}, \cref{thm:main_theorem}, and \cref{thm:LPI-MainTheorem}.
\end{proof}

\begin{remark}
For comparison, if one uses \cite[Proposition 6]{Ser1} to bound $D_L$, then \eqref{eqn:LMO} implies that $P(C,L/F)\ll (n_L^{\omega(D_L)}~\mathrm{rad}(D_L))^{40n_L}$.  We can replace $\omega(D_L)$ with 1 if $L/\Q$ is Galois.
\end{remark}

\subsection{$\mathrm{GL}_2$ extensions}

We will now review some facts about $\mathrm{GL}_2$ extensions of $\mathbb{Q}$ and class functions to prove Theorems \ref{thm:EC_order}-\ref{thm:MF_Fourier}.  Let $f(z)=\sum_{n=1}^{\infty}a_f(n)e^{2\pi i n z}\in\Z[[e^{2\pi i z}]]$ be a non-CM newform of even weight $k\geq2$ and level $N\geq1$.  By Deligne, there exists a representation
\[
\rho_{f,\ell}: \mathrm{Gal}(\bar{\mathbb{Q}}/\mathbb{Q})\to \mathrm{GL}_2(\mathbb{F}_{\ell})
\]
with the property that if $p\nmid \ell N$ and $\sigma_p$ is a Frobenius element at $p$ in $\mathrm{Gal}(\bar{\mathbb{Q}}/\mathbb{Q})$, then $\rho_{f,\ell}$ is unramified at $p$ and
\[
\mathrm{tr}~\rho_{f,\ell}(\sigma_p)\equiv a_f(p)\pmod{\ell},\qquad \det\rho_{f,\ell}(\sigma_p)\equiv p^{k-1}\pmod{\ell}.
\]
If $\ell$ is sufficiently large, the representation is surjective.  Let $L=L_{\ell}$ be the subfield of $\bar{\mathbb{Q}}$ fixed by the kernel of $\rho_{f,\ell}$.  Then $L/\mathbb{Q}$ is a Galois extension unramified outside $\ell N$ whose Galois group is $\ker\rho_{f,\ell}$, which is isomorphic to a subgroup of
\[
G=G_{\ell}=\{\textup{$A\in \mathrm{GL}_2(\mathbb{F}_{\ell})$: $\det A$ is a $(k-1)$-th power in $\mathbb{F}_{\ell}^\times$}\}.
\]
If $\ell$ is sufficiently large, then the representation is surjective, in which case
\begin{equation}
\label{eqn:GL}
\ker\rho_{f,\ell}\cong G.
\end{equation}
When $k=2$, $f$ is necessarily the newform of a non-CM elliptic curve $E/\mathbb{Q}$; the conductor $N_E$ of $E$ is also the level of $f$.  In this situation, we write $\rho_{f,\ell}=\rho_{E,\ell}$, and $L$ is the $\ell$-division field $\mathbb{Q}(E[\ell])$.  It is conjectured that $\ker\tilde{\rho}_{E,\ell}\cong \mathrm{GL}_2(\mathbb{F}_\ell)$ for all $\ell>37$.  When $E/\mathbb{Q}$ is non-CM and has squarefree level, it follows from the work of Mazur \cite{Mazur_Torsion} that $\ker\tilde{\rho}_{E,\ell}\cong \mathrm{GL}_2(\mathbb{F}_\ell)$ for all $\ell\geq11$.

\begin{lem}
\label{lem:EC_order}
Let $L/\mathbb{Q}$ be a $\mathrm{GL}_2(\mathbb{F}_\ell)$ extension which is unramified outside of $\ell N$ for some $N\in\Z$.  Let $C\subset \mathrm{GL}_2(\bF_{\ell})$ be a conjugacy class which intersects the abelian subgroup $D\subset G$ of diagonal matrices.  There exists a prime $p\nmid \ell N$ such that $[L/\mathbb{Q},p]=C$ and $p\ll \ell^{(5209+1389\omega(N))\ell^2}\mathrm{rad}(N)^{1737\ell(\ell+1)}$.
\end{lem}
\begin{proof}
If $K=L^D$ is the subfield of $L$ fixed by $D$, then $[L:K]=(\ell-1)^2$ and $[K:\mathbb{Q}]=\ell(\ell+1)$.  Moreover, $\mathrm{rad}(D_L)\mid \ell~\mathrm{rad}(N)$.  The result now follows immediately from \cref{thm:main_theorem_v2}.
%
%
\end{proof}

\begin{proof}[Proof of \cref{thm:EC_order}]
It follows from the proof of \cite[Theorem 4]{VKM2} and Mazur's torsion theorem \cite{Mazur_Torsion} that it suffices to consider $\ell\geq11$.  Let $L=\mathbb{Q}(E[\ell])$ be the $\ell$-division field of $E/\mathbb{Q}$.  For $p\nmid \ell N_E$, we have that $E(\bF_p)$ has an element of order $\ell$ if and only if
\begin{equation}
\label{eqn:order_frob}
\mathrm{tr}~\rho_{\ell,E}(\sigma_p)\equiv \det\rho_{\ell,E}(\sigma_p)+1\pmod\ell,
\end{equation}
where $\sigma_p$ is Frobenius automorphism at $p$ in $\mathrm{Gal}(\overline{\mathbb{Q}}/\mathbb{Q})$.  Suppose that $\mathrm{Gal}(L/\mathbb{Q})\cong \mathrm{GL}_2(\mathbb{F}_\ell)$.  The Frobenius automorphisms in $\mathrm{GL}_2(\bF_\ell)$ which satisfy \eqref{eqn:order_frob} form a union of conjugacy classes in $\mathrm{GL}_2(\bF_\ell)$ which includes the identity element.  The subgroup $D$ of diagonal matrices is a maximal abelian subgroup of $\mathrm{GL}_2(\bF_\ell)$.  Thus $\pi_{\{\mathrm{id}\}}(x,L/\mathbb{Q})$ is a lower bound for the function that counts the primes $p\leq x$ such that $p\nmid \ell N_E$ and $E(\bF_p)$ has an element of order $\ell$.  Upon calculating $[L:K]=(\ell-1)^2$, $[K:\Q]=\ell(\ell+1)$, and $\mathrm{rad}(D_L)\mid\ell~\mathrm{rad}(N)$, we apply \cref{lem:EC_order} to $\pi_{\{\mathrm{id}\}}(x,L/\mathbb{Q})$ yields the claimed result.

Suppose now that $\mathrm{Gal}(L/\mathbb{Q})$ is not isomorphic to $\mathrm{GL}_2(\mathbb{F}_\ell)$.  The possible cases are described in the proof of \cite[Theorem 4]{VKM2}.  Applying similar analysis to all of these cases, one sees that the above case gives the largest upper bound for the least prime $p$ such that $\ell\mid\#E(\mathbb{F}_p)$.
\end{proof}

\begin{proof}[Proof of \cref{thm:MF_Fourier}]

We assume $\ell$ is a prime for which \eqref{eqn:GL} is satisfied, and we assume that $\gcd(k-1,\ell-1)=1$ so that $G\cong\mathrm{GL}_2(\mathbb{F}_{\ell})$.  To prove the theorem, we consider the prime counting function
\[
\pi_f(x;\ell,a)=\#\{\textup{$p\leq x$: $p\nmid\ell N$, $a_f(p)\equiv a\pmod\ell$, $\ell$ splits in $\mathbb{Q}((a_f(p)^2-4p^{k-1})^{1/2})$}\}.
\]
Note that for $p\nmid \ell N$, $a_f(p)^2-4p^{k-1}=\mathrm{tr}(\rho_{f,\ell}(\sigma_p))-4\det(\rho_{f,\ell}(\sigma_p))^2$, where $\sigma_p$ is Frobenius at $p$ in $\mathrm{Gal}(\overline{\mathbb{Q}}/\mathbb{Q})$.  The subset $C\subset G$ given by
\[
C=\{\textup{$A\in G$: $\mathrm{tr}(A)\equiv a\pmod{\ell}$, $\mathrm{tr}(A)^2-4\det(A)$ is a square in $\mathbb{F}_{\ell}^\times$}\}
\]
is a conjugacy-invariant subset of $G$.  Thus we study the function $\widetilde{\pi}_{C}(x,L/\mathbb{Q})$.  Let $B\subset G$ denote the subgroup of upper triangular matrices; the condition that $\mathrm{tr}(A)^2-4\det(A)$ is a square in $\mathbb{F}_{\ell}^\times$ means that $\sigma_p$ is conjugate to an element in $B$.  If $\Gamma$ is a maximal set of elements $\gamma\in B$ which are non-conjugate in $G$ with $\mathrm{tr}(\gamma)\equiv a\pmod{q}$, we have that $C=\bigsqcup_{\gamma\in \Gamma}C_G(\gamma)$, where $C_G(\gamma)$ denotes the conjugacy class of $\gamma$ in $G$.  Since $B$ is a subgroup of $G$ with the property that every element of $C$ is conjugate to an element of $B$, it follows from the work in \cite{MMS} that
{\small
\begin{align*}
\widetilde{\pi}_{C}(x,L/\mathbb{Q})&=\sum_{\gamma\in\Gamma}\frac{\widetilde{\pi}_{C_B(\gamma)}(x,L/L^B)}{[\mathrm{Cent}_G(\gamma):\mathrm{Cent}_B(\gamma)]},
\end{align*}}%
where $\mathrm{Cent}_G(\gamma)$ is the centralizer of $\gamma$ in $G$ (and similarly for $B$).  If $C_{1}=\bigsqcup_{ \textup{$\gamma\in\Gamma$ non-scalar}}C_B(\gamma)$, then it follows that $\widetilde{\pi}_C(x;L/\mathbb{Q})\geq\frac{1}{|G|}\widetilde{\pi}_{C_{1}}(x,L/L^B)$ for all $x\geq2$.

\subsubsection*{Case 1: $a\not\equiv0\pmod{\ell}$}

Let $U$ be the normal subgroup of $B$ consisting of the matrices whose diagonal entries are both 1.  We observe that $U\cdot C_{1}\subset C_{1}$; therefore, using the work in \cite{MMS}, we have that $\frac{1}{|G|}\widetilde{\pi}_{C_{1}}(x,L^U/L^B)\geq\frac{1}{n_L}\tilde{\pi}_{C_2}(x,L^U/L^B)$ for $x\geq2$, where $C_2$ is the image of $C_{1}\cap B$ in $B/U$.  Since $L^U/L^B$ is an abelian extension and all of the ramified primes divide $\ell N$, the theorem now follows from \eqref{eqn:def_phi_tilde}, \eqref{eqn:lower_bound_pi}, and \cref{thm:main_theorem_v2}.  It is straightforward to compute $n_{L^B}=\ell+1$ and $[L^U:L^B]=(\ell-1)^2$.

\subsubsection*{Case 2: $a\equiv0\pmod{\ell}$}

Let $H$ be the normal subgroup of $B$ consisting of matrices whose eigenvalues are both equal.  We have that $H\cdot C_{1}\subset C_{1}$ since multiplying a trace zero matrix by a scalar does not change the trace.  Let $C_3$ be the image of $C_{1} \cap B$ in $B/H$.  The arguments are now the same as in the previous case, with $L^H$ replacing $L^U$.  In fact, since $B/H\cong \mathbb{F}_{\ell}^\times$ is abelian of order $\ell-1$ and $C_3$ is a singleton, we obtain a slightly better numerical constant in the exponent than what is stated in \cref{thm:MF_Fourier} when $a\equiv0\pmod \ell$.
\end{proof}



\bibliographystyle{abbrv}
\bibliography{LPI}
\end{document}